\documentclass[11pt]{amsart}  
\usepackage{amsbsy,amsmath,amsthm,amssymb,color,verbatim}
\usepackage[backrefs]{amsrefs}
\usepackage{float}
\usepackage{fullpage,cancel,hyperref}

\usepackage{float}
\usepackage{wrapfig}

\usepackage[dvipsnames]{xcolor}
\usepackage{graphicx}
\usepackage{subcaption}
\usepackage{xcolor}
\usepackage{multicol}
\usepackage{supertabular}
\usepackage{tikz}
\usepackage{tkz-graph}
\tikzstyle{vertex}=[circle, draw, inner sep=0pt, minimum size=4pt]
\newcommand{\vertex}{\node[vertex]}
\tikzstyle{vtx}=[circle, draw, inner sep=0pt, minimum size=8pt]
\newcommand{\vtx}{\node[vtx]}
\definecolor{darkgreen}{cmyk}{.9,0,.9,.2}
\definecolor{midgray}{gray}{0.60}
\definecolor{lightgray}{gray}{0.90}
\definecolor{lmgray}{gray}{0.70}
\usepackage{colortbl}
\def\a{\alpha}
\def\l{\lambda}

\def\ZZ{{\mathbb Z}}

\allowdisplaybreaks
 \def\fq{\left(\displaystyle\frac{1}{q}\right)}

\def\A{\mathcal{A}}

\def\a{\alpha}
\def\w{\varpi}

\usepackage{multirow}
\usepackage{longtable}

\usepackage{makecell}
\setcellgapes{4pt}

\usepackage{relsize}
\makeatletter
\newtheorem*{rep@theorem}{\rep@title}
\newcommand{\newreptheorem}[2]{%
\newenvironment{rep#1}[1]{%
 \def\rep@title{#2 \ref{##1}}%
 \begin{rep@theorem}}%
 {\end{rep@theorem}}}
\makeatother

\makeatletter
\newtheorem*{rep@conjecture}{\rep@title}
\newcommand{\newrepconjecture}[2]{%
\newenvironment{rep#1}[1]{%
 \def\rep@title{#2 \ref{##1}}%
 \begin{rep@conjecture}}%
 {\end{rep@conjecture}}}
\makeatother

\makeatletter
\newcommand{\addresseshere}{%
  \enddoc@text\let\enddoc@text\relax
}
\makeatother

\theoremstyle{definition}

\usepackage{multicol}

\newtheorem{theorem}{Theorem}[section]
\newreptheorem{theorem}{Theorem}
\newtheorem{proposition}{Proposition}[section] 
\newtheorem{conjecture}{Conjecture}[section]
\newtheorem{lemma}{Lemma}[section] 
\newtheorem{corollary}{Corollary}[section] 
\newrepconjecture{conjecture}{Conjecture}

\title{When is the $q$-multiplicity of a weight a power of $q$?}

\author{Pamela E. Harris}
\address{Department of Mathematics and Statistics, Williams College, United States}
\email{peh2@williams.edu}

\author{Margaret Rahmoeller}
\address{Department of Mathematics, Computer Science, and Physics, Roanoke College}
\email{rahmoeller@roanoke.edu}
\thanks{}

\author{Lisa Schneider}
\address{Department of Mathematics and Computer Science, Salisbury University}
\email{lmschneider@salisbury.edu}
\thanks{}

\author{Anthony Simpson}
\address{Department of Mathematics and Statistics, Williams College, United States}
\email{als7@williams.edu}

\date{\today}

\begin{document}

\maketitle

\begin{abstract}
Berenshtein and Zelevinskii provided an exhaustive list of pairs of weights $(\lambda,\mu)$ of simple Lie algebras $\mathfrak{g}$ (up to Dynkin diagram isomorphism) for which the multiplicity of the weight $\mu$ in the representation of $\mathfrak{g}$ with highest weight $\lambda$ is equal to one. 
Using Kostant's weight multiplicity formula we describe and enumerate the contributing terms to the multiplicity for subsets of these pairs of weights and show that, in these cases, the cardinality of these contributing sets is enumerated by (multiples of) Fibonacci numbers. We conclude by using these results to compute the associated $q$-multiplicity for the pairs of weights considered, and conjecture that in all cases the $q$-multiplicity of such pairs of weights is given by a power of $q$.
\end{abstract}


\section{Introduction}
Throughout this manuscript we let $\mathfrak g$ be a simple complex Lie algebra and $\mathfrak{h}$ a Cartan subalgebra with $\dim\mathfrak{h}=r$. Following the notation of \cite{GW}, we let $\Phi$ denote the set of roots corresponding to $(\mathfrak {g,h})$, $\Phi^+\subseteq\Phi$ be a set of positive roots, and $\Delta\subseteq\Phi^+$ be the set of simple roots. The set of integral and dominant integral weights are denoted by  $P(\mathfrak g)$ and $P_+(\mathfrak g)$ respectively. The Weyl group is denoted by $W$ and recall that $W$ is generated by the reflections $s_1,\ldots, s_r$, where $s_i$ is the root reflection corresponding to the simple root $\alpha_i\in\Delta$. For any $w\in W$ we let $\ell(w)$ denote the length of $w$, which represents the minimum nonnegative integer $k$ such that $\sigma$ is a product of $k$ reflections. 

The theorem of the highest weight states that every irreducible (complex) representation of $\mathfrak{g}$ is a highest weight representation $L(\lambda)$ with highest weight $\lambda$. If $\mu$ is a  weight of $L(\lambda)$ then one can compute the multiplicity of this weight using Kostant's weight multiplicity formula~\cite{KMF}:
    \begin{align}
        m(\lambda,\mu)=\displaystyle\sum_{\sigma\in W}^{}(-1)^{\ell(\sigma)}\wp(\sigma(\lambda+\rho)-(\mu+\rho))\label{mult formula},
    \end{align}
where $\wp$ denotes Kostant's partition function  $\wp : \mathfrak h^* \rightarrow \ZZ$, which  is the nonnegative integer-valued function such that for each $ \xi \in \mathfrak h^*$, $\wp(\xi)$ counts the number of ways $\xi$ may be written as a nonnegative linear combination of positive roots.

Although Equation \eqref{mult formula} provides a way to compute weight multiplicities, its practical use is limited. For example, since the order of the Weyl group (which indexes the sum) increases factorially in terms of the rank of the Lie algebra, the number of terms in the sum presents a complication in its practical use. The second complication arises due to the lack of general closed formulas for the value of the partition function involved. One thing that has been noted in past research, is that many of the terms contribute trivially to the sum, i.e. the value of the partition function is zero, and hence these terms do not contribute to the overall multiplicity. This latter point has motivated work in determining the \emph{Weyl alternation set} (corresponding to the integral weights $\lambda$ and $\mu$), defined by
    \begin{align}
        \mathcal A(\lambda,\mu)=\{\sigma\in W:\wp(\sigma(\lambda+\rho)-(\mu+\rho))>0\}.\label{KWMF}
    \end{align}
Note that a Weyl group element $\sigma$ is in $\mathcal{A}(\lambda,\mu)$ if and only if the expression  $\sigma(\lambda+\rho)-(\mu+\rho)$ can be written as a nonnegative $\ZZ$-linear combination of positive roots.

Determining Weyl alternation sets provides a way to describe the complexity in computing weight multiplicities. Since its definition in 2011, there are only a few cases where a concrete description of the elements of these sets exists. This includes the case where $\lambda$ denotes the highest root of the classical Lie algebras and, hence, $L(\lambda)$ corresponds to the adjoint representation. In this case, Harris, Insko, and Williams established that the cardinalities of the Weyl alternation sets $\mathcal {A}(\lambda,0)$
are given by linear recurrences with constant coefficients \cites{PHThesis,Harris,HIW}. In 2017, Chang, Harris, and Insko described the sets $\A(\beta,\mu)$, where $\beta$ is the sum of all of the simple roots in a classical Lie algebra and $\mu$ is an integral weight, and showed that the cardinality of these sets are given by the Fibonacci numbers or multiples of the Lucas numbers \cite{CHI}. 

In this work we extend this body of knowledge by considering pairs of weights $(\lambda,\mu)$ for which it is known that $m(\lambda,\mu)=1$, but whose Weyl alternation sets are unknown. 
For the simple Lie algebras, a complete list of pairs of weight $(\lambda,\mu)$ satisfying $m(\lambda,\mu)=1$ (up to isomorphism of Dynkin diagram) were provided in the work of Berenshtein  and Zelevinskii \cite[Theorem 1.3]{BZ} and is given as follows.
\begin{enumerate}\setlength{\itemindent}{.25in}
    \item[(List A)] Type $A_r$ $(r\geq 1)$: $\lambda=\ell\varpi_1$, $\mu=\sum_{1\leq i\leq r}m_i\varpi_i$, where $m_i\in\mathbb{Z}_+$ and \\ $(\ell-\sum_{1\leq i\leq r} im_i)\in (r+1)\mathbb{N}$.
    \item[(List B)] Type $B_r$ $(r\geq 2)$: $\lambda=\ell\varpi_1$, $\mu=\sum_{1\leq i\leq r}m_i\varpi_i$, where $m_i\in\mathbb{Z}_+$ is even  and \\ $(\ell-1)=(\sum_{1\leq i< r} im_i)+rm_r/2$.
    \item[(List G)] Type $G_2$:
    \begin{itemize}
        \item $\lambda=\ell\varpi_2$, $\mu=m_1\varpi_1+m_2\varpi_2$, where $m_1,m_2\in\mathbb{Z}_+$ and $3\ell-1=2m_1+3m_2$
    \item\label{thm:part3} $\lambda=\varpi_1$, $\mu=0$
    \end{itemize}
    where in all cases $\varpi_1$ and $\varpi_2$ denote fundamental weights of the respective Lie algebra.
\end{enumerate}

In Sections \ref{sec:A} and \ref{sec:B} we provide a description of the elements of Weyl alternation sets $\A(\lambda,\mu)$ in the Lie algebra of type $A_r$ and $B_r$, respectively, for certain pairs of weights $(\lambda,\mu)$ by establishing the following main results.
\begin{theorem}\label{thm:mainA}
In type $A_r$ with $r\geq 1$ and $\ell\in\mathbb{Z}_+$:
\begin{enumerate}
    \item If $\lambda = \ell\w_1$, $\mu=\sum_{1\leq i\leq r}m_i\w_i$, where $\ell, m_i\in\mathbb{Z}_+$ and $(\ell-\sum im_i)=0 \mod (r+1)$, then $s_1\notin \A(\lambda,\mu)$ for any choice of $\ell$ and $r$.
    \item If $\lambda=\ell\varpi_1$, $\mu=\sum_{1\leq i\leq r}m_i\varpi_i$, where $m_i\in\mathbb{Z}_+$ and $(\ell-\sum_{1\leq i\leq r} im_i)=0$, $m_j=0$ for all $3\leq j\leq r$, then 
$\A(\lambda,\mu)=\{1\}$.
\item If $\lambda=\ell\varpi_1$, $\mu=\sum_{1\leq i\leq r}m_i\varpi_i$, where $m_i\in\mathbb{Z}_+$ and $(\ell-\sum_{1\leq i\leq r} im_i)= 0$ and $m_j=0$ for all $4\leq j\leq r$ and $m_3\neq 0$, then 
$\A(\lambda,\mu)=\{1,s_2\}$.
\item If $\lambda=\ell\varpi_1$, $\mu=\sum_{1\leq i\leq r}m_i\varpi_i$, where $m_i\in\mathbb{Z}_+$ and $(\ell-\sum_{1\leq i\leq r} im_i)= 0$, $m_j=0$ for all $5\leq j\leq r$ and $m_4\neq 0$, then 
\[
\A(\lambda,\mu)=\begin{cases}
\{1,s_2,s_3,s_2s_3\}, & \text{if } m_4=1\\
\{1,s_2,s_3,s_2s_3,s_3s_2,s_2s_3s_2\}, & \text{if } m_4\ge 2.
\end{cases}
\]
\end{enumerate}
\end{theorem}

\begin{theorem}\label{thm:mainB}
In type $B_r$  with $r\geq 3$:
\begin{enumerate}
    \item Let  $\ell$ be a positive odd integer. Then   $\sigma\in\A(\ell\w_1,(\ell-1-4k)\w_1+2k\w_2)$ where $k\in \ZZ_{\geq 0}$ if and only if $\sigma$ is a product of nonconsecutive $s_i$'s with $1<i\leq r$.
    \item If $\ell, r \in 2\ZZ_{\ge 0}$ under the condition $\ell-1=\sum_{1\le i< r} im_i + \frac{rm_r}{2}$ with $m_i\in 2\ZZ_{\ge 0}$, then $\A(\ell\w_1,\mu)=\emptyset$.
    \item  All $\sigma\in\A(\ell\w_1,(\ell-1-4j-6k)\w_1+2j\w_2+2k\w_3)$ where $j,k\in \ZZ_{\geq 0}$ and $k\neq 0$ are products of nonconsecutive $s_i$'s with $1<i\leq r$ or products of the form $\{s_{i_1}\dots s_{i_n}s_2s_3\}$ with nonconsecutive $s_{i_j}$'s for $4<i_j\leq r$.
\end{enumerate} 
\end{theorem}
Note that specializing Theorem \ref{thm:mainB} to $\ell=1$ and $\mu=0$ recovers the result presented in \cite[Theorem 3.1]{CHI}, which shows that the cardinality of the associated Weyl alternation set is given by a Fibonacci number. Extending our results to the remaining pairs of weights listed in List A and List B is much more difficult since the Weyl group action on the highest weight of these representations is not as straightforward to describe. 
The complications that arise in describing these Weyl alternation sets are twofold. First, for a given $\lambda$ there is a finite set of possible weights $\mu$, which are linear combinations of fundamental weights satisfying the needed conditions. Second, whenever $\mu$ includes a fundamental weight with high index, one must consider many more conditions in order to determine whether certain Weyl group elements appear in the corresponding Weyl alternation set. 

In the Lie algebra of type $G_2$, we provide a complete description of the Weyl alternation sets $\A(\lambda,\mu)$  for all pairs of weights $(\lambda,\mu)$ given in \eqref{thm:part3}. In Section \ref{sec:G2}, we present the proof of the following.
\begin{theorem}\label{thm:mainG2}
Let $\lambda=\ell\w_2$ with $\ell\geq 1$ be a weight of the Lie algebra of type $G_2$. If $\mu=m_1\varpi_1+m_2\varpi_2$, where $m_1,m_2\in\mathbb{Z}_{\geq 0}$ and $3\ell-1=2m_1+3m_2$, then 
$\A(\lambda,\mu)=\{1,s_1\}$ when $\lambda\neq 0$ and $\A(0,\mu)=\emptyset$.    \end{theorem}

Our last set of results is concerned with the $q$-multiplicity for pairs of weights $(\lambda,\mu)$ in the lists provided above. To make this precise, recall that for a weight  $\xi$, the value of the $q$-analog of Kostant's partition function is the polynomial $\wp_q(\xi)=c_0+c_1q+\cdots+c_k q^k$, where $c_j$ denotes number of ways to write $\xi$ as a nonnegative integral sum of exactly $j$ positive roots. Then the $q$-analog of Kostant's weight multiplicity formula is defined, in \cite{LL}, as:
\begin{center}
$m_q(\lambda,\mu)=\sum\limits_{\sigma\in W}^{}(-1)^{\ell(\sigma)}\wp_q(\sigma(\lambda+\rho)-(\mu+\rho))$.
\end{center}
We provide the $q$-multiplicity for the pairs of weights considered in Theorems \ref{thm:mainA}, \ref{thm:mainB}, and \ref{thm:mainG2}. We state these results below and provide their proof in Section~\ref{sec:qanalog}.
\begin{theorem}\label{thm:qmainA}
In type $A_r$ with $r\geq 1$ and $\ell\in\mathbb{Z}_+$:
\begin{enumerate}
    \item If $\lambda=\ell\varpi_1$, $\mu=\sum_{1\leq i\leq r}m_i\varpi_i$, where $m_i\in\mathbb{Z}_+$ and $(\ell-\sum_{1\leq i\leq r} im_i)=0$, $m_j=0$ for all $3\leq j\leq r$, then 
$m_q(\lambda,\mu)=q^{m_2}$.
\item If $\lambda=\ell\varpi_1$, $\mu=\sum_{1\leq i\leq r}m_i\varpi_i$, where $m_i\in\mathbb{Z}_+$ and $(\ell-\sum_{1\leq i\leq r} im_i)= 0$ and $m_j=0$ for all $4\leq j\leq r$ and $m_3\neq 0$, then 
$m_q(\lambda,\mu)=q^{m_2+3m_3}$.
\item If $\lambda=\ell\varpi_1$, $\mu=\sum_{1\leq i\leq r}m_i\varpi_i$, where $m_i\in\mathbb{Z}_+$ and $(\ell-\sum_{1\leq i\leq r} im_i)= 0$, $m_j=0$ for all $5\leq j\leq r$ and $m_4\neq 0$, then 
$m_q(\lambda,\mu)=q^{m_2+3m_3+6}$ when $m_4=1$ and $m_q(\lambda,\mu)=q^{m_2+3m_3+6m_4}$ when $m_4\ge 2$.
\end{enumerate}
\end{theorem}

\begin{theorem}\label{thm:qmainB}
In type $B_r$  with $r\geq 3$:
\begin{enumerate}
    \item If  $\ell$ is a positive odd integer and $k\in \ZZ_{\geq 0}$, then   $m_q(\ell\w_1,(\ell-1-4k)\w_1+2k\w_2)=q^{2k+r}$.
    \item If $\ell, r \in 2\ZZ_{\ge 0}$ under the condition $\ell-1=\sum_{1\le i< r} im_i + \frac{rm_r}{2}$ with $m_i\in 2\ZZ_{\ge 0}$, then $m_q(\ell\w_1,\mu)=0$.
    \item  If $\ell$ is a positive odd integer and $j,k\in \ZZ_{\geq 0}$ with $k\neq 0$, then $m_q(\ell\w_1,(\ell-1-4j-6k)\w_1+2j\w_2+2k\w_3)=q^{r+2j+6k-2}$.
\end{enumerate} 
\end{theorem}

\begin{theorem}\label{thm:qmainG2}
Let $\lambda=\ell\w_2$ with $\ell\geq 1$ be a weight of $G_2$. If $\mu=m_1\varpi_1+m_2\varpi_2$, where $m_1,m_2\in\mathbb{Z}_{\geq 0}$ and $3\ell-1=2m_1+3m_2$, then 
$m_q(\lambda,\mu)=q^{n+2}$ where $m_1=3n+1$.    
\end{theorem}

Section \ref{sec:future} ends the manuscript by providing ample evidence for the following.
\begin{conjecture}\label{con:powerofq}
If $\lambda,\mu$ are a pair of weights for which $m(\lambda,\mu)=1$, then $m_q(\lambda,\mu)=q^{f(r)}$, where $f(r)$ is a function of $r$, the rank of the Lie algebra.
\end{conjecture}
This conjecture is rather curious, as there are infinite ways that the $q$-multiplicity when evaluated at $q=1$ would yield a value of one. Yet what we see is that the only term remaining, after a tremendous amount of (polynomial) cancellation, is simply a power of $q$. Unfortunately, the techniques we present do not lend themselves well to provide a direct proof of the above conjecture, and we would welcome a new approach to such a problem. 

Lastly, in Section \ref{sec:future}, we present a new direction for research related to a construction of a poset we call the \emph{Weyl alternation poset}. This poset is created by computing all Weyl alternation sets $\A(\lambda,\mu)$ of a specified simple Lie algebra, and considering the set containment of these Weyl alternation sets. We present the cases of the Lie algebras of types $A_2$, $B_2$, and $G_2$, in which we fix $\mu=0$ and vary $\lambda$ in the integral weight lattice.
Studying these posets would be of interest in general, but would require techniques beyond the scope presented here.

\section{Type \texorpdfstring{$A$}{A}}\label{sec:A}
In this section, we consider the Lie algebra of type $A_r$ for  $r\geq 1$. The set of simple roots is given by $\Delta=\{\alpha_1,\alpha_{2}, \cdots, \alpha_{r}\}$,
and the set of positive roots is given by
\[
\Phi^+ = \Delta \cup \{\alpha_i+\alpha_{i+1}+\cdots+\alpha_j: 1 \leq i < j \leq r\}.
\]
For $1 \leq i \leq r$, the fundamental weights are defined  by \[\varpi_i = \frac{r+1-i}{r+1}(\alpha_1+2\alpha_2+\cdots+(i-1)\alpha_{i-1}) +\frac{i}{r+1}((r-i+1)\alpha_i+(r-i)\alpha_{i+1}+\cdots+\alpha_r). \]
The weight $\rho$ is defined as the half sum of the positive roots, $\rho=\frac{1}{2} \sum_{\alpha\in\Phi^+} \alpha$, which is equivalent to $\rho = \varpi_1+\varpi_2+\cdots+\varpi_r$. As noted previously, the Weyl group elements are generated by reflections about the hyperplanes that lie perpendicular to the simple roots. We denote these simple reflections by $s_i$, where $1 \leq i \leq r$, whose action on the simple roots is defined by
$
s_{i}(\alpha_{j})=\alpha_{j}$ if $|i-j|>1$, $s_{i}(\alpha_{j})=
    -\alpha_{j} $ if $i=j$, and $s_{i}(\alpha_{j})=
    \alpha_{i}+\alpha_{j}$ if $|i-j|=1$. 
We also recall that the Weyl group elements act on the fundamental weights by
\[s_{i}(\varpi_{j})=\begin{cases}\varpi_j-\alpha_i&\mbox{if $i=j$}\\
\varpi_j&\mbox{otherwise.}
\end{cases}
\]

We provide a proof of Theorem \ref{thm:mainA} by establishing the following technical results.

\begin{proposition}
If $\lambda = \ell\w_1$, $\mu=\sum_{1\leq i\leq r}m_i\w_i$, where $\ell, m_i\in\mathbb{Z}_+$ and $(\ell-\sum_{1\leq i\leq r} im_i)=0 \mod (r+1)$, then $s_1\notin \A(\lambda,\mu)$ for any choice of $\ell$ and $r$ in $\mathbb{Z}_+$.
\end{proposition}
\begin{proof}
Note that $\ell = (r+1)n +\sum im_i$ for some $n\in \ZZ_{\ge 0}$ and that $s_1(\lambda) = s_1(\ell\w_1)=\ell(\w_1 - \a_1)$. Hence,
\begin{align*}
    s_1(\lambda + \rho) - \rho - \mu &= \ell(\w_1 - \a_1)+(\rho-\a_1) - \rho - \mu \\
      &= \frac{\ell }{r+1}(r\a_1+(r-1)\a_2+\cdots+\a_r) - (\ell+1)\a_1 - (m_1\w_1+m_2\w_2+\cdots+m_r\w_r).
\end{align*}
Note that the coefficient of $\a_1$ in the last displayed equation is 
\begin{align*}
    \frac{\ell r}{r+1}-(\ell+1)-\sum_{i=1}^r \frac{m_i(r+1-i)}{r+1} 
      &= -\frac{\ell+r+1+\sum_{i=1}^r m_i(r+1-i)}{r+1}.
\end{align*}
Since $1\le i\le r$ and $\ell$, $r$, and $m_i$ are all nonnegative integers, this coefficient is negative for all choices of $r$ and $\ell$. Thus, $\wp(s_1(\lambda + \rho) - \rho - \mu )=0$ and $s_1\notin\A(\lambda,\mu)$.
\end{proof}

The following result establishes part 2 of Theorem \ref{thm:mainA}.
\begin{proposition}
Consider the Lie algebra of type $A_r$ with $r\geq 1$ and let $\ell\in\mathbb{Z}_+$. If $\lambda=\ell\varpi_1$, $\mu=\sum_{1\leq i\leq r}m_i\varpi_i$, where $m_i\in\mathbb{Z}_+$ and $(\ell-\sum_{1\leq i\leq r} im_i)=0$, $m_j=0$ for all $3\leq j\leq r$, then 
$\A(\lambda,\mu)=\{1\}$.
\end{proposition}
\begin{proof} We begin by establishing that under the given conditions, $s_i\notin\A(\lambda,\mu)$ for any $1<i\leq r$. Hence, it suffices to show that if  $1<i\leq r$, then
\[s_i(\lambda+\rho)-\rho-\mu=c_1\a_1+\cdots+c_r\a_r\]
and there exists $1\leq j\leq r$ such that $c_j<0$.
For $1<i\leq r$ note
\begin{align*}
s_i(\lambda+\rho)-\rho-\mu&=\ell\w_1+\rho-\a_i-\rho-\mu\\
&=\frac{2m_2}{r+1}(r\a_1+(r-1)\a_{2}+\cdots+\a_r)-\a_i\\
&\qquad\frac{-m_2}{r+1}((r-1)\a_1)-\frac{2m_2}{r+1}((r-1)\a_2+(r-2)\a_3+\cdots+\a_r)\\
&=m_2\a_1-\a_i.
\end{align*}
Hence, as desired, the coefficient $c_i=-1$ and $s_i\notin\A(\lambda,\mu)$. 
To complete the proof it suffices to note that since $\mu=m_1\w_1+m_2\w_2$, and $\ell=m_1+2m_2$, we have
\begin{align}
    1(\lambda+\rho)-\rho-\mu&=\lambda-\mu=(\ell-m_1)\w_1-m_2\w_2=m_2\a_1.\label{eq:1part2thmA}
\end{align}
Thus $\A(\lambda,\mu)=\{1\}$.
\end{proof}

The following result establishes part 3 of Theorem \ref{thm:mainA}.
\begin{proposition}
Type $A_r$ $(r\geq 1)$: If $\lambda=\ell\varpi_1$, $\mu=\sum_{1\leq i\leq r}m_i\varpi_i$, where $m_i\in\mathbb{Z}_+$ and $(\ell-\sum_{1\leq i\leq r} im_i)= 0$, $m_j=0$ for all $4\leq j\leq r$ and $m_3\neq 0$, then 
$\A(\lambda,\mu)=\{1,s_2\}$.
\end{proposition}

\begin{proof}We begin by establishing that under the given conditions $s_i\notin\A(\lambda,\mu)$ for any $2<i\leq r$. 
Hence, it suffices to show that if  $2<i\leq r$, then
\[s_i(\lambda+\rho)-\rho-\mu=c_1\a_1+\cdots+c_r\a_r\]
and there exists $1\leq j\leq r$ such that $c_j<0$.
For $2<i\leq r$ note
\begin{align}\nonumber
    s_i(\lambda+\rho)-\rho-\mu
&=(\ell-m_1)\w_1-m_2\w_2-m_3\w_3-\a_i\\\nonumber
&=\frac{(2m_2+3m_3)r-m_2(r-1)-m_3(r-2)}{r+1}\a_1\\\nonumber
&\qquad+\frac{(2m_2+3m_3)(r-1)-2m_2(r-1)-2m_3(r-2)}{r+1}\a_2-\a_i\\
&=(m_2+2m_3)\a_1+m_3\a_2-\a_i.\label{eq:1part3thmA}
\end{align}
Hence, as desired, the coefficient of $\a_i$ is negative, and $s_i\notin\A(\lambda,\mu)$, whenever $2<i\leq r$.

To complete the proof it suffices to show that $1$ and $s_2$ are both in $\A(\lambda,\mu)$. First note that since $\mu=m_1\w_1+m_2\w_2+m_3\w_3$, we know $\ell=m_1+2m_2+3m_3$,  hence $\ell-m_1=2m_2+3m_3$. Then
\begin{align*}
    1(\lambda+\rho)-\rho-\mu&=(2m_2+3m_3)\a_1+m_3\a_2\mbox{ and}\\
    s_2(\lambda+\rho)-\rho-\mu&=(2m_2+3m_3)\a_1+(m_3-1)\a_2.
\end{align*}
Since $m_3\geq 1$, we have established that $1,s_2\in\A(\lambda,\mu)$.
\end{proof}

The following result establishes part 4 of Theorem \ref{thm:mainA}.

\begin{proposition}\label{prop:m4}
Type $A_r$ $(r\geq 1)$: If $\lambda=\ell\varpi_1$, $\mu=\sum_{1\leq i\leq r}m_i\varpi_i$, where $m_i\in\mathbb{Z}_+$ and $(\ell-\sum_{1\leq i\leq r} im_i)= 0$, $m_j=0$ for all $5\leq j\leq r$ and $m_4\neq 0$, then 
$\A(\lambda,\mu)=\{1,s_2,s_3,s_2s_3\}$ when $m_4=1$ and $\A(\lambda,\mu)=\{1,s_2,s_3,s_2s_3,s_3s_2,s_2s_3s_2\}$ when $m_4\ge 2$.
\end{proposition}

\begin{proof}
For $i\geq 2$, we compute
\begin{align*}
    s_i(\lambda+\rho)-\rho-\mu&=\ell\w_1+\rho-\a_i-\rho-\mu\\
&=(\ell-m_1)\w_1-m_2\w_2-m_3\w_3-m_4\w_4-\a_i\\
&=\frac{(2m_2+3m_3+4m_4)r-m_2(r-1)-m_3(r-2)-m_4(r-3)}{r+1}\a_1\\
&\qquad+\frac{(2m_2+3m_3+4m_4)(r-1)-2m_2(r-1)-2m_3(r-2)-2m_4(r-3)}{r+1}\a_2\\
&\qquad+\frac{(2m_2+3m_3+4m_4)(r-2)-2m_2(r-2)-3m_3(r-2)-3m_4(r-3)}{r+1}\a_3-\a_i\\
&=(m_2+2m_3+3m_4)\a_1+(m_3+2m_4)\a_2+m_4\a_3-\a_i.
\end{align*}
Note that, in this case $\ell=m_1+2m_2+3m_3+4m_4$; hence $\ell-m_1=2m_2+3m_3+4m_4$. If $i \geq 4$, then $s_i\left(\lambda + \rho\right) - \rho - \mu$ has a negative coefficient for $\a_i$, and hence $s_i\notin \A(\lambda,\mu)$. Thus, $\A(\lambda,\mu)$ can only contain Weyl group elements formed by the simple reflections $s_2$ and $s_3$.

Using the similar calculations for $i\neq j$ where $i,j\ge 2$ and $|i-j|=1$, note that
\begin{align*}
    s_is_j(\lambda+\rho)-\rho-\mu&=\ell\w_1+\rho-2\a_i-\a_j-\rho-\mu\\
&=(\ell-m_1)\w_1-m_2\w_2-m_3\w_3-m_4\w_4-2\a_i-\a_j\\
&=(m_2+2m_3+3m_4)\a_1+(m_3+2m_4)\a_2+m_4\a_3-2\a_i-\a_j.
\end{align*}
where the same conversion from fundamental weights to simple roots applies. Notice that if $m_4=1$, $s_2s_3$ is the only element of this form in the Weyl alternation set. If $m_4\geq 2$, then $s_3s_2$ is also in the Weyl alternation set. Also, we have
\begin{align*}
    s_2s_3s_2(\lambda+\rho)-\rho-\mu&=\ell\w_1+\rho-2\a_2-2\a_3-\rho-\mu\\
&=(\ell-m_1)\w_1-m_2\w_2-m_3\w_3-m_4\w_4-2\a_2-2\a_3\\
&=(m_2+2m_3+3m_4)\a_1+(m_3+2m_4)\a_2+m_4\a_3-2\a_2-2\a_3.
\end{align*}
so $s_2s_3s_2$ is also in the Weyl alternation set when $m_4\geq 2$.
\end{proof}

\subsection{Low rank examples: types \texorpdfstring{$A_2$}{A2} and \texorpdfstring{$A_3$}{A3}}\label{sec:posetsA}
In this section, we consider the Lie algebra of types $A_2$ and $A_3$. By fixing the weight $\mu$ and varying the weight $\lambda$, we can make a diagram over the fundamental weight lattice of the Lie algebra by coloring two weights with the same color when they have the same Weyl alternation set.

Figure \ref{typeAstuff} presents Weyl alternation diagrams in the Lie algebra of type $A_2$ where we fix a weight $\mu$ in the root lattice and vary $\lambda$ over the fundamental weight lattice. The case of $\mu=0$ (Figure \ref{zeroA}) first appeared in \cite{PHThesis} and later, Harris in collaboration with  Lescinsky and  Mabie, generalized these results to all integral weights $\mu$ of $\mathfrak{sl}_{3}(\mathbb{C})$~\cite{HLM}. Figures \ref{fig:a1+a2} and \ref{fig:2a1+2a2} illustrate the cases where $\mu=\w_1+\w_2$ and $\mu=2\w_1+2\w_2$, respectively. 

\begin{figure}[H]
\centering
\input{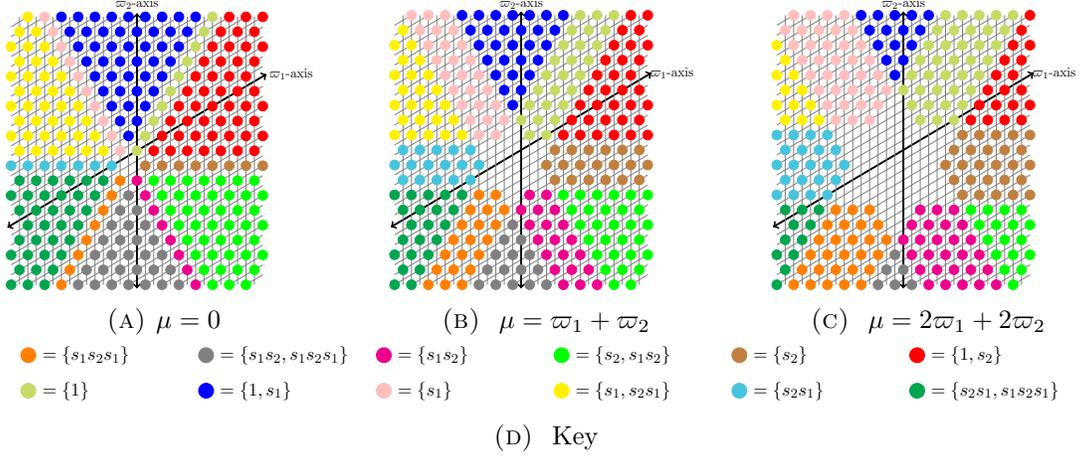}
    \caption{Weyl alternation diagram for the Lie algebra of type $A_2$ for a fixed $\mu$ and $\lambda=m\w_1+n\w_2$ with $m,n\in\mathbb{N}$.}\label{typeAstuff}
\end{figure}

As depicted in Figure \ref{typeAstuff}, the colored regions represent the intersection of specific hyperplanes (within the fundamental weight lattice) arising from the action of the elements of the Weyl group on the input of the partition function in \eqref{KWMF}. 
That is, each element $\sigma$ of the Weyl group gives rise to a region on the fundamental weight lattice arising from the intersection of some hyperplanes, from which one can create Weyl alternation diagrams.

To illustrate how to create Weyl alternation diagrams we fix $r=3$. For each pair $(\ell\w_1,\mu)$ for which $m(\ell\w_1,\mu)=1$, we have that $\mu = m_1\varpi_1+m_2\varpi_2+m_3\varpi_3$ and so $\ell - (m_1+2m_2+3m_3) = 4p$ for $p \in \mathbb{Z}_{\geq0}$. We also have that 
\begin{align*}
    1(\lambda+\rho)-\rho-\mu &= \ell\varpi_1-(m_1\varpi_1+m_2\varpi_2+m_3\varpi_3) =
     \left(m_2+2m_3+3p\right)\alpha_1+\left(m_3+2p\right)\alpha_2+p\alpha_3.
\end{align*}
Hence, if $i\geq2$, then
\begin{align*}
    s_i(\lambda+\rho)-\rho-\mu &= \lambda-\mu - \alpha_i\\
    s_is_{i+1}(\lambda+\rho)-\rho-\mu&=\lambda-\mu - 2\alpha_i - \alpha_{i+1} \\
    s_{i+1}s_i(\lambda+\rho)-\rho-\mu &= \lambda-\mu - \alpha_i - 2\alpha_{i+1} \\
  s_is_{i+1}s_i(\lambda+\rho)-\rho-\mu &= \lambda-\mu -2\alpha_i - 2\alpha_{i+1}.
  \end{align*}
The above equations imply that
\begin{itemize}
    \item $s_2 \in \A(\lambda,\mu)$  if and only if $m_3+2p \geq 1$, 
    \item $s_3 \in \A(\lambda,\mu)$ if and only if $p\geq 1$
    \item $s_2s_3 \in \A(\lambda,\mu)$ if and only if  $m_3+2p \geq 2$  and $p \geq 1$ 
    \item $s_3s_2 \in \A(\lambda,\mu)$ if and only if  $m_3+2p \geq 1$  and $p \geq 2$
    \item $s_2s_3s_2 \in \A(\lambda,\mu)$ if and only if $m_3+2p \geq 2$   and $p \geq 2$
\end{itemize}

Hence, if $p=0$, then $\A(\lambda, \mu)=\{1\}$; if $p=1$,  then $\A(\lambda, \mu)=\{1,s_2,s_3,s_2s_3\}$; and if $p\ge 2$, then $\A(\lambda, \mu)=\{1,s_2,s_3,s_2s_3,s_3s_2,s_2s_3s_2\}$.

From the inequalities above, the inclusion of $\sigma\in W$ in the Weyl alternation set $\A(\lambda, \mu)$ as prescribed relies only on the value of the nonnegative integer $p$. 
We note that the diagrams illustrating the allowable $\mu$ is either one of the highlighted weights in Figure~\ref{fig:typeA3} or there is no such $\mu$, meaning that $\A(\lambda,\mu)=\emptyset$.

\begin{figure}[H]
\centering
\begin{subfigure}[b]{0.40\textwidth}\centering
    \resizebox{!}{1.55in}{%
  \includegraphics{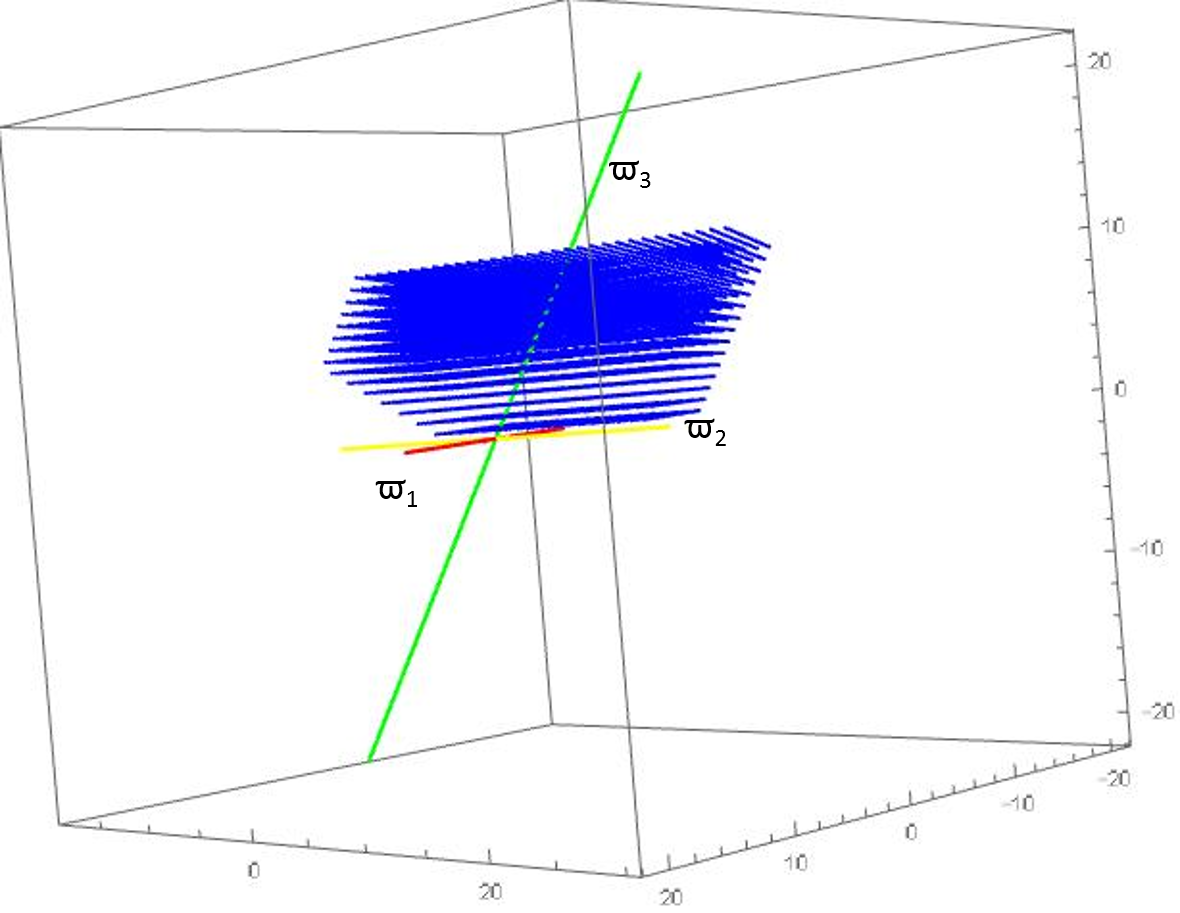}}
  \caption{Standard view in 3D}
  \label{fig:A3-standard}
\end{subfigure}
\qquad
\begin{subfigure}[b]{0.40\textwidth}\centering
    \resizebox{!}{1.55in}{%
  \includegraphics{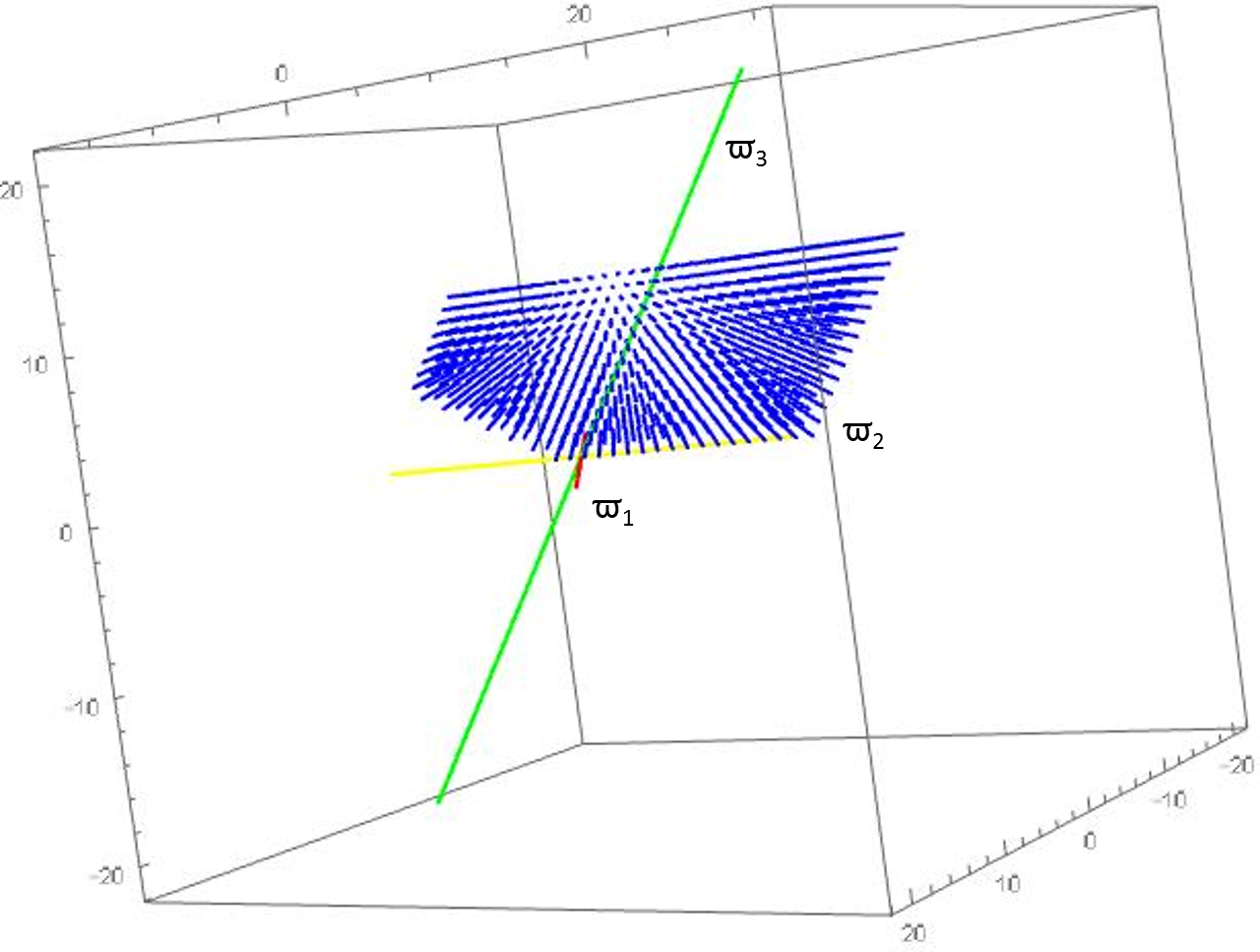}}
  \caption{View from positive $\w_1$-axis in 3D}
  \label{fig:A3-w1axis}
\end{subfigure}\\
\begin{subfigure}[b]{0.40\textwidth}\centering
    \resizebox{!}{1.55in}{%
  \includegraphics{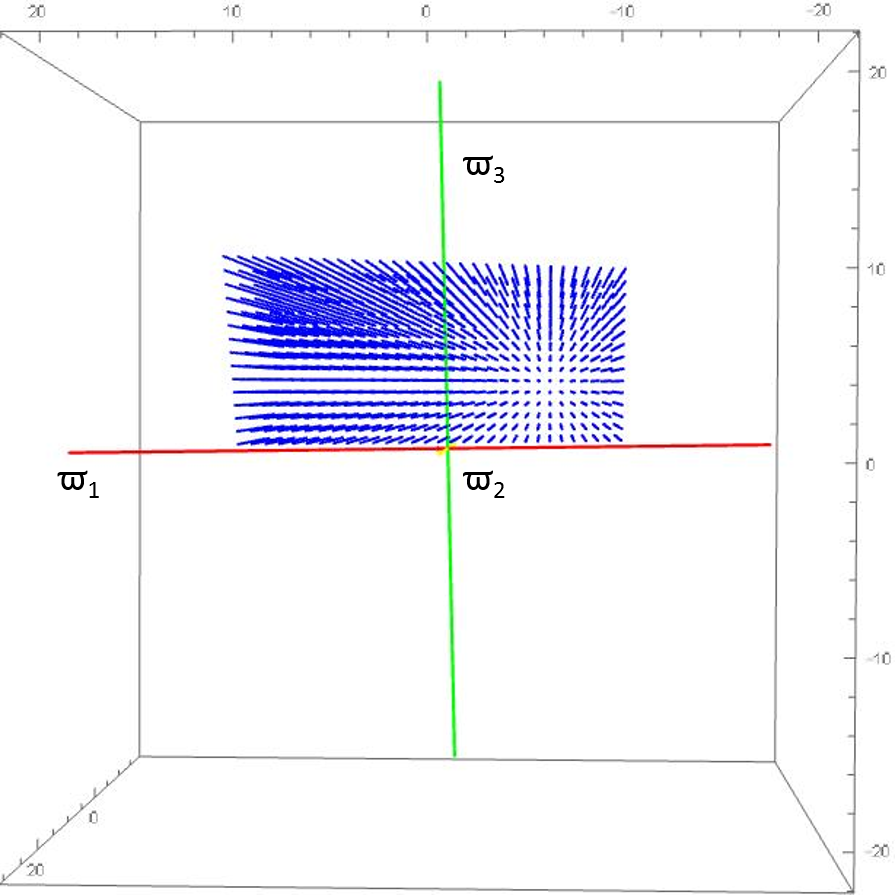}}
  \caption{View from positive $\w_2$-axis in 3D}
  \label{fig:A3-w2axis}
\end{subfigure}
\qquad
\begin{subfigure}[b]{0.40\textwidth}\centering
    \resizebox{!}{1.55in}{%
  \includegraphics{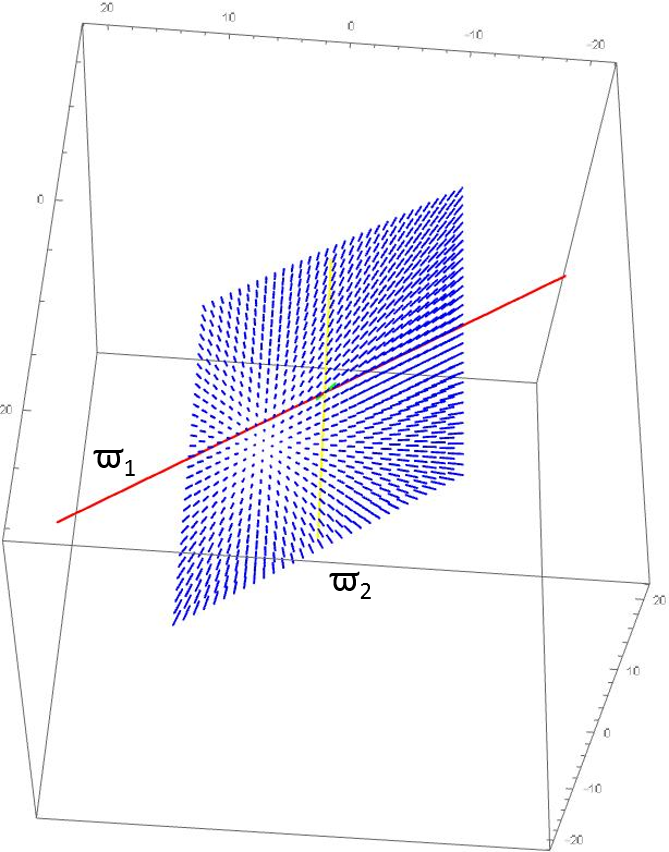}}
  \caption{View from positive $\w_3$-axis in 3D}
  \label{fig:A3-w3axis}  
\end{subfigure}
\caption{Diagrams highlighting weights in the Weyl alternation sets $\A(\lambda,\mu)$ when $\lambda=\ell\w_1$ in the Lie algebra of type $A_3$.}\label{fig:typeA3}
\end{figure}

These low rank computations illustrate the complexity in determining whether an element $\sigma$ of the Weyl group is in a specific Weyl alternation set $\A(\lambda,\mu)$. Deciding whether or not $\sigma$ is in $\A(\lambda,\mu)$ requires determining if $\sigma(\lambda+\rho)-\rho-\mu$ lies in a specified region of the integral root lattice. Although this may sound simple, the number of such regions as we vary over all pairs of weights of a Lie algebra is currently unknown. 

\section{Type \texorpdfstring{$B$}{B}}\label{sec:B}
In this section, we consider the Lie algebra of type $B_r$ for $r\geq 2$. The set of simple roots is given by $\Delta=\{\alpha_1,\ldots,\alpha_{r}\}$
and the set of positive roots is given by
$
\Phi^+=\{\varepsilon_i-\varepsilon_j,\varepsilon_i+\varepsilon_j:\;1\le i<j\le n\}\cup\{\varepsilon_i: 1\le i\le r\}.
$
The fundamental weights are defined by $\varpi_i=\varepsilon_1+\cdots+\varepsilon_i$ for $1\le i\le r-1$ and $\varpi_r=\frac{1}{2}(\varepsilon_1+\varepsilon_2+\cdots+\varepsilon_r)$. The weight $\rho$ is defined as half the sum of the positive roots, which is equivalent to $\rho = \varpi_1+\cdots+\varpi_r$. The Weyl group elements are generated by the simple root reflections $s_1,s_2,\ldots,s_r$ which act on the simple roots and fundamental weights as follows. If $1\leq i\leq r-1$, then $s_i(\a_i)=-\a_i$, $s_{i}(\a_{i-1})=\a_{i-1}+\a_i$, $s_{i}(\a_{i+1})=\a_i+\a_{i+1}$, and $s_r(\a_r)=-\a_r$, $s_r(\a_{r-1})=\a_{r-1}+2\a_r$.
For any $1\leq i,j\leq r$, \[s_{i}(\varpi_{j})=\begin{cases}\varpi_j-\alpha_i&\mbox{if $i=j$}\\
\varpi_j&\mbox{otherwise.}
\end{cases}
\]

We provide a proof of Theorem \ref{thm:mainB} by establishing the following technical results.

\begin{lemma}
In type $B_r$ with $\ell, r \in 2\ZZ_{\ge 0}$ under the condition $\ell-1=\sum_{1\le i< r} im_i + \frac{rm_r}{2}$ with $m_i\in 2\ZZ_{\ge 0}$, then $\A(\ell\w_1,\mu)=\emptyset$.
\end{lemma}

\begin{proof}
Suppose $\ell$ is even. If $r$ is even, then $\frac{rm_r}{2}$ is even since $m_r$ is even. Thus, $\sum_{1\le i <r} im_i + \frac{rm_r}{2}$ is also even. Then $\ell=1+\sum_{1\le i <r} im_i + \frac{rm_r}{2}$ must be odd. Since $\ell$ is even, then there are no values of the $m_i$'s ($1\le i\le r$) which satisfy the condition. Hence, $\A(\ell\w_1,\mu)=\emptyset$.
\end{proof}

\begin{proposition}\label{prop:sinonconsec}
  Let $\sigma=s_{i_1}s_{i_2}\cdots s_{i_k}$ where the indices of the simple reflections form a collection of nonconsecutive integers $2\leq i_1,\dots,i_k \leq r$. If $\lambda = \ell\w_1$ and $\mu = (\ell - 1 - 2b-3c)\w_1 + b\w_2 + c\w_3$, for $\ell,b,c$ nonnegative integers such that $\ell > b$ and $\ell > c$, then we have $\sigma(\lambda + \rho) - \mu - \rho$ is a nonnegative $\mathbb{Z}$-linear combination of positive roots.
\end{proposition}

\begin{proof}
  Let $\sigma=s_{i_1}s_{i_2}\cdots s_{i_k}$ for some collection of nonconsecutive integers $2\leq i_1,\dots,i_k \leq r$. Note that $\sigma(\lambda) = \lambda$ and $\sigma(\rho)=\rho-\sum_{j=1}^k \a_{i_j}$. Hence, if $\lambda = \ell\w_1$ and $\mu = (\ell - 1 - 2b-3c)\w_1 + b\w_2 + c\w_3$, for $\ell,b,c$ nonnegative integers such that $\ell > b$ and $\ell > c$, then we have $\sigma(\lambda+\rho)-\mu-\rho = \lambda - \mu - \sum_{j=1}^k \a_{i_j} = \w_1+(b+2c)\a_1 + c\a_2 - \sum_{j=1}^k \a_{i_j}$.
\end{proof}

\begin{proposition}\label{lem:Bnonconsprod}
In type $B_r$  with $r\geq 3$ and an odd, positive integer $\ell$, all  $\sigma\in\A(\ell\w_1,(\ell-1-4k)\w_1+2k\w_2)$ where $k\in \ZZ_{\geq 0}$ are products of nonconsecutive $s_i$'s with $1<i\leq r$.
\end{proposition}

\begin{proof}
Let $\l=\ell\w_1$ and $\mu=(\ell-1-4k)\w_1+2kw_2$ for some $k\in \ZZ{\geq 0}.$ First, we note that $\ell-1=\ell-1-4k+2(2k)$ so $\ell > 2k$ and $\l\succ\mu$.

For $r\geq 3$, note that $\ell\w_1=\ell(\a_1+\cdots+\a_r)$. We claim that the only elements $\sigma\in W$ for which $\wp(\sigma(\ell\w_1+\rho)-\rho-(\ell-1-4k)\w_1-2k\w_2)>0$ are $\sigma=s_{i_1}s_{i_2}\cdots s_{i_m}$ such that $i_1,i_2,\ldots,i_m$ are nonconsecutive integers between 2 and $r$ (inclusive).

\noindent$(\Leftarrow)$ Let $\sigma = 1$, then $1(\l + \rho) - (\mu+\rho) = \l-\mu=\w_1+2k\a_1$ is a nonnegative $\ZZ$-linear combination of positive roots. Thus $1 \in \mathcal A(\l,\mu)$. Proposition \ref{prop:sinonconsec} implies that if $\sigma=s_{i_1}s_{i_2}\cdots s_{i_m}$ for some nonconsecutive integers $2 \leq i_1,\ldots,i_m \leq r$, then $\sigma\in\mathcal{A}(\l,\mu)$.

\noindent$(\Rightarrow)$ Suppose $\sigma\in\mathcal A(\l,\mu)$. We proceed by induction on $\ell(\sigma)$. If $\ell(\sigma)=0$, then $\sigma=1$, which satisfies the needed condition. If $\ell(\sigma)=1$, then $\sigma=s_i$ for some $1 \le i \le r$. If $i=1$, then $s_{1}(\l + \rho) - (\mu+\rho) =  \w_1+2k\a_1-(\ell+1)\a_1$, which implies $s_{1} \notin \mathcal A(\l,\mu)$. Thus, $\sigma\in\mathcal{A}(\l,\mu)$ cannot contain $s_{1}$ in its reduced word expression.
If $1<i\leq r$, then $s_{i}(\l + \rho) - (\mu+\rho) =\w_1+2k\a_1-\a_i$, and   $s_i\in\mathcal{A}(\l,\mu)$ and $s_i$ is of the required form.

If $\ell(\sigma)=2$, then $\sigma=s_is_j$ for distinct integers $i,j$ satisfying $1 < i,j \leq r$. 
Without loss of generality, assume $i < j$. 
If $i,j$ are consecutive integers, then
$i = j-1$, with $1 < i,j < r$ or $i = r-1$ and $j = r$. For the case $1<i,j<r$, note
    $s_{j-1}s_{j}(\l + \rho) - (\mu+\rho )
         = \l-\mu-2\alpha_{j-1}-\alpha_{j}$. 
For the case $i=r-1$ and $j=r$, note
    $s_{r-1}s_{r}(\l + \rho) - (\mu+\rho)
         = \l-\mu-2\alpha_{r-1}-\alpha_{r}$.
None of these can be written as a nonnegative $\ZZ$-linear combination of positive roots. Thus, $s_{r-1}s_r$, $s_{r}s_{r-1}$, $s_{j-1}s_j$, $s_{j}s_{j-1} \notin \mathcal A(\l,\mu)$.
Moreover, any $\sigma\in W$ containing $s_{j}s_{j-1}$ or $s_{j-1}s_{j}$ in its reduced word expression cannot be in $\mathcal A(\l,\mu)$ for all $2<j\leq r$. The case where $i,j$ are nonconsecutive was already considered in Proposition \ref{prop:sinonconsec}.
\end{proof}

Before stating our next result we recall that the Fibonacci numbers are defined by the recurrence relation $F_n=F_{n-1}+F_{n-2}$ for all $n\geq 3$ and $F_1=F_2=1$. 
\begin{corollary}\label{cor:Bnonconsprodcount}
In type $B_r$ with $r\geq 3$, if $\ell$ is an odd positive integer and  $k\in \ZZ_{\geq 0}$, then $\left|\A(\ell\w_1,(\ell-1-4k)\w_1+2k\w_2)\right| = F_{r+1}$. 
\end{corollary}

\begin{proof}
We show $\left|\mathcal{A}\left(\lambda,\mu\right)\right| = F_{r+1}$ for $B_r$ using induction. First, if $r=2$, then $\left|\mathcal{A}\left(\lambda,\mu\right)\right|=\left|\{1,s_2\}\right| = 2 = F_3$. If $r=3$, then $\left|\mathcal{A}\left(\lambda,\mu\right)\right|=\left|\{1,s_2,s_3\}\right| = 3 = F_4$. Assume for all $r$, $4\leq r \leq m$, $\left|\mathcal{A}\left(\lambda,\mu\right)\right|=\left|\{1,s_2\}\right| = F_{r+1}$, the $(r+1)^{\text{st}}$ Fibonacci number. We consider the case when $r = m+1$. All elements $\sigma\in\mathcal{W}$ consisting of nonconsecutive products of the generators $s_2,s_3,\dots,s_m$ will either contain $s_{m+1}$ or not. If they don't contain $s_{m+1}$, then by our inductive hypothesis, the number of Weyl group elements consisting of nonconsecutive products of $s_2,s_3,\dots,s_m$ is given by $F_{m+1}$. If the Weyl group elements contain $s_{m+1}$, then we must count the number of nonconsecutive products of $s_2,s_3,\dots,s_{m-1}$, which by our inductive hypothesis is given by $F_m$. Therefore, $\left|\mathcal{A}\left(\lambda,\mu\right)\right|= F_{m}+F_{m+1} = F_{m+2}$. 
\end{proof}
The following result establishes part 3 of Theorem \ref{thm:mainB}.
\begin{proposition}
In type $B_r$ with $r>3$ and an odd, positive integer $\ell$, all $\sigma \in \A(\ell\w_1,(\ell-1-4j-6k)\w_1+2j\w_2+2k\w_3)$ where $j\in \ZZ_{\geq 0}$, $k\in \ZZ_{>0}$, $\ell>2j$, and $\ell > 4k$ are either products of nonconsecutive $s_i$'s with $1<i\leq r$ or products of the form $s_{i_1}\dots s_{i_m}s_2s_3$ with nonconsecutive $s_{i_p}$'s for $4<i_p\leq r$.
\end{proposition}

\begin{proof}
\noindent$(\Leftarrow)$ Let $\sigma = 1$. Then $1(\lambda + \rho) - \rho - \mu = \w_1 + (2j+4k)\a_1+2k\a_2$, which is a nonnegative $\ZZ$-linear combination of positive roots. Thus $1 \in \mathcal A(\l,\mu)$. Note that by Proposition \ref{prop:sinonconsec}, $\sigma = s_{i_1}s_{i_2}\cdots s_{i_m}\in \mathcal A(\l,\mu)$. We need only show that $\sigma = s_{i_1}s_{i_2}\cdots s_{i_m}s_2s_3 \in \mathcal A(\l,\mu)$. Note that $s_{i_1}\cdots s_{i_m}s_2s_3(\lambda + \rho) = s_{i_1}\cdots s_{i_m}(\lambda + \rho - 2\a_2 - \a_3) = \lambda + \rho - \sum_{p=1}^k \a_{i_p} + s_{i_1}\cdots s_{i_k}(-2\a_2 - \a_3)$. Since the $s_{i_p}$'s are a collection of nonconsecutive integers between $5$ and $r$, inclusive, $s_{i_1}\cdots s_{i_m}(-2\a_2 - \a_3) = -2\a_2 - \a_3$. Hence, $s_{i_1}\cdots s_{i_m}s_2s_3(\lambda + \rho) - \rho - \mu = \w_1 + (2j+4k)\a_1 + (2k-2)\a_2 - \a_3 - \sum_{p=1}^k \a_{i_p} \in \mathcal A(\l,\mu)$ if $k \neq 0$.
  
\noindent$(\Rightarrow)$ Suppose $\sigma \in \mathcal{A}(\lambda,\mu)$. Note that $\lambda - \mu = \w_1 + (2j+4k)\a_1 + 2k\a_2$. We proceed by induction on $\ell(\sigma)$. If $\ell(\sigma)=0$, then $\sigma = 1$, which satisfies the needed condition. If $\ell(\sigma) =1$, then $\sigma = s_i$ for some $1 \leq i \leq r$. If $i = 1$, then $s_1(\lambda + \rho) - \mu - \rho = \w_1 + (2j+4k)\a_1 + 2k\a_2 - (\ell+1)\a_1 \notin \mathcal{A}(\lambda,\mu)$. Thus $\sigma \in \mathcal{A}(\lambda,\mu)$ cannot contain $s_1$ in its reduced word expression. If $1 < i \leq r$, then $s_i(\lambda+\rho) - \mu - \rho = \w_1 + (2j+4k)\a_1 + 2k\a_2 - \a_i$ and so $s_i \in \mathcal{A}(\lambda,\mu)$ and $s_i$ is of the required form. 
  
  Now suppose $\ell(\sigma) = 2$. Then $\sigma = s_is_m$ for distinct integers $i,m$ such that $1 < i,m \leq r$. Without loss of generality, assume $i < m$. If $i = 2$ and $m = 3$, then we have $s_2s_3(\lambda + \rho) - \mu - \rho = \w_1 + (2j+4k)\a_1 + 2k\a_2 - 2\a_2 - \a_3$, which means $s_2s_3 \in \mathcal{A}(\lambda,\mu)$ if $k\neq 0$ and $s_2s_3$ is of the required form. Note that $s_3s_2(\lambda + \rho) - \mu - \rho = \lambda - \mu - \a_2 - 2\a_3$, which is not in $\mathcal{A}(\lambda,\mu)$ and is not of the required form. If $i,m$ are any other consecutive pair of integers, then $i = m-1$ with either $2 < i,m < r$ or $i = r-1$ and $m=r$. Suppose $2 < i,m < r$. Then $s_{m-1}s_m(\lambda+\rho) - \mu - \rho = \lambda - \mu - 2\a_{m-1} - \a_m$. If $i = r-1$ and $m = r$, then $s_{r-1}s_r(\lambda + \rho) - \mu - \rho = \lambda - \mu - 2\a_{r-1} - \a_r$. Note that neither of these cases can be written as a nonnegative $\mathbb{Z}$-linear combination of positive roots. Thus, $s_{r-1}s_r,s_rs_{r-1},s_{m-1}s_m,s_ms_{m-1} \notin \mathcal{A}(\lambda, \mu)$ for all cases except $s_2s_3$. Moreover, any $\sigma \in W$ containing $s_ms_{m-1}$ or $s_{m-1}s_m$ in its reduced word expression cannot be in $\mathcal{A}(\lambda,\mu)$ for all $2 < m \leq r$, except $s_2s_3$. If $i,m$ are nonconsecutive, then by Proposition \ref{prop:sinonconsec}, $s_is_m \in \mathcal{A}(\lambda,\mu)$.
  \end{proof}

\begin{corollary}
In type $B_r$ with $r>3$, if $\ell$ is an odd positive integer and $j\in \ZZ_{\geq 0}$, $k\in \ZZ_{>0}$, with $\ell>2j$, and $\ell > 4k$ , then $\left|\A(\ell\w_1,(\ell-1-4j-6k)\w_1+2j\w_2+2k\w_3)\right|=2F_{r}$.
\end{corollary}
 \begin{proof} With $\lambda,\mu$ as given, we show $\left|\mathcal{A}\left(\lambda,\mu\right)\right| = 2F_{r}$ for $B_r$. As in Corollary \ref{cor:Bnonconsprodcount}, we know there are $F_{r+1}$ elements coming from the products of nonconsecutive $s_i$'s, $1 < i \leq r$. By a similar argument, there are $F_{r-2}$ elements $\sigma = s_{i_1}s_{i_2}\cdots s_{i_m}s_2s_3$ for some nonconsecutive integers $4 < i_1, i_2, \dots, i_m \leq r$. Thus, $\left|\mathcal{A}\left(\lambda,\mu\right)\right| = F_{r+1} + F_{r-2} = F_r + F_{r-1} + F_{r-2} = 2F_{r}$.
\end{proof}

The condition for $B_r$ such that $\ell$ is even and $r$ is odd is quite different from the three conditions we posed in Theorem \ref{thm:mainB}. When r is odd and $\ell$ is even, the constraint $\ell - 1 = \sum_{1\leq i < r} im_i + \frac{rm_r}{2}$ forces $m_r = 2(2k+1)$, since $m_i$ is even for all $i$. The Weyl alternation sets $\mathcal{A}\left(\lambda,\mu\right)$ for $B_3$ are easy to compute and remain quite small. But when $r = 5$, the Weyl alternation sets grow very quickly. In the three conditions discussed in Theorem \ref{thm:mainB}, the elements of the Weyl alternation sets had a nice pattern of being products of mostly nonconsecutive simple reflections; however, we don't see similar patterns for the case when $r$ is odd and $\ell$ is even, as seen in the tables in Appendix \ref{alttables}. As the coefficient $m_r$ increases, the elements in the Weyl alternation set $\mathcal{A}(\lambda,\mu)$ become even more complicated.

\subsection{Low rank examples: \texorpdfstring{$B_2$}{B2} and \texorpdfstring{$B_3$}{B3}} \label{sec:posetsB}
In this section we consider the Lie algebra of types $B_2$ and $B_3$. By fixing the weight $\mu$ and varying the weight $\lambda$ we can make a diagram over the fundamental weight lattice of the Lie algebra by coloring two weights with the same color when they have the same Weyl alternation set.

Figure \ref{fig:diagramB2mu=0} presents the Weyl alternation diagram when $\mu=0$. Observe that there are 24 distinct colored dots and the integral weights without a colored dot represent those weights that have empty Weyl alternation set. This illustrates the 25 distinct Weyl alternation types given in~\cite[Theorem 2.3.1]{PHThesis}. 

\begin{figure}[H]
    \centering
    \includegraphics[width=2in]{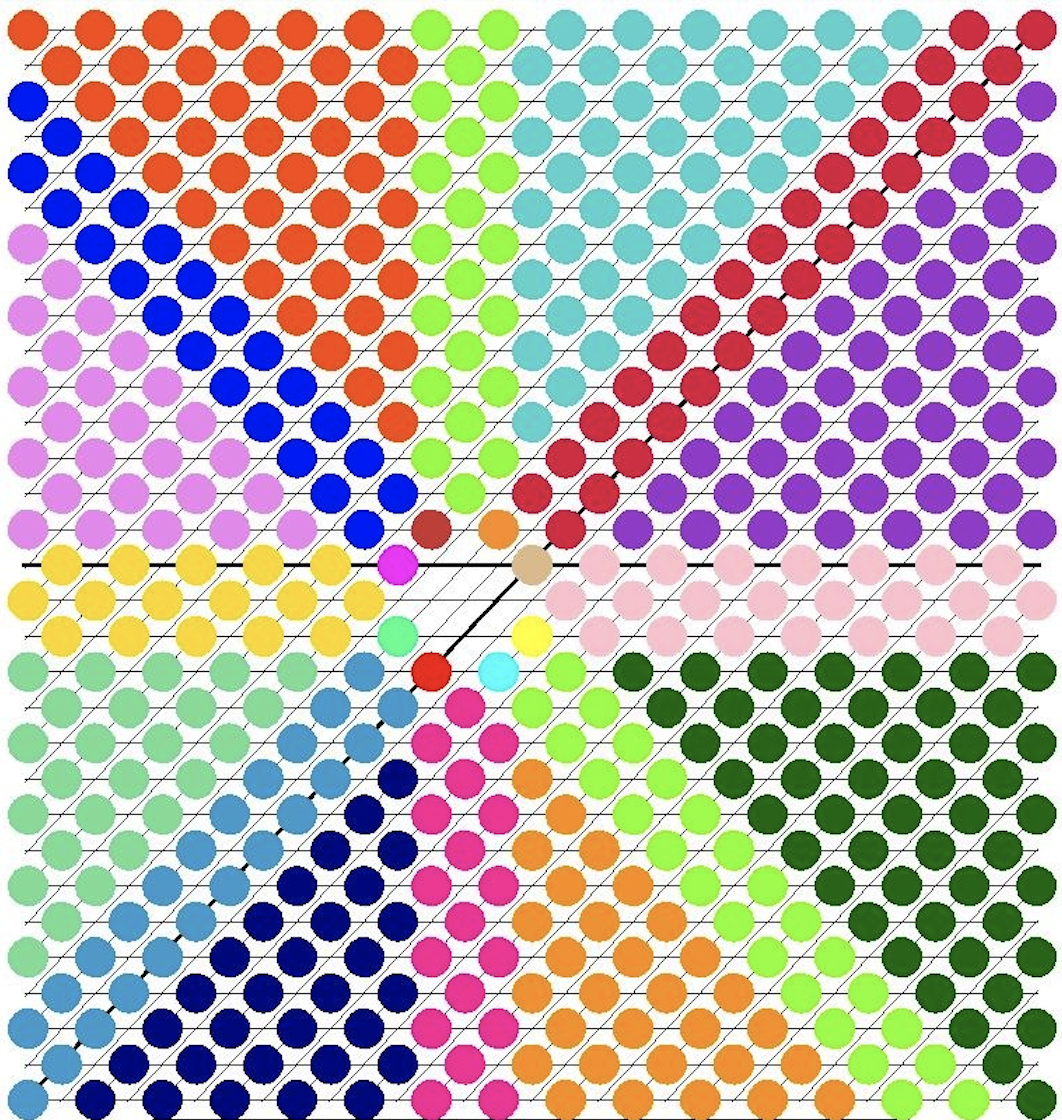}
    \includegraphics[width=2.5in]{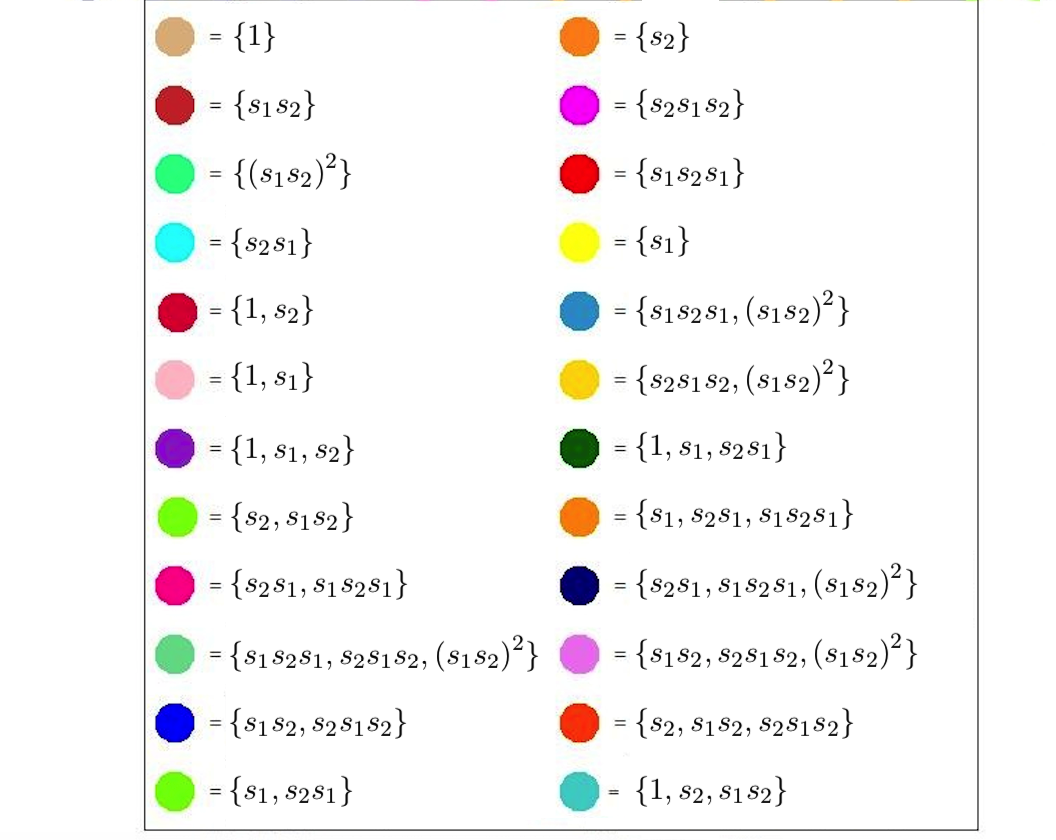}
    \caption{Weyl alternation diagram for the Lie algebra of type $B_2$ when $\mu=0$ and $\lambda=m\w_1+n\w_2$ with $m,n\in\mathbb{N}$.}
    \label{fig:diagramB2mu=0}
\end{figure}

Now we fix $r=3$. For each pair $(\ell\w_1,\mu)$ for which $m(\ell\w_1,\mu)=1$, we have that $\mu = m_1\varpi_1+m_2\varpi_2+m_3\varpi_3$ and $m_1,m_2,m_3\in2\mathbb{Z}_{\geq0}$ so $\ell = m_1+2m_2+\frac{3}{2}m_3$. We also have that 
\begin{align*}
    1\left(\lambda+\rho\right)-\left(\mu+\rho\right) &= \lambda-\mu
                               = \left(m_2+m_3+1\right)\a_1+\left(\frac{m_3+2}{2}\right)\a_2+\a_3\\
    s_2\left(\lambda+\rho\right)-\left(\mu+\rho\right) &= \lambda-\mu - \a_2
                               = \left(m_2+m_3+1\right)\a_1+\left(\frac{m_3}{2}\right)\a_2+\a_3\\  
    s_3\left(\lambda+\rho\right)-\left(\mu+\rho\right) &= \lambda-\mu-\a_3
                               = \left(m_2+m_3+1\right)\a_1+\left(\frac{m_3+2}{2}\right)\a_2\\
    s_2s_3\left(\lambda+\rho\right)-\left(\mu+\rho\right) &= \lambda-\mu-2\a_2-a_3
                               = \left(m_2+m_3+1\right)\a_1+\left(\frac{m_3-2}{2}\right)\a_2\\
    s_3s_2\left(\lambda+\rho\right)-\left(\mu+\rho\right) &= \lambda-\mu-\a_2-3\a_3
                               = \left(m_2+m_3+1\right)\a_1+\left(\frac{m_3}{2}\right)\a_2-2\a_3\\
    s_2s_3s_2\left(\lambda+\rho\right)-\left(\mu+\rho\right) &= \lambda-\mu-3\a_2-3\a_3
                               = \left(m_2+m_3+1\right)\a_1+\left(\frac{m_3-4}{2}\right)\a_2-2\a_3\\  
    s_3s_2s_3\left(\lambda+\rho\right)-\left(\mu+\rho\right) &= \lambda-\mu-2\a_2-4\a_3
                               = \left(m_2+m_3+1\right)\a_1+\left(\frac{m_3-2}{2}\right)\a_2-3\a_3\\
    \left(s_3s_2\right)^2\left(\lambda+\rho\right)-\left(\mu+\rho\right) &= \lambda-\mu-3\a_2-4\a_3
                               = \left(m_2+m_3+1\right)\a_1+\left(\frac{m_3-4}{2}\right)\a_2-3\a_3\\
\end{align*}
Therefore, if $m_3=0$, then $\A(\lambda,\mu)=\{1,s_2,s_3\}$. If $m_3\in 2\mathbb{Z}_{>0}$, then $\A(\lambda,\mu)=\{1,s_2,s_3,s_2s_3\}$. To display when $\sigma\in W$ is an element of $\A(\lambda,\mu)$, we provide Figure~\ref{fig:typeB3}, which show different views of the allowed choices for $\mu$. 

\begin{figure}[H]
\centering
\begin{subfigure}[b]{0.40\textwidth}\centering
    \resizebox{!}{1.55in}{%
  \includegraphics{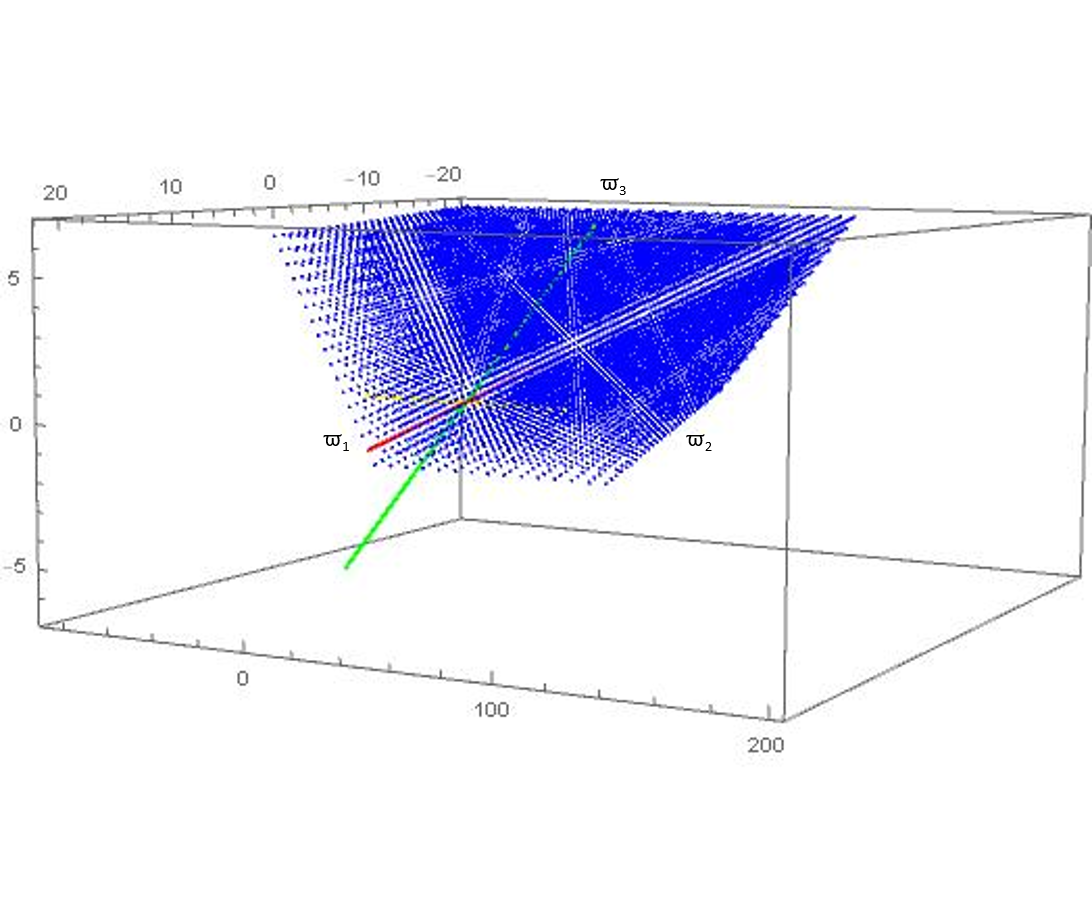}}
  \caption{Standard view in 3D}
  \label{fig:B3-standard}
\end{subfigure}
\qquad
\begin{subfigure}[b]{0.40\textwidth}\centering
    \resizebox{!}{1.55in}{%
  \includegraphics{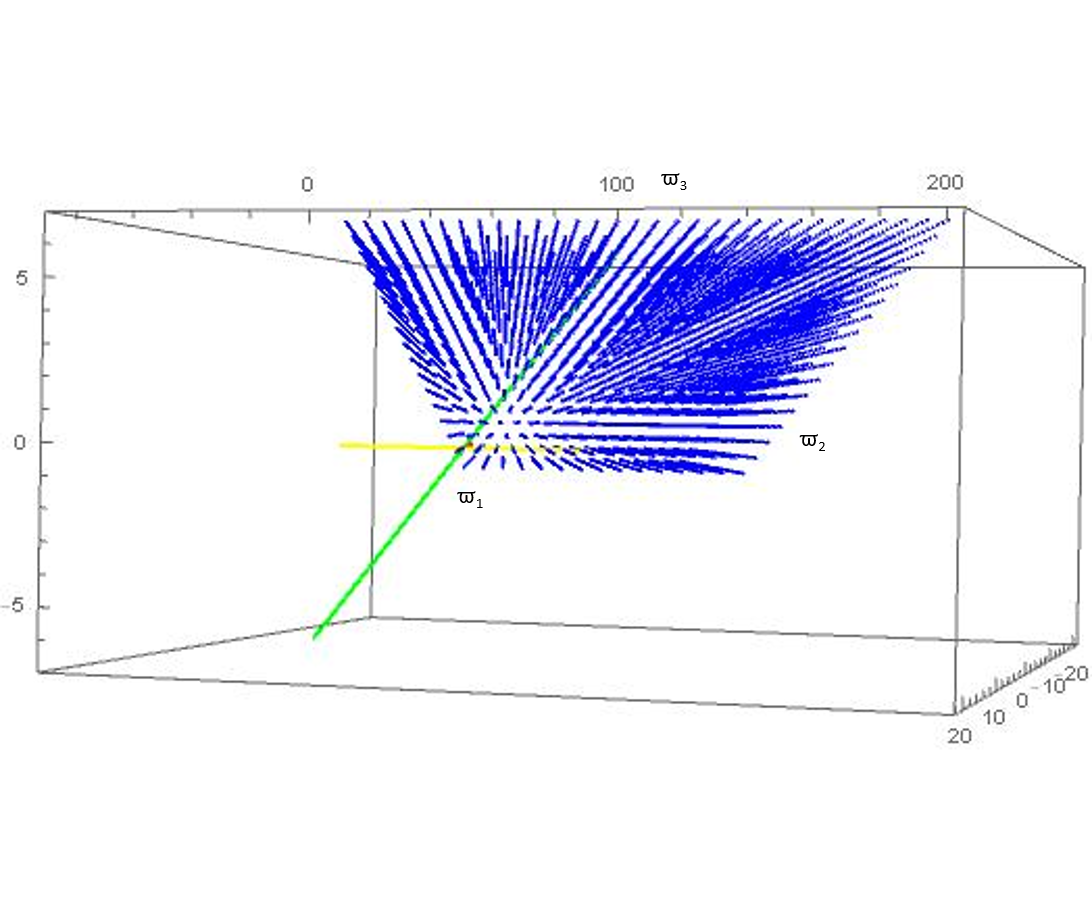}}
  \caption{View from positive $\w_1$-axis in 3D}
  \label{fig:B3-w1axis}
\end{subfigure}\\
\begin{subfigure}[b]{0.40\textwidth}\centering
    \resizebox{!}{1.55in}{%
  \includegraphics{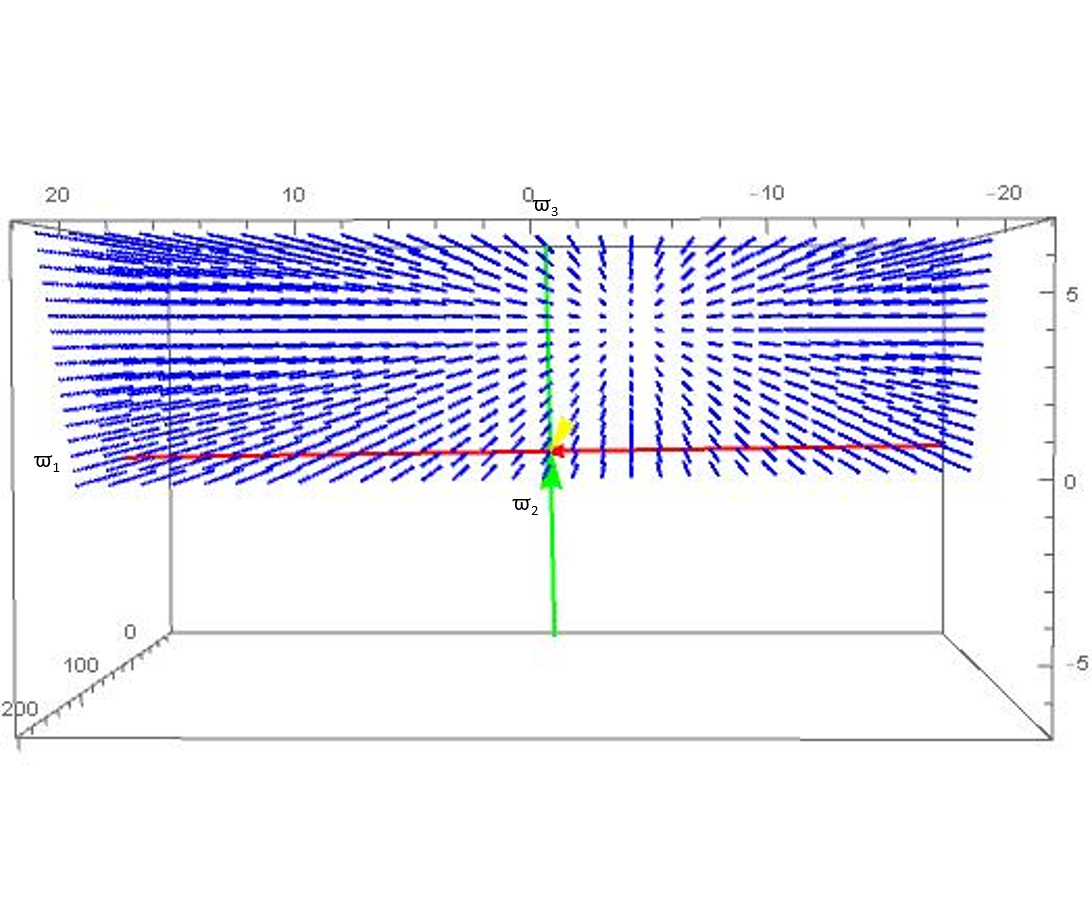}}
  \caption{View from positive $\w_2$-axis in 3D}
  \label{fig:B3-w2axis}
\end{subfigure}
\qquad
\begin{subfigure}[b]{0.40\textwidth}\centering
    \resizebox{!}{1.55in}{%
  \includegraphics{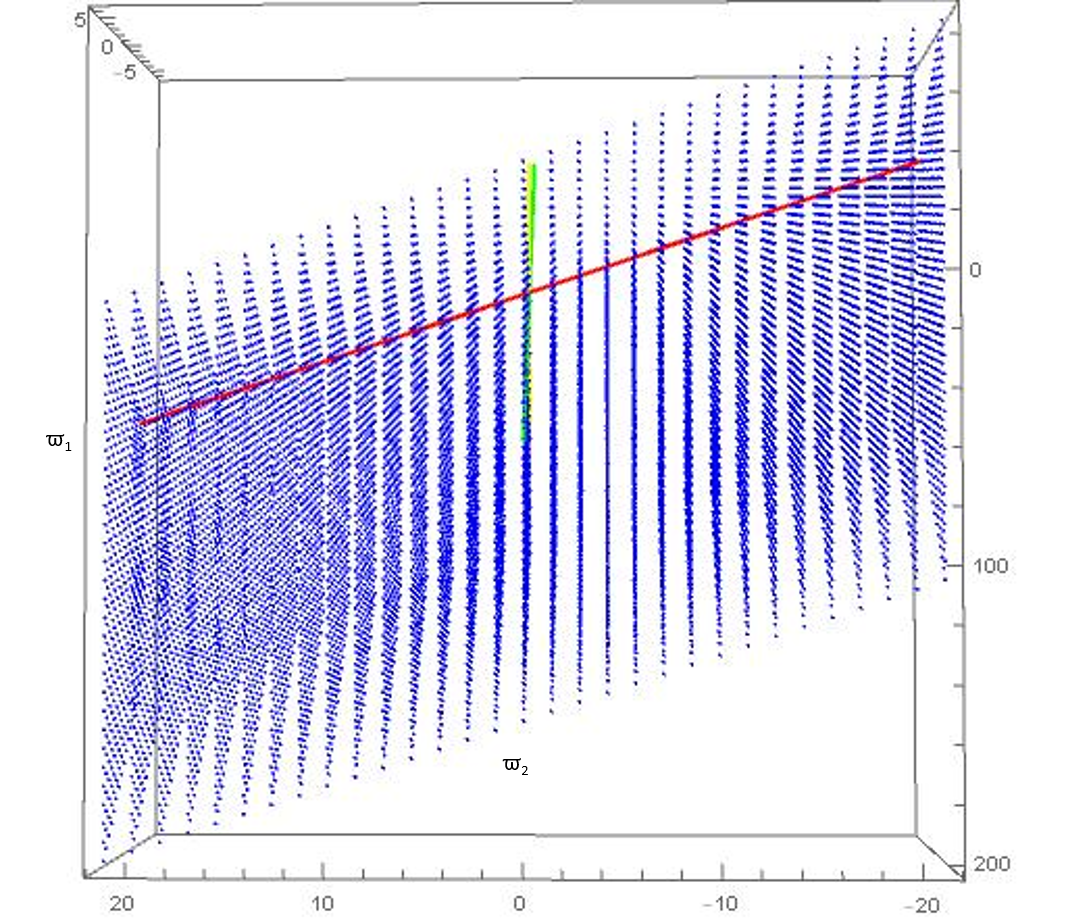}}
  \caption{View from positive $\w_3$-axis in 3D}
  \label{fig:B3-w3axis}  
\end{subfigure}
\caption{Diagrams highlighting weights in the Weyl alternation sets $\A(\lambda,\mu)$ when $\lambda=\ell\w_1$ in the Lie algebra of type $B_3$.}\label{fig:typeB3}
\end{figure}

\section{Type \texorpdfstring{$G_2$}{G2}}\label{sec:G2}
The positive roots of $G_2$ are 
$\Phi^+=\{\a_1,\a_2,\a_1+\a_2,2\a_1+\a_2,3\a_1+\a_2,3\a_1+2\a_2\}$ and the fundamental weights in terms of the simple roots are given by
$\w_1=2\a_1+\a_2$ and
$\w_2=3\a_1+2\a_2$.
\begin{theorem}\label{sets:G2}
Let $\lambda=\ell\w_2$ with $\ell\geq 1$ be a weight of $G_2$. If $\mu=m_1\varpi_1+m_2\varpi_2$, where $m_1,m_2\in\mathbb{Z}_{\geq 0}$ and $3\ell-1=2m_1+3m_2$, then 
$\A(\lambda,\mu)=\{1,s_1\}$ when $\lambda\neq 0$ and $\A(0,\mu)=\emptyset$.    \end{theorem}

\begin{proof}
Let $\lambda=\ell\w_2$ with $\ell\geq 1$ be a weight of $G_2$. 
If $\mu=m_1\varpi_1+m_2\varpi_2$, where $m_1,m_2\in\mathbb{Z}_+$ and $3\ell-1=2m_1+3m_2$, then we can parametrize $m_1=3n+1$ and $m_2=\ell-1-2n$, with $0\leq n\leq \lfloor\frac{\ell-2}{2}\rfloor$. Then \[\mu=(3n+1)\w_1+(\ell-1-2n)\w_2=(3n+1)(2\a_1+\a_2)+(\ell-1-2n)(3\a_1+2\a_2)=(3\ell-1)\a_1+(2\ell-n-1)\a_2\]
and some simple computations establish
\begin{align*}
1(\lambda+\rho)-\rho-\mu
&=\a_1+(n+1)\a_2\\
s_1(\lambda+\rho)-\rho-\mu
&=(n+1)\a_2\\
s_2(\lambda+\rho)-\rho-\mu
&=\a_1+(n-\ell)\a_2\\
s_2s_1(\lambda+\rho)-\rho-\mu
  &=3\a_1+(n-\ell+2)\a_2\\
(s_2s_1)^2(\lambda+\rho)-\rho-\mu
&=(-3\ell-5)\a_1+(-3\ell+n-4)\a_2\\
(s_2s_1)^3(\lambda+\rho)-\rho-\mu
&=(-6\ell-11)\a_1+(-4\ell+n-5)\a_2
\\
(s_2s_1)^4(\lambda+\rho)-\rho-\mu&=(-6\ell-8)\a_1+(-3\ell+n-3)\a_2\\
(s_2s_1)^5(\lambda+\rho)-\rho-\mu&=(-3\ell-3)\a_1+(-\ell+n)\a_2\\
s_1s_2s_1(\lambda+\rho)-\rho-\mu&=(-3\ell-5)\a_1+(-\ell+n-1)\a_2\\
s_1(s_2s_1)^2(\lambda+\rho)-\rho-\mu&=(-6\ell-9)\a_1+(-3\ell+n-4)\a_2\\
s_1(s_2s_1)^3(\lambda+\rho)-\rho-\mu&=(-6\ell-8)\a_1+(-4\ell+n-5)\a_2\\
s_1(s_2s_1)^4(\lambda+\rho)-\rho-\mu&=(-3\ell-3)\a_1+(-3\ell+n-3)\a_2.
\end{align*}
Based on these computations and the fact that $0\leq n\leq \lfloor\frac{\ell-2}{2}\rfloor$, we have that the coefficient of at least one of the simple roots is negative for the Weyl group elements that are not $1$ or $s_1$. This means that all Weyl group elements, aside from $1$ and $s_1$, will not contribute to the multiplicity $m(\lambda,\mu)$. Thereby implying that $\A(\lambda,\mu)=\{1,s_1\}$ whenever $\ell\geq 1$ and $\mu=(3n+1)\w_1+(\ell-1-2n)\w_2$. In the case where $\ell=0$, then there are no $\mu$ satisfying the required condition, hence $\A(\lambda,\mu)=\emptyset$.
\end{proof}
 
Figure \ref{fig} gives the Weyl alternation diagram of $G_2$, when we take $\mu=0$ and vary $\lambda$ over the $\mathbb{Z}$-linear combinations of the fundamental weights of $G_2$. Observe that there are 60 distinct colored dots and the integral weights without a colored dot represent those weights that have empty Weyl alternation set. This illustrates the 61 distinct Weyl alternation types given in~\cite[Theorem 2.6.1]{PHThesis}. Figure \ref{fig2} displays the allowable choices of $\mu$ that result in the Weyl alternation set as stated in Theorem \ref{sets:G2}.

\begin{figure}[H]
\centering
    \includegraphics[width=0.4\textwidth]{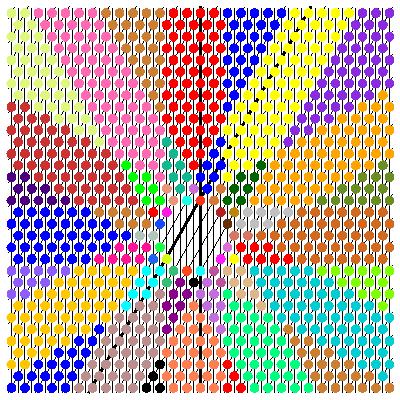}
  \caption{{The $\mu=0$ Weyl alternation diagram of $G_2$.}}\label{fig}
\end{figure}

\begin{figure}[H]
\centering
    \includegraphics[width=0.2\textwidth]{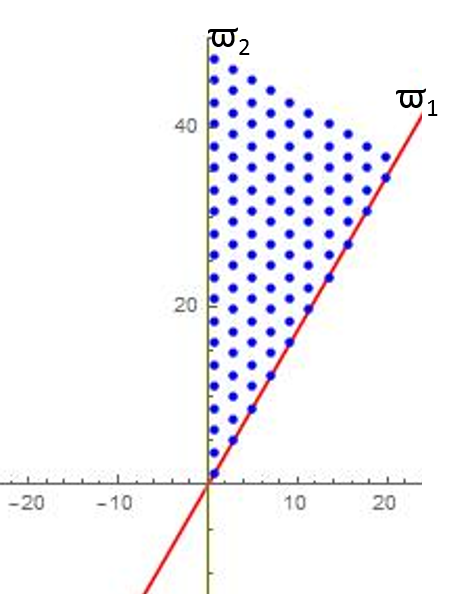}
  \caption{{Diagram highlighting weights in the Weyl alternation sets $\A(\lambda, \mu)$ when $\lambda = \ell \w_2$ and $\mu=m_1\w_1+m_2\w_2$ satisfy the conditions of Theorem \ref{sets:G2} in the Lie algebra of type $G_2$.}}\label{fig2}
\end{figure}

\

\section{A \texorpdfstring{$q-$}{q-}analog}\label{sec:qanalog}

Recall that for a weight  $\xi$, the value of the $q$-analog of Kostant's partition function is the polynomial $\wp_q(\xi)=c_0+c_1q+\cdots+c_k q^k$, where $c_j$ denotes number of ways to write $\xi$ as a nonnegative integral sum of exactly $j$ positive roots. Then the $q$-analog of Kostant's weight multiplicity formula is defined, in \cite{LL}, as:
\begin{center}
$m_q(\lambda,\mu)=\sum\limits_{\sigma\in W}^{}(-1)^{\ell(\sigma)}\wp_q(\sigma(\lambda+\rho)-(\mu+\rho))$.
\end{center}

Of course, for  multiplicity one pairs of weights evaluating $m_q(\lambda,\mu)$ at $q=1$ recovers the fact that $m(\lambda,\mu)=1$. In what follows we provide the $q$-multiplicity for the pairs of weights considered in Theorems \ref{thm:mainA}, \ref{thm:mainB}, and \ref{thm:mainG2}.

In Appendix \ref{app:tables} we provide tables for the value of $m_q(\ell\w_1,\mu)$, for multiplicity one pairs in type $A_r$ (with $1\leq r\leq 10$) and $B_r$ (with $3\leq r\leq 10$) in the cases where $1\leq \ell\leq 100$ and with all possible $\mu$ so that $m(\ell\w_1,\mu)=1$. These tables provide ample evidence in support of the following. 

\begin{repconjecture}{con:powerofq}
If $\lambda,\mu$ are a pair of weights for which $m(\lambda,\mu)=1$, then $m_q(\lambda,\mu)=q^{f(r)}$, where $f(r)$ is a function of $r$, the rank of the Lie algebra.
\end{repconjecture}
Unfortunately, the techniques we present, do not lend themselves well to provide a direct proof of the above conjecture and we would welcome a new approach to such a problem.

\subsection{Type \texorpdfstring{$A$}{A}} In this section we provide formulas for the $q$-multiplicity of the pairs of weight described in Theorem \ref{thm:mainA}. To do so we first recall that for any nonnegative integers $n,m$, 
\begin{align}
    \wp_q(n\a_1+m\a_2)=q^{max(n,m)}+q^{max(n,m)+1}+\cdots+q^{n+m}\label{eq:A2qanalog}
\end{align}
see \cite[Proposition 3.1]{HarrisLauber} for a proof. Our work will also require the following technical result.

\begin{proposition}\label{prop:q}
If $m,n,k$ are  nonnegative integers, then 
\[\wp_q(m\a_1+n\a_2+k\a_3)=q^{m+n+k}\sum_{f=0}^{\min(m,n,k)}\sum_{d=0}^{\min(m-f,n-f)}\sum_{e=0}^{\min(n-f-d,k-f)}\left(\frac{1}{q}\right)^{d+e+2f}.\]
\end{proposition}

\begin{proof}
It is known that $\wp(m\a_1+n\a_2+k\a_2)=\sum_{f=0}^{\min(m,n,k)}\sum_{d=0}^{\min(m-f,n-f)}\sum_{e=0}^{\min(n-f-d,k-f)}1$, where 
\begin{itemize}
\item $d$ accounts for the number of times the positive root $\a_1+\a_2$ appears
\item $e$ accounts for the number of times the positive root $\a_2+\a_3$ appears and 
\item $f$ accounts for the number of times the positive root $\a_1+\a_2+\a_3$ appears
\end{itemize}
in the partition of the weight $m\a_1+n\a_2+k\a_3$, \cite[Proposition 2]{HPPS2018}. Note that it is enough to count these instances as they uniquely determine the number of times we must use the positive roots $\a_1$, $\a_2$, and $\a_3$. That is we must use $m-d-f$ of the positive root $\a_1$, $n-d-e-f$ of the positive root  $\a_2$, and $k-e-f$ of the positive root $\a_3$. It suffices to keep track of the total number of positive roots used in each of these partitions. This is given by $d+e+f+(m-d-f)+(n-d-e-f)+(k-e-f)=m+n+k-d-e-2f$ from which the result follows.
\end{proof}

\begin{reptheorem}{thm:qmainA}
In type $A_r$ with $r\geq 1$ and $\ell\in\mathbb{Z}_+$:
\begin{enumerate}
    \item If $\lambda=\ell\varpi_1$, $\mu=\sum_{1\leq i\leq r}m_i\varpi_i$, where $m_i\in\mathbb{Z}_+$ and $(\ell-\sum_{1\leq i\leq r} im_i)=0$, $m_j=0$ for all $3\leq j\leq r$, then 
$m_q(\lambda,\mu)=q^{m_2}$.
\item If $\lambda=\ell\varpi_1$, $\mu=\sum_{1\leq i\leq r}m_i\varpi_i$, where $m_i\in\mathbb{Z}_+$ and $(\ell-\sum_{1\leq i\leq r} im_i)= 0$ and $m_j=0$ for all $4\leq j\leq r$ and $m_3\neq 0$, then 
$m_q(\lambda,\mu)=q^{m_2+3m_3}$.
\item If $\lambda=\ell\varpi_1$, $\mu=\sum_{1\leq i\leq r}m_i\varpi_i$, where $m_i\in\mathbb{Z}_+$ and $(\ell-\sum_{1\leq i\leq r} im_i)= 0$, $m_j=0$ for all $5\leq j\leq r$ and $m_4\neq 0$, then 
$m_q(\lambda,\mu)=q^{m_2+3m_3+6}$ when $m_4=1$ and $m_q(\lambda,\mu)=q^{m_2+3m_3+6m_4}$ when $m_4\ge 2$.
\end{enumerate}
\end{reptheorem}
\begin{proof}
To establish part (1), recall from part (2) of Theorem \ref{thm:mainA} that $\A(\lambda,\mu)=\{1\}$ and $1(\lambda+\rho)-\rho-\mu=\lambda-\mu=m_2\a_1$ (see equation \eqref{eq:1part2thmA}). Hence $m_q(\lambda,\mu)=\wp_q(m_2\a_1)=q^{m_2}.$

To establish part (2), recall from part (3) of Theorem \ref{thm:mainA} that $\A(\lambda,\mu)=\{1,s_2\}$ and $1(\lambda+\rho)-\rho-\mu=\lambda-\mu=(m_2+2m_3)\a_1+m_3\a_2$ and $s_2(\lambda,\mu)=(m_2+2m_3)\a_1+(m_3-1)\a_2$ (see equation \eqref{eq:1part3thmA}). Hence \begin{align*}
    m_q(\lambda,\mu)&=\wp_q((m_2+2m_3)\a_1+m_3\a_2)-\wp_q((m_2+2m_3)\a_1+(m_3-1)\a_2)\\
    &=(q^{m_2+2m_3}+q^{m_2+2m_3+1}+\cdots+q^{m_2+3m_3})-(q^{m_2+2m_3}+q^{m_2+2m_3+1}+\cdots+q^{m_2+3m_3-1})\\
    &=q^{m_2+3m_3}.
\end{align*}
To establish part (3), recall from part (4) of Theorem \ref{thm:mainA} that $\A(\lambda,\mu)=\{1,s_2,s_3,s_2s_3\}$ when $m_4=1$ and $\A(\lambda,\mu)=\{1,s_2,s_3,s_2s_3,s_3s_2,s_2s_3s_2\}$ when $m_4\geq 2$. From the proof of Proposition \ref{prop:m4} note that when $m_4=1$ we have
\begin{align*}
 1(\lambda+\rho)-\rho-\mu&=(m_2+2m_3+3)\a_1+(m_3+2)\a_2+\a_3\\   
 s_2(\lambda+\rho)-\rho-\mu&=(m_2+2m_3+3)\a_1+(m_3+1)\a_2+\a_3\\
 s_3(\lambda+\rho)-\rho-\mu&=(m_2+2m_3+3)\a_1+(m_3+2)\a_2\\
 s_2s_3(\lambda+\rho)-\rho-\mu&=(m_2+2m_3+3)\a_1+m_3\a_2.\\
 \end{align*}
 Hence, by Proposition \ref{prop:q} we have 
 
\begin{align*}
\wp_q((m_2+2m_3+3)\a_1+(m_3+2)\a_2+\a_3)
&=q^{m_2+3m_3+6}\sum_{f=0}^{1}\sum_{d=0}^{m_3+2-f}\sum_{e=0}^{\min(m_2+2m_3+3-f-d,1-f)}\left(\frac{1}{q}\right)^{d+e+2f}\\
 &\hspace{-.3in}=q^{m_2+3m_3+6}\left[\sum_{d=0}^{m_3+2}\sum_{e=0}^{1}\left(\frac{1}{q}\right)^{d+e}
 +\sum_{d=0}^{m_3+1}\left(\frac{1}{q}\right)^{d+2}\right]\\
&\hspace{-.3in}=q^{m_2+2m_3+3}\left[q^{m_3+3}+2q^{m_3+2}+3q^{m_3+1}+\cdots+3q+2 \right].
\end{align*}
Using the same technique as above, we find that 
\begin{align*}
\wp_q((m_2+2m_3+3)\a_1+(m_3+1)\a_2+\a_3)&=q^{m_2+2m_3+3}(q^{m_3+2}+2q^{m_3+1}+3q^{m_3}+\cdots+3q+2)
\end{align*}
while using \eqref{eq:A2qanalog} yields
\begin{align*}
\wp_q((m_2+2m_3+3)\a_1+(m_3+2)\a_2)&=q^{m_2+2m_3+3}+q^{m_2+2m_3+4}+\cdots+q^{m_2+3m_3+5}\\
\wp_q((m_2+2m_3+3)\a_1+m_3\a_2)&=q^{m_2+2m_3+3}+q^{m_2+2m_3+4}+\cdots+q^{m_2+3m_3+3}.
\end{align*}

 Therefore
 \begin{align*}
 m_q(\lambda,\mu)&=
 \wp_q((m_2+2m_3+3)\a_1+(m_3+2)\a_2+\a_3)-\wp_q((m_2+2m_3+3)\a_1+(m_3+1)\a_2+\a_3)\\
 &\qquad-\wp_q((m_2+2m_3+3)\a_1+(m_3+2)\a_2)+\wp_q((m_2+2m_3+3)\a_1+m_3\a_2).\\
 &=[q^{m_2+2m_3+3}(q^{m_3+3}+2q^{m_3+2}+3q^{m_3+1}+\cdots+3q^2+3q+2)]\\
 &\qquad-[q^{m_2+2m_3+3}(q^{m_3+2}+2q^{m_3+1}+3q^{m_3}+\cdots+3q^2+3q+2)]\\
 &\qquad-[q^{m_2+2m_3+3}+q^{m_2+2m_3+4}+\cdots+q^{m_2+3m_3+5}]\\
 &\qquad+[q^{m_2+2m_3+3}+q^{m_2+2m_3+4}+\cdots+q^{m_2+3m_3+3}].
\\
&=q^{m_2+3m_3+6}.
 \end{align*}
 
 In the case where $m_4\geq 2$ we know $\A(\lambda,\mu)=\{1,s_2,s_3,s_2s_3,s_3s_2,s_2s_3s_2\}$. From the proof of Proposition \ref{prop:m4} note that when $m_4\geq 2$ we have
\begin{align}\label{eq:id}   
 1(\lambda+\rho)-\rho-\mu&=(m_2+2m_3+3m_4)\a_1+(m_3+2m_4)\a_2+m_4\a_3\\
 s_2(\lambda+\rho)-\rho-\mu&=(m_2+2m_3+3m_4)\a_1+(m_3+2m_4-1)\a_2+m_4\a_3\nonumber\\
 s_3(\lambda+\rho)-\rho-\mu&=(m_2+2m_3+3m_4)\a_1+(m_3+2m_4)\a_2+(m_4-1)\a_3\nonumber\\
 s_2s_3(\lambda+\rho)-\rho-\mu&=(m_2+2m_3+3m_4)\a_1+(m_3+2m_4-2)\a_2+(m_4-1)\a_3\nonumber\\
 s_3s_2(\lambda+\rho)-\rho-\mu&=(m_2+2m_3+3m_4)\a_1+(m_3+2m_4-1)\a_2+(m_4-2)\a_3\nonumber\\
 s_2s_3s_2(\lambda+\rho)-\rho-\mu&=(m_2+2m_3+3m_4)\a_1+(m_3+2m_4-2)\a_2+(m_4-2)\a_3\nonumber.
\end{align}
By applying Proposition \ref{prop:q} to Equation \eqref{eq:id} we have that 
\begin{align}
    \wp_q(1(\lambda+\rho)-\rho-\mu)&=q^{m_2+3m_3+6m_4}\displaystyle\sum_{f=0}^{m_4}\sum_{d=0}^{m_3+2m_4-f}\sum_{e=0}^{\min(m_3+2m_4-f-d,m_4-f)}\fq^{d+e+2f}.\label{eq:split}
    \end{align}
We now split the sum indexed by $d$ in Equation \eqref{eq:split} to reflect when the fact that $m_3+2m_4-f-d\leq m_4-f$ if and only if $m_3+m_4\leq d$. This implies that $\wp_q(1(\lambda+\rho)-\rho-\mu)$ is equal to    
    \begin{align}
    &q^{m_2+3m_3+6m_4}\left[\displaystyle\sum_{f=0}^{m_4}\left(\sum_{d=0}^{m_3+m_4-1}\sum_{e=0}^{m_4-f}\fq^{d+e+2f}+\sum_{d=m_3+m_4}^{m3+2m4-f}\sum_{e=0}^{m3+2m4-f-d}\fq^{d+e+2f}\right)\right].\label{eq:complicated}
    \end{align}
    By applying the geometric sum formula to Equation  \eqref{eq:complicated} we find that $\wp_q(1(\lambda+\rho)-\rho-\mu)$ is equal to    
    {\footnotesize \begin{align*}\frac{q^{m_2+3m_3+6m_4}}{(q-1)^3(q+1)}\left(\left(\fq^{m_4} - q\right) q \left(\fq^{m_4} - q^2\right) - \fq^{
  m_3 + 2 m_4} (1 + q) \left(\fq^{m_4} + q (-2 + m4 (-1 + q) + q)\right)\right).
    \end{align*}}
    Using the same technique as above we find that
    \begin{align*}
       \wp_q(s_2(\lambda+\rho)-\rho-\mu)& = F(q)\left(\fq^{m_4-2} + \fq^{ m4-1} - \fq^{2 m_4} - 
   q^3\right)\\&\hspace{.3in}+F(q) \fq^{
    m_3 + 2 m_4} \left(1 + q\right) \left(\fq^{m_4} + q \left( m_4 \left(q-1\right) + q-2\right)\right)\\
    \wp_q(s_3(\lambda+\rho)-\rho-\mu)&=F(q)
 \left( \fq^{m_4}-1\right) \left(\fq^{m_4} - q\right) q^2 \\
 &\hspace{.3in}-F(q)\left(-1 + 
      m_4 \left( q-1\right) + \fq^{m_4}\right) \fq^{m_3 + 2 m_4} \left(1 + q\right)\\
      \wp_q(s_2s_3(\lambda+\rho)-\rho-\mu)&=F(q) \left( \fq^{m_4}-1\right) \left(\fq^{m_4} - q\right) \\
      &\hspace{.3in}-F(q) \left(-1 + m_4 ( q-1) + \fq^
      {m_4}\right) \fq^{m_3 + 2 m_4} (1 + q)\\
      \wp_q(s_3s_2(\lambda+\rho)-\rho-\mu)&=
      F(q) \left( \fq^{ m_4-1}-1\right) \left(\fq^{m_4}-1\right) q^2\\ &\hspace{.3in}-F(q) \fq^{
    m_3 + 2 m_4} (1 + q) \left(m_4 (q-1 ) + \left(\fq^{m_4}-1\right) q\right)\\
    \wp_q(s_2s_3s_2(\lambda+\rho)-\rho-\mu)&=F(q)  \left(\fq^{ m_4-1}-1\right) \left(\fq^{m_4}-1\right) q \\
    &\hspace{.3in}-F(q) \fq^{
    m_3 + 2 m_4} (1 + q) \left(m_4 ( q-1) + \left(\fq^{m_4}-1\right) q\right)
    \end{align*}
    where
    \begin{align*}
        F(q):=-\frac{q^{m_2 + 3 m_3 + 
  6 m_4}}{\left(q-1\right)^3\left(q+1\right)}.
    \end{align*}
    By taking the sum of the above terms, with their sign corresponding to the length of the associated Weyl group element, we find that 
    \[m_q(\lambda,\mu)=q^{m_2 + 3 m_3 + 6 m_4}.\qedhere\]
\end{proof}
\subsection{Type \texorpdfstring{$B$}{B}} In this section we provide formulas for the $q$-multiplicity of the pairs of weight described in Theorem \ref{thm:mainB}.
We begin by proving the following technical results.

\begin{lemma}\label{lem1:qversiontypeB}
If $1\leq i<j\leq r$, then 
\[\wp_q(\a_i+\a_{i+1}+\cdots+\a_j)=q(1+q)^{j-i}.\]
\end{lemma}
\begin{proof}
The only positive roots that we can use are of the form $\a_{i'}+\cdots+\a_{j'}$ where $i\leq i'\leq j'\leq j$. We now proceed by induction on $j-i$. If $j-i=0$, then we are partitioning $\a_i$, which only has one partition with a single positive root. Thus $\wp_q(\a_i)=q=q(1+q)^0$. If $j-i=1$, then we are partitioning $\a_i+\a_{i+1}$. The possible partitions are the one consisting of only the single simple roots $\a_i$ and $\a_{i+1}$, and the partition using the positive root $\a_i+\a_{i+1}$, which contributes a $q^2$ and $q$ to $\wp_q(\a_{i}+\a_{i+1})$, respectively. Thus $\wp_q(\a_{i}+\a_{i+1})=q^2+q=q(1+q)^{1}=q(1+q)^{2-1}$. Assume that the result holds for $j-i\leq k-1$. Now we count the the partitions of $\a_i+\a_{i+1}+\cdots+\a_{j}+\a_{j+1}$ by considering  which positive root contains $\a_{j+1}$. These positive roots are of the form $\a_\ell+\cdots+\a_{j+1}$, with $i\leq \ell\leq j+1$. Thus 
\begin{align*}
    \wp_q(\a_i+\a_{i+1}+\cdots+\a_{j}+\a_{j+1})&=q\sum_{\ell=i}^{j+1}\wp_q(\a_i+\a_{i+1}+\cdots+\a_{j}+\a_{j+1}-(\a_\ell+\a_{\ell+1}+\cdots+\a_{j+1}))\\
    &=q\left(1+\sum_{\ell=i+1}^{j+1}\wp_q(\a_i+\a_{i+1}+\cdots+\a_{\ell-1})\right).
\end{align*}
By the induction hypothesis, this yields 
\begin{align*}
    \wp_q(\a_i+\a_{i+1}+\cdots+\a_{j}+\a_{j+1})&=q\left(1+\sum_{\ell=i+1}^{j+1}q(1+q)^{\ell-1-i}\right)\\
    &=q\left(1+q\left(\frac{1-(1+q)^{j-i+1}}{1-(1+q)}\right)\right)\\
    &=q(1+q)^{j+1-i},
\end{align*}
as claimed.
\end{proof}

\begin{lemma}\label{lem:qversiontypeB}
If $\ell\geq 3$ and $x,y\in\mathbb{Z}_{>0}$, then 
\[\wp_q(x\a_1+y\a_2+\a_3+\a_4+\cdots+\a_\ell)=q(1+q)^{\ell-3}[\langle x-1,y-1\rangle+\langle x,y-1\rangle+\langle x,y\rangle],\]
where $\langle m,n\rangle:=\wp_q(m\a_1+n\a_2)$ for all nonnegative integers $m$ and $n$.
\end{lemma}

\begin{proof}
We will account for all partitions of $x\a_1+y\a_2+\a_3+\a_4+\cdots+\a_\ell$ by considering which positive root contains $\a_\ell$. These roots are of the form $\a_i+\cdots+\a_\ell$ where $1\leq i\leq \ell$. Thus 
\begin{align}
    \wp_q(x\a_1+y\a_2+\a_3+\a_4+\cdots+\a_\ell)&=\sum_{i=1}^{\ell}q\wp_q(x\a_1+y\a_2+\a_3+\a_4+\cdots+\a_\ell-(\a_i+\cdots+\a_\ell)),\label{eq:qrecursion}
\end{align}
where the factor of $q$ is accounting for the use of the positive root $\a_i+\cdots+\a_\ell$.  We now proceed by induction on $\ell$. Notice that when $\ell=3$, by Equation \eqref{eq:qrecursion} we know 
\begin{align*}
    \wp_q(x\a_1+y\a_2+\a_3)&=q[\wp_q((x-1)\a_1+(y-1)\a_2)+\wp_q(x\a_1+(y-1)\a_2)+\wp_q(x\a_1+y\a_2)\\
    &=q(1+q)^{3-3}[\langle x-1,y-1\rangle+\langle x,y-1\rangle+\langle x,y\rangle].
\end{align*}  
When $\ell=4$, by Equation \eqref{eq:qrecursion} we know 
\begin{align*}
    \wp_q(x\a_1+y\a_2+\a_3+\a_4)&=q[\wp_q((x-1)\a_1+(y-1)\a_2)+\wp_q(x\a_1+(y-1)\a_2)+\wp_q(x\a_1+y\a_2)]\\
    &\hspace{.5in}+q\wp_q(x\a_1+y\a_2+\a_3)\\
    &=q(1+q)^{4-3}[\langle x-1,y-1\rangle+\langle x,y-1\rangle+\langle x,y\rangle].
\end{align*}
Now assume that for any $3\leq k\leq \ell-1$ the result holds. Notice that when $k=\ell$, using the recurrence in Equation \eqref{eq:qrecursion} and the induction hypothesis we have 
\begin{align*}
    \wp_q\left(x\a_1+y\a_2+\sum_{i=3}^{\ell}\a_i\right)&=q\sum_{i=1}^{\ell}\wp_q(x\a_1+y\a_2+\a_3+\a_4+\cdots+\a_\ell-(\a_i+\cdots+\a_\ell))\\
    &=q[\wp_q((x-1)\a_1+(y-1)\a_2)+\wp_q(x\a_1+(y-1)\a_2)+\wp_q(x\a_1+y\a_2)]\\
    &\qquad+q\sum_{i=4}^{\ell}\wp_q(x\a_1+y\a_2+\a_3+\a_4+\cdots+\a_{i-1})\\
    &=q[\langle x-1,y-1\rangle+\langle x,y-1\rangle+\langle x,y\rangle]\\
    &\qquad+q\sum_{i=4}^{\ell}q(1+q)^{(i-1)-3}[\langle x-1,y-1\rangle+\langle x,y-1\rangle+\langle x,y\rangle]\\
    &=q\left(1+q\sum_{i=0}^{\ell-4}(1+q)^{i}\right)[\langle x-1,y-1\rangle+\langle x,y-1\rangle+\langle x,y\rangle]\\
    &=q\left(1+q\left(\frac{1-(1+q)^{\ell-3}}{1-(1-q)}\right)\right)[\langle x-1,y-1\rangle+\langle x,y-1\rangle+\langle x,y\rangle]\\
    &=q(1-(1-(1+q)^{\ell-3})[\langle x-1,y-1\rangle+\langle x,y-1\rangle+\langle x,y\rangle]\\
    &=q(1+q)^{\ell-3}[\langle x-1,y-1\rangle+\langle x,y-1\rangle+\langle x,y\rangle],
\end{align*}
as claimed.

\end{proof}
\begin{reptheorem}{thm:qmainB}
In type $B_r$  with $r\geq 3$:
\begin{enumerate}
    \item If  $\ell$ is a positive odd integer and $k\in \ZZ_{\geq 0}$, then   $m_q(\ell\w_1,(\ell-1-4k)\w_1+2k\w_2)=q^{2k+r}$.
    \item If $\ell, r \in 2\ZZ_{\ge 0}$ under the condition $\ell-1=\sum_{1\le i< r} im_i + \frac{rm_r}{2}$ with $m_i\in 2\ZZ_{\ge 0}$, then $m_q(\ell\w_1,\mu)=0$.
    \item  If $j,k\in \ZZ_{\geq 0}$ and $k\neq 0$, then $m_q(\ell\w_1,(\ell-1-4j-6k)\w_1+2j\w_2+2k\w_3)=q^{r+2j+6k-2}$.
\end{enumerate} 
\end{reptheorem}
\begin{proof}
Part (1):  By Lemma \ref{lem:Bnonconsprod} $\sigma\in\A(\ell\w_1,(\ell-1-4k)\w_1+2k\w_2)$ with $k\in \ZZ_{\geq 0}$ if and only if $\sigma=s_{i_1}s_{i_2}\cdots s_{i_m}$ where  $2\leq i_1,i_2,\ldots,i_m\leq r$ are nonconsecutive integers. Without loss of generality, we assume that $2\leq i_1<i_2<\ldots<i_m\leq r$. We consider the cases where $s_r$ is a factor in $\sigma$ and when it is not.\\

\noindent{\bf Case 1.} Suppose $\sigma=s_{i_1}\cdots s_{i_m}s_r$, where $2\leq i_1<i_2<\cdots<i_m\leq r-2$ are nonconsecutive integers. Then $\ell(\sigma)=1+m$ where $m\geq 0$. Notice
\begin{align}
    \sigma(\lambda+\rho)-\rho-\mu=
    (2k+1)\a_1+\a_2+\a_3+\dots+\a_r-\sum_{j=1}^m \a_{i_j}-\a_r.
    \end{align}
Let $\tau=\sigma(\lambda+\rho)-\rho-\mu=\tau_0+\tau_1+\cdots+\tau_{m-1}+\tau_m$, where
\begin{align*}
    \tau_0&=(2k+1)\a_1+\a_2+\dots+\a_{i_1-1}\\
    \tau_1&=\a_{i_1+1}+\a_{i_1+2}+\dots+\a_{i_2-1}\\
    \tau_2&=\a_{i_2+1}+\a_{i_1+2}+\dots+\a_{i_3-1}\\
    &\;\;\vdots\\
    \tau_{m-1}&=\a_{i_{m-1}+1}+\a_{i_1+2}+\dots+\a_{i_m-1}\\
    \tau_{m}&=\a_{i_{m}+1}+\a_{i_1+2}+\dots+\a_{r-1}.
\end{align*}
Partitioning $\tau$ as a sum of positive roots reduces to partitioning each of the weights $\tau_1,\ldots,\tau_m$. Since the sets of positive roots used to partition  $\tau_i$ and $\tau_j$ are disjoint for all $i\neq j$, we know that
\[\wp_q(\tau)=\wp_q(\tau_0)\wp_q(\tau_1)\cdots\wp_q(\tau_m).\]
Then by Lemma \ref{lem:qversiontypeB} we have that 
\begin{align*}
    \wp_q(\tau_0)&=q(1+q)^{(i_1-1)-3}[\wp_q(2k\a_1)+\wp_q((2k+1)\a_1)+\wp_q((2k+1)\a_1+\a_2)]\\
    &=q(1+q)^{i_1-4}\left(q^{2k}(1+q)^2\right).
\end{align*}
By Lemma \ref{lem1:qversiontypeB} we have that 
\begin{align*}
    \wp_q(\tau_1)\wp_q(\tau_2)\cdots\wp_q(\tau_{m-1})&=\displaystyle\prod_{j=1}^{m-1} q(1+q)^{(i_{j+1}-1)-(i_{j}+1)}=q^{m-1}(1+q)^{i_m-i_1-2(m-1)}
    \end{align*}
    and
    \begin{align*}
    \wp_q(\tau_m)&=q(1+q)^{r-1-(i_m+1)}.
\end{align*}
Thus 
\begin{align*}
    \wp_q(\tau)&=\left(q(1+q)^{(i_1-1)-3}\left(q^{2k}(1+q)^2\right)\right)\left(q^{m-1}(1+q)^{i_m-i_1-2(m-1)}\right)\left(q(1+q)^{r-1-(i_m+1)}\right)\\
    &=q^{2k}q^{1+m}(1+q)^{r-2-2m}.
\end{align*}
Now notice that there exist $\binom{r-2-m}{m}$ elements $\sigma\in\A(\lambda,\mu)$  such that $\sigma$ contains $s_r$ and $\ell(\sigma)=1+m\geq 1$, while  $\max\{\ell(\sigma):\sigma\in\A(\lambda,\mu)\mbox{ containing $s_r$}\}=\lfloor\frac{r-2}{2}\rfloor$. So we may compute that \begin{align}
    \sum_{\substack{\sigma\in\A(\lambda,\mu)\\\mbox{containing $s_r$}}}(-1)^{\ell(\sigma)}\wp_q(\sigma(\lambda+\rho)-\rho-\mu)&=q^{2k}\sum_{m=0}^{\lfloor\frac{r-2}{2}\rfloor}(-1)^{1+m}\binom{r-2-m}{m}q^{1+m}(1+q)^{r-2-2m}.
    \label{eq:actualcase1}
\end{align}
Proposition 3.3 in \cite{Harris} established that if $r\geq 0$, then 
\begin{align}
    \sum_{m=0}^{\lfloor\frac{r}{2}\rfloor}(-1)^m\binom{r-m}{m}q^{1+m}(1+q)^{r-2m}=\sum_{i=1}^{r+1}q^i.\label{prop:3.3}
\end{align}
Applying this result to Equation \eqref{eq:actualcase1} yields
\begin{align}
  -q^{2k}\sum_{m=0}^{\lfloor\frac{r-2}{2}\rfloor}(-1)^{m}\binom{r-2-m}{m}q^{1+m}(1+q)^{r-2-2m}=-q^{2k}\sum_{i=1}^{r-1}q^i.\label{eq:hassr}
\end{align}\\

\noindent{\bf Case 2.} Suppose $\sigma=s_{i_1}\cdots s_{i_m}$, where $2\leq i_1,i_2,\ldots,i_m\leq r-1$ are nonconsecutive integers. Then $\ell(\sigma)=m$ where $m\geq 0$. Notice
\begin{align}
    \sigma(\lambda+\rho)-\rho-\mu=
    (2k+1)\a_1+\a_2+\a_3+\dots+\a_r-\sum_{j=1}^m \a_{i_j}.
    \end{align}
Let $\tau=\sigma(\lambda+\rho)-\rho-\mu=\tau_0+\tau_1+\cdots+\tau_{m-1}+\tau_m$, where
\begin{align*}
    \tau_0&=(2k+1)\a_1+\a_2+\dots+\a_{i_1-1}\\
    \tau_1&=\a_{i_1+1}+\a_{i_1+2}+\dots+\a_{i_2-1}\\
    \tau_2&=\a_{i_2+1}+\a_{i_1+2}+\dots+\a_{i_3-1}\\
    &\;\;\vdots\\
    \tau_{m-1}&=\a_{i_{m-1}+1}+\a_{i_1+2}+\dots+\a_{i_m-1}\\
    \tau_{m}&=\a_{i_{m}+1}+\a_{i_1+2}+\dots+\a_{r}.
\end{align*}
As in the previous case, partitioning $\tau$ as a sum of positive roots reduces to partitioning each of the weights $\tau_0,\tau_1,\ldots,\tau_m$. Since the sets of positive roots used to partition  $\tau_i$ and $\tau_j$ are disjoint for all $i\neq j$, we know that
\[\wp_q(\tau)=\wp_q(\tau_0)\wp_q(\tau_1)\cdots\wp_q(\tau_m).\]
Then by Lemma \ref{lem:qversiontypeB} we have that 
\begin{align*}
    \wp_q(\tau_0)&=q(1+q)^{(i_1-1)-3}[\wp_q(2k\a_1)+\wp_q((2k+1)\a_1)+\wp_q((2k+1)\a_1+\a_2)]\\
    &=q(1+q)^{i_1-4}\left(q^{2k}(1+q)^2\right).
\end{align*}
By Lemma \ref{lem1:qversiontypeB} we have that 
\begin{align*}
    \wp_q(\tau_1)\wp_q(\tau_2)\cdots\wp_q(\tau_{m-1})&=\displaystyle\prod_{j=1}^{m-1} q(1+q)^{(i_{j+1}-1)-(i_{j}+1)}=q^{m-1}(1+q)^{i_m-i_1-2(m-1)}
    \end{align*}
    and
    \begin{align*}
    \wp_q(\tau_m)&=q(1+q)^{{r}-(i_m+1)}.
\end{align*}
Thus 
\begin{align*}
    \wp_q(\tau)&=\left(q(1+q)^{i_1-4}\left(q^{2k}(1+q)^2\right)\right)\left(q^{m-1}(1+q)^{i_m-i_1-2(m-1)}\right)\left(q(1+q)^{r-(i_m+1)}\right)\\
    &=q^{2k}q^{1+m}(1+q)^{r-1-2m}.
\end{align*}
Now notice that there exist $\binom{r-1-m}{m}$ elements $\sigma\in\A(\lambda,\mu)$  such that $\sigma$ does not contains $s_r$ and $\ell(\sigma)=m\geq 0$, while  $\max\{\ell(\sigma):\sigma\in\A(\lambda,\mu)\mbox{ not containing $s_r$}\}=\lfloor\frac{r-1}{2}\rfloor$. So we may compute that \begin{align}
    \sum_{\substack{\sigma\in\A(\lambda,\mu)\\\mbox{not containing $s_r$}}}(-1)^{\ell(\sigma)}\wp_q(\sigma(\lambda+\rho)-\rho-\mu)&=q^{2k}\sum_{m=0}^{\lfloor\frac{r-1}{2}\rfloor}(-1)^{1+m}\binom{r-1-m}{m}q^{1+m}(1+q)^{r-1-2m}.
    \label{eq:actualcase2}
\end{align}
Applying \eqref{prop:3.3} to Equation \eqref{eq:actualcase1} yields
\begin{align}
   q^{2k}\sum_{m=0}^{\lfloor\frac{r-1}{2}\rfloor}(-1)^{m}\binom{r-1-m}{m}q^{1+m}(1+q)^{r-1-2m}=q^{2k}\sum_{i=1}^{r}q^i.\label{eq:hasnosr}
\end{align}

The result then follows from \eqref{eq:hassr} and \eqref{eq:hasnosr} since
\begin{align*}
    m_{q}(\lambda,\mu)&=\sum_{\substack{\sigma\in\A(\lambda,\mu)\\\mbox{not containing $s_r$}}}\hspace{-.25in}(-1)^{\ell(\sigma)}\wp_q(\sigma(\lambda+\rho)-\rho-\mu)\;\;+\hspace{-.2in}\sum_{\substack{\sigma\in\A(\lambda,\mu)\\\mbox{containing $s_r$}}}\hspace{-.2in}(-1)^{\ell(\sigma)}\wp_q(\sigma(\lambda+\rho)-\rho-\mu)\\
    &=q^{2k}\sum_{i=1}^{r}q^i-q^{2k}\sum_{i=1}^{r-1}q^i\\
    &=q^{2k+r}.
\end{align*}

\noindent
Part (2) follows directly from the fact that in this case $\A(\ell\w_1,\mu)=\emptyset$. \\

\noindent Part (3). In this case $\sigma\in\A(\lambda,\mu)$ is of the form $\sigma=s_{i_1}s_{i_2}\cdots s_{i_m}$ with $2\leq i_1,i_2,\ldots, i_m\leq r$ are nonconsecutive integers or $\sigma=s_{i_1}s_{i_2}\cdots s_{i_m}s_2s_3$ with $5\leq i_1,i_2,\ldots, i_m\leq r$ are nonconsecutive integers. We consider cases depending on whether $\sigma\in\A(\lambda,\mu)$ contains $s_r$ or not, and whether it contains one of $s_2, s_3$, or $s_2s_3$, or contains none of them. Then using the same techniques as we did in Part (1) of this proof we can compute the value of the $q$-multiplicity by taking the sum of the terms given in Table \ref{tab:typeB}, where 
\begin{align*}
    f(q)&=\wp_q((2j+4k)\a_1+(2k-1)\a_2)+\wp_q((2j+4k+1)\a_1+(2k-1)\a_2)\\&\hspace{1in}+\wp_q((2j+4k+1)\a_1+2k\a_2)\\
    g(q)&=q^{2j+4k+1}\displaystyle\sum_{i=1}^{2k+1}q^i\\
    p(q)&=q^{2j+4k+1}\displaystyle\sum_{i=1}^{2k-1}q^i\\
    h(q)&=\wp_q((2j+4k)\a_1+2k\a_2)+\wp_q((2j+4k+1)\a_1+2k\a_2)\\&\hspace{1in}+\wp_q((2j+4k+1)\a_1+(2k+1)\a_2).
\end{align*}
Thus the $q$-multiplicity can be computed by taking the sum of the terms in the Table \ref{tab:typeB} which establishes
\begin{align*}
m_q(\lambda,\mu)&=\displaystyle\sum_{\sigma\in\A(\lambda,\mu)}(-1)^{\ell(\sigma)}\wp_q(\sigma(\lambda+\rho)-\rho-\mu)\\
&=q^{r-3}\left[ q(h(q)-f(q))+(p(q)-g(q))\right]\\
&=q^{r-3}\left[ q\left(q^{2j+4k}(q^{2k})+q^{2j+4k}(q^{2k+1}+q^{2k})\right)-q^{2j+4k+1}\left(q^{2k}+q^{2k+1}\right)\right]\\
&=q^{r-3}\left[ q^{2j+6k+1}\right]\\
&=q^{r-2+2j+6k}.
\end{align*}
\begin{table}[h!]
    \centering
    \begin{tabular}{|l|c|c|}\hline
Value of $\displaystyle\sum_{\sigma\in\A(\lambda,\mu)}(-1)^{\ell(\sigma)}\wp_q(\sigma(\lambda+\rho)-\rho-\mu)$ when:         & $\sigma $ contains $s_r$ & $\sigma$ does not contain $s_r$\\\hline
     $\sigma$ contains $s_2$    & $f(q)\displaystyle\sum_{i=1}^{r-3}q^i$ &$-f(q)\displaystyle\sum_{i=1}^{r-2}q^i$\\\hline
     $\sigma$     contains $s_3$& $g(q)\displaystyle\sum_{i=1}^{r-4}q^i$ &$-g(q)\displaystyle\sum_{i=1}^{r-3}q^i$\\\hline
         $\sigma$ contains $s_2s_3$&  $-p(q)\displaystyle\sum_{i=1}^{r-4}q^i$&$p(q)\displaystyle\sum_{i=1}^{r-3}q^i$\\\hline
         $\sigma$ does not contain $s_2$, $s_3$, or $s_2s_3$&$-h(q)\displaystyle\sum_{i=1}^{r-3}q^i$&$h(q)\displaystyle\sum_{i=1}^{r-2}q^i$\\\hline
    \end{tabular}
    \caption{Computing parts of the $q$-multiplicity on subsets of $\A(\lambda,\mu)$.}
    \label{tab:typeB}
\end{table}
\end{proof}

\subsection{Type \texorpdfstring{$G_2$}{G2}} In this section we provide a formula for the $q$-multiplicity of the pairs of weight described in Theorem \ref{thm:mainG2}.
\begin{reptheorem}{thm:qmainG2}
Let $\lambda=\ell\w_2$ with $\ell\geq 1$ be a weight of $G_2$. If $\mu=m_1\varpi_1+m_2\varpi_2$, where $m_1,m_2\in\mathbb{Z}_{\geq 0}$ and $3\ell-1=2m_1+3m_2$, then 
$m_q(\lambda,\mu)=q^{n+2}$ where $m_1=3n+1$.    
\end{reptheorem}

\begin{proof}
Since the positive roots of $G_2$ are 
$\Phi^+=\{\a_1,\a_2,\a_1+\a_2,2\a_1+\a_2,3\a_1+\a_2,3\a_1+2\a_2\}$
and by Proposition \ref{sets:G2} we have that if $\lambda=\ell\w_2$ with $\ell\geq 1$ and $\mu=m_1\varpi_1+m_2\varpi_2$, where $m_1,m_2\in\mathbb{Z}_{\geq 0}$ and $3\ell-1=2m_1+3m_2$, then
\begin{align*}
    m_q(\lambda,\mu)&=(-1)^0\wp_q(1(\lambda+\rho)-\rho-\mu)+(-1)^1\wp_q(s_1(\lambda+\rho)-\rho-\mu)\\
    &=\wp_{q}(\a_1+(n+1)\a_2)-\wp_{q}((n+1)\a_2)\\
    &=(q^{n+2}+q^{n+1})-q^{n+1}\\
    &=q^{n+2}
\end{align*}
as claimed.
\end{proof}

\section{Future Work}\label{sec:future}
We provide a few directions for future study. The first direction would be to provide a proof of Conjecture \ref{con:powerofq}. Note that in Appendix \ref{app:tables} we provide ample evidence in support of this conjecture, yet the complication remaining is that the Weyl alternation sets are very complicated to compute for all remaining multiplicity one pairs of weights given in \cite{BZ}. 

Another interesting direction for research is to 
study the \emph{Weyl alternation poset}. This is the poset arising from the set containment of the Weyl alternation sets $\A(\lambda,\mu)$ for all integral weights of a Lie algebra. By fixing $\mu=0$ and varying $\lambda$ in the integral weight lattice, Figures \ref{fig:A2HasseDiagram}, \ref{fig:B2HasseDiagram}, and \ref{fig:G2HasseDiagram} illustrate the Weyl alternation posets of the Lie algebras of type $A_2$,  $B_2$ and $G_2$, respectively.
\begin{figure}[H]
    \centering
    \resizebox{4in}{!}{
\begin{tikzpicture}
      \node[](a0) at (0,0) {$\emptyset$};
      \node[right](a1) at (1.5*1,1.5*0){$\{\sigma_0\}$};
      \node[above](a2) at (1.5*0.5, 1.5*0.87){$\{\sigma_{1}\}$};
      \node[above](a3) at (1.5*-.5,1.5*.87){$\{\sigma_{2}\}$};
      \node[left](a4) at (1.5*-1,1.5*0){$\{\sigma_{3}\}$};
      \node[below](a5) at (1.5*-.5, 1.5*-.87){$\{\sigma_{4}\}$};
      \node[below](a6) at (1.5*.5,1.5*-.87){$\{\sigma_{5}\}$};
      \draw(a0)--(a1);
      \draw(a0)--(a2);
      \draw(a0)--(a3);
      \draw(a0)--(a4);
      \draw(a0)--(a5);
      \draw(a0)--(a6);
      \node[right](b1) at (3*.9,3*.5){$\{\sigma_0,\;\sigma_1\}$};
      \node[above](b2) at (3*0, 3*1){$\{\sigma_{1},\;\sigma_2\}$};
      \node[above](b3) at (3*-.9,3*.5){$\{\sigma_{2},\;\sigma_3\}$};
      \node[left](b4) at (3*-.9,3*-.5){$\{\sigma_{3},\;\sigma_4\}$};
      \node[below](b5) at (3*0, 3*-1){$\{\sigma_{4},\;\sigma_5\}$};
      \node[below](b6) at (3*.9,3*-.5){$\{\sigma_{5},\;\sigma_0\}$};
      \draw(b1)--(a1);
      \draw(b1)--(a2);
      \draw(b2)--(a2);
      \draw(b2)--(a3);
      \draw(b3)--(a3);
      \draw(b3)--(a4);
      \draw(b4)--(a4);
      \draw(b4)--(a5);
      \draw(b5)--(a5);
      \draw(b5)--(a6);
      \draw(b6)--(a6);
      \draw(b6)--(a1);
\node[right] at (5.5,3){{\bf{Key}}};
\node[right] at (5.5,2){$\sigma_{0}:=1$};
\node[right] at (5.5,1){$\sigma_{1}:=s_2$};
\node[right] at (5.5,0){$\sigma_{2}:=s_1s_2$};
\node[right] at (5.5,-1){$\sigma_{3}:=s_2s_1s_2$};
\node[right] at (5.5,-2){$\sigma_{4}:=(s_1s_2)^2$};
\node[right] at (5.5,-3){$\sigma_{5}:=s_1s_2s_1$};
\draw(5,-3.5)--(8,-3.5)--(8,3.5)--(5,3.5)--(5,-3.5);
   \end{tikzpicture}
    }
    \caption{Hasse diagram of the Weyl alternation sets $\A(\lambda,0)$ of Lie algebra of type $A_2$.}
    \label{fig:A2HasseDiagram}
\end{figure}
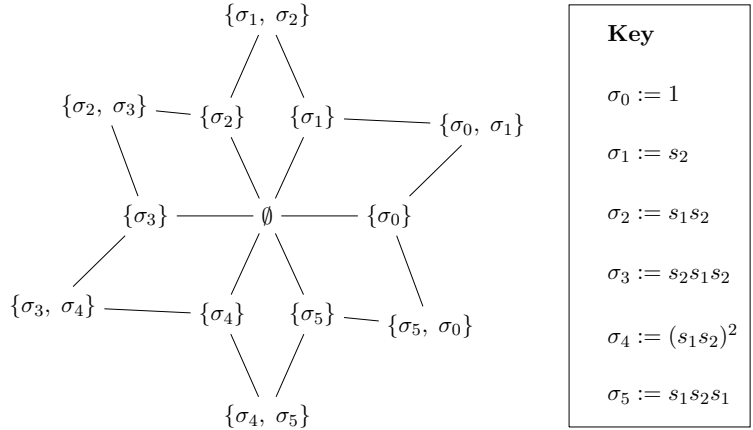

\begin{figure}[H]
    \centering
    \resizebox{5in}{!}{
\begin{tikzpicture}
      \node[](a0) at (0,0) {$\emptyset$};
      \node[right](a1) at (1.5*1,1.5*0){$\{\sigma_0\}$};
      \node[right](a2) at (1.5*0.71, 1.5*0.71){$\{\sigma_{1}\}$};
      \node[above](a3) at (1.5*0.0,1.5*1){$\{\sigma_{2}\}$};
      \node[left](a4) at (1.5*-.71,1.5*.71){$\{\sigma_{3}\}$};
      \node[left](a5) at (1.5*-1, 1.5*0){$\{\sigma_{4}\}$};
      \node[left](a6) at (1.5*-.71,1.5*-.71){$\{\sigma_{5}\}$};
      \node[below](a7) at (1.5*0.0,1.5*-1){$\{\sigma_{6}\}$};
      \node[right](a8) at (1.5*.71, 1.5*-.71){$\{\sigma_{7}\}$};
      \draw(a0)--(a1);
      \draw(a0)--(a2);
      \draw(a0)--(a3);
      \draw(a0)--(a4);
      \draw(a0)--(a5);
      \draw(a0)--(a6);
      \draw(a0)--(a7);
      \draw(a0)--(a8);
      \node[right](b1) at (3*.92,3*.38){$\{\sigma_0,\;\sigma_1\}$};
      \node[above](b2) at (3*0.4, 3*0.92){$\{\sigma_{1},\;\sigma_2\}$};
      \node[above](b3) at (3*-.4,3*.92){$\{\sigma_{2},\;\sigma_3\}$};
      \node[left](b4) at (3*-.92,3*.38){$\{\sigma_{3},\;\sigma_4\}$};
      \node[left](b5) at (3*-.92, 3*-.38){$\{\sigma_{4},\;\sigma_5\}$};
      \node[below](b6) at (3*-.4,3*-.92){$\{\sigma_{5},\;\sigma_6\}$};
      \node[below](b7) at (3*.4,3*-.92){$\{\sigma_{6},\;\sigma_7\}$};
      \node[right](b8) at (3*.92, 3*-.38){$\{\sigma_{7},\sigma_0\}$};
      \draw(b1)--(a1);
      \draw(b1)--(a2);
      \draw(b2)--(a2);
      \draw(b2)--(a3);
      \draw(b3)--(a3);
      \draw(b3)--(a4);
      \draw(b4)--(a4);
      \draw(b4)--(a5);
      \draw(b5)--(a5);
      \draw(b5)--(a6);
      \draw(b6)--(a6);
      \draw(b6)--(a7);
      \draw(b7)--(a7);
      \draw(b7)--(a8);
      \draw(b8)--(a8);
      \draw(b8)--(a1);
      \node[right](c1) at (4.5*1,4.5*0){$\{\sigma_0,\;\sigma_1,\;\sigma_7\}$};
      \node[right](c2) at (4.5*0.71, 4.5*0.71){$\{\sigma_0,\;\sigma_1,\;\sigma_2\}$};
      \node[above](c3) at (4.5*0.0,4.5*1){$\{\sigma_1,\;\sigma_{2},\;\sigma_3\}$};
      \node[left](c4) at (4.5*-.71,4.5*.71){$\{\sigma_{2},\;\sigma_3,\;\sigma_4\}$};
      \node[left](c5) at (4.5*-1, 4.5*0){$\{\sigma_{3},\;\sigma_4,\;\sigma_5\}$};
      \node[left](c6) at (4.5*-.71,4.5*-.71){$\{\sigma_{4},\;\sigma_5,\;\sigma_6\}$};
      \node[below](c7) at (4.5*0.0,4.5*-1){$\{\sigma_{5},\;\sigma_6,\;\sigma_7\}$};
      \node[right](c8) at (4.5*.71, 4.5*-.71){$\{\sigma_0,\;\sigma_{6},\;\sigma_7\}$};
      \draw(c2)--(b1);
      \draw(c2)--(b2);
      \draw(c3)--(b2);
      \draw(c3)--(b3);
      \draw(c4)--(b3);
      \draw(c4)--(b4);
      \draw(c5)--(b4);
      \draw(c5)--(b5);
      \draw(c6)--(b5);
      \draw(c6)--(b6);
      \draw(c7)--(b6);
      \draw(c7)--(b7);
      \draw(c8)--(b7);
      \draw(c8)--(b8);
      \draw(c1)--(b8);
      \draw(c1)--(b1);      
\node[right] at (9,4){{\bf{Key}}};
\node[right] at (9,3){$\sigma_{0}:=1$};
\node[right] at (9,2){$\sigma_{1}:=s_2$};
\node[right] at (9,1){$\sigma_{2}:=s_1s_2$};
\node[right] at (9,0){$\sigma_{3}:=s_2s_1s_2$};
\node[right] at (9,-1){$\sigma_{4}:=(s_1s_2)^2$};
\node[right] at (9,-2){$\sigma_{5}:=s_1s_2s_1$};
\node[right] at (9,-3){$\sigma_{6}:=s_2s_1$};
\node[right] at (9,-4){$\sigma_{7}:=s_1$};
\draw(8.5,-4.5)--(12,-4.5)--(12,4.5)--(8.5,4.5)--(8.5,-4.5);
   \end{tikzpicture}
    }
    \caption{Hasse diagram of the Weyl alternation sets $\A(\lambda,0)$ of Lie algebra of type $B_2$.}
    \label{fig:B2HasseDiagram}
\end{figure}
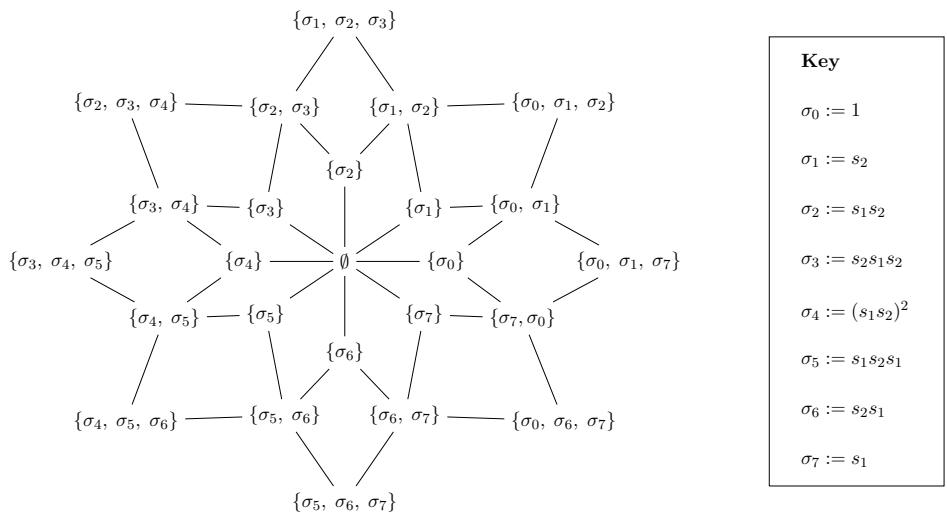

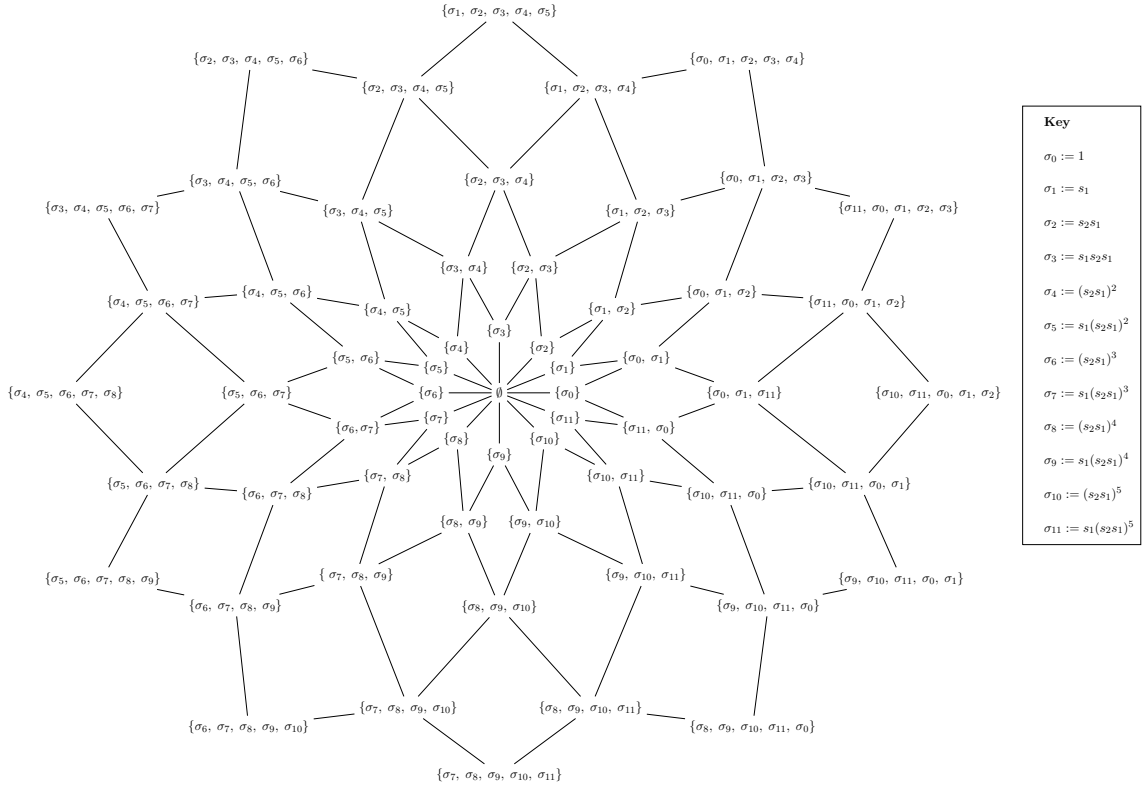
\begin{figure}[H]
\centering
  \resizebox{6in}{!}{
\begin{tikzpicture}
      \node[](a0) at (0,0) {$\emptyset$};
      \node[right](a1) at (1.5*1,1.5*0){$\{\sigma_0\}$};
      \node[right](a2) at (1.5*0.9, 1.5*0.5){$\{\sigma_{1}\}$};
      \node[right](a3) at (1.5*0.5,1.5*0.9){$\{\sigma_{2}\}$};
      \node[above](a4) at (1.5*0,1.5*1){$\{\sigma_{3}\}$};
      \node[left](a5) at (1.5*-0.5, 1.5*0.9){$\{\sigma_{4}\}$};
      \node[left](a6) at (1.5*-0.9,1.5* 0.5){$\{\sigma_{5}\}$};
      \node[left](a7) at (1.5*-1,1.5*0){$\{\sigma_{6}\}$};
      \node[left](a8) at (1.5*-0.9, 1.5*-0.5){$\{\sigma_{7}\}$};
      \node[left](a9) at (1.5*-0.5, 1.5*-0.9){$\{\sigma_{8}\}$};
      \node[below](a10) at (1.5*0,1.5*-1){$\{\sigma_{9}\}$};
      \node[right](a11) at (1.5*.5, 1.5*-0.9){$\{\sigma_{10}\}$};
      \node[right](a12) at (1.5*.9, 1.5*-0.5) {$\{\sigma_{11}\}$};
      \draw(a0)--(a1);
      \draw(a0)--(a2);
      \draw(a0)--(a3);
      \draw(a0)--(a4);
      \draw(a0)--(a5);
      \draw(a0)--(a6);
      \draw(a0)--(a7);
      \draw(a0)--(a8);
      \draw(a0)--(a9);
      \draw(a0)--(a10);
      \draw(a0)--(a11);
      \draw(a0)--(a12);
\node[right](b1) at (3.5*.71,3.5*.71){$\{\sigma_{1},\;\sigma_{2}\}$};
      \node[above](b2) at (3.5*0.3, 3.5*0.97){$\{\sigma_{2},\;\sigma_{3}\}$};
      \node[above](b3) at (3.5*-0.3, 3.5*0.97){$\{\sigma_{3},\;\sigma_{4}\}$};
      \node[left](b4) at (3.5*-0.7,3.5*0.7){$\{\sigma_{4},\;\sigma_{5}\}$};
      \node[left](b5) at (3.5*-0.97, 3.5*0.3){$\{\sigma_{5},\;\sigma_{6}\}$};
      \node[left](b6) at (3.5*-0.97, 3.5*-0.3){$\{\sigma_{6},\sigma_{7}\}$};
      \node[left](b7) at (3.5*-0.7, 3.5*-0.7){$\{\sigma_{7},\;\sigma_{8}\}$};
      \node[below](b8) at (3.5*-0.3, 3.5*-1.){$\{\sigma_{8},\;\sigma_{9}\}$};
      \node[below](b9) at (3.5*0.3, 3.5*-1.){$\{\sigma_{9},\;\sigma_{10}\}$};
      \node[right](b10) at (3.5*0.7, 3.5*-0.7){$\{\sigma_{10},\;\sigma_{11}\}$};
      \node[right](b11) at (3.5*1., 3.5*-0.3){$\{\sigma_{11},\;\sigma_0\}$};
      \node[right](b12) at (3.5*1., 3.5*0.3) {$\{\sigma_0,\;\sigma_{1}\}$};
      \draw(b12)--(a1);
      \draw(b12)--(a2);
      \draw(b1)--(a2);
      \draw(b1)--(a3);
      \draw(b2)--(a3);
            \draw(b2)--(a4);
      \draw(b3)--(a4);
            \draw(b3)--(a5);
      \draw(b4)--(a5);
            \draw(b4)--(a6);
      \draw(b5)--(a6);
            \draw(b5)--(a7);
      \draw(b6)--(a7);
            \draw(b6)--(a8);
      \draw(b7)--(a8);
            \draw(b7)--(a9);
      \draw(b8)--(a9);
            \draw(b8)--(a10);
      \draw(b9)--(a10);
            \draw(b9)--(a11);
      \draw(b10)--(a11);
            \draw(b10)--(a12);
      \draw(b11)--(a12);
            \draw(b11)--(a1);
            \node[right](c1) at (6*1,6*0){$\{\sigma_0,\;\sigma_{1},\;\sigma_{11}\}$};
      \node[right](c2) at (6*0.9, 6*0.5){$\{\sigma_0,\;\sigma_{1},\;\sigma_{2}\}$};
      \node[right](c3) at (6*0.5,6*0.9){$\{\sigma_{1},\;\sigma_{2},\;\sigma_{3}\}$};
      \node[above](c4) at (6*0,6*1){$\{\sigma_{2},\;\sigma_{3},\;\sigma_{4}\}$};
      \node[left](c5) at (6*-0.5, 6*0.9){$\{\sigma_{3},\;\sigma_{4},\;\sigma_{5}\}$};
      \node[left](c6) at (6*-0.9,6* 0.5){$\{\sigma_{4},\;\sigma_{5},\;\sigma_{6}\}$};
      \node[left](c7) at (6*-1,6*0){$\{\sigma_{5},\;\sigma_{6},\;\sigma_{7}\}$};
      \node[left](c8) at (6*-0.9, 6*-0.5){$\{\sigma_{6},\;\sigma_{7},\;\sigma_{8}\}$};
      \node[left](c9) at (6*-0.5, 6*-0.9){$\{\;\sigma_{7},\;\sigma_{8},\;\sigma_{9}\}$};
      \node[below](c10) at (6*0,6*-1){$\{\sigma_{8},\;\sigma_{9},\;\sigma_{10}\}$};
      \node[right](c11) at (6*.5, 6*-0.9){$\{\sigma_{9},\;\sigma_{10},\;\sigma_{11}\}$};
      \node[right](c12) at (6*.9, 6*-0.5) {$\{\sigma_{10},\;\sigma_{11},\;\sigma_0\}$};
	\draw(b12)--(c1);
      \draw(b12)--(c2);
      \draw(b1)--(c2);
      \draw(b1)--(c3);
      \draw(b2)--(c3);
            \draw(b2)--(c4);
      \draw(b3)--(c4);
            \draw(b3)--(c5);
      \draw(b4)--(c5);
            \draw(b4)--(c6);
      \draw(b5)--(c6);
            \draw(b5)--(c7);
      \draw(b6)--(c7);
            \draw(b6)--(c8);
      \draw(b7)--(c8);
            \draw(b7)--(c9);
      \draw(b8)--(c9);
            \draw(b8)--(c10);
      \draw(b9)--(c10);
            \draw(b9)--(c11);
      \draw(b10)--(c11);
            \draw(b10)--(c12);
      \draw(b11)--(c12);
            \draw(b11)--(c1);
                  \node[right](d1) at (9*.71,9*.71){$\{\sigma_0,\;\sigma_{1},\;\sigma_{2},\;\sigma_{3}\}$};
      \node[above](d2) at (9*0.3, 9*0.97){$\{\sigma_{1},\;\sigma_{2},\;\sigma_{3},\;\sigma_{4}\}$};
      \node[above](d3) at (9*-0.3, 9*0.97){$\{\sigma_{2},\;\sigma_{3},\;\sigma_{4},\;\sigma_{5}\}$};
      \node[left](d4) at (9*-0.7,9*0.7){$\{\sigma_{3},\;\sigma_{4},\;\sigma_{5},\;\sigma_{6}\}$};
      \node[left](d5) at (9*-0.97, 9*0.3){$\{\sigma_{4},\;\sigma_{5},\;\sigma_{6},\;\sigma_{7}\}$};
      \node[left](d6) at (9*-0.97, 9*-0.3){$\{\sigma_{5},\;\sigma_{6},\;\sigma_{7},\;\sigma_{8}\}$};
      \node[left](d7) at (9*-0.7, 9*-0.7){$\{\sigma_{6},\;\sigma_{7},\;\sigma_{8},\;\sigma_{9}\}$};
      \node[below](d8) at (9*-0.3, 9*-1.){$\{\sigma_{7},\;\sigma_{8},\;\sigma_{9},\;\sigma_{10}\}$};
      \node[below](d9) at (9*0.3, 9*-1.){$\{\sigma_{8},\;\sigma_{9},\;\sigma_{10},\;\sigma_{11}\}$};
      \node[right](d10) at (9*0.7, 9*-0.7){$\{\sigma_{9},\;\sigma_{10},\;\sigma_{11},\;\sigma_0\}$};
      \node[right](d11) at (9*1., 9*-0.3){$\{\sigma_{10},\;\sigma_{11},\;\sigma_0,\;\sigma_{1}\}$};
      \node[right](d12) at (9*1., 9*0.3) {$\{\sigma_{11},\;\sigma_0,\;\sigma_{1},\;\sigma_{2}\}$};
	\draw(d12)--(c1);
      \draw(d12)--(c2);
      \draw(d1)--(c2);
      \draw(d1)--(c3);
      \draw(d2)--(c3);
            \draw(d2)--(c4);
      \draw(d3)--(c4);
            \draw(d3)--(c5);
      \draw(d4)--(c5);
            \draw(d4)--(c6);
      \draw(d5)--(c6);
            \draw(d5)--(c7);
      \draw(d6)--(c7);
            \draw(d6)--(c8);
      \draw(d7)--(c8);
            \draw(d7)--(c9);
      \draw(d8)--(c9);
            \draw(d8)--(c10);
      \draw(d9)--(c10);
            \draw(d9)--(c11);
      \draw(d10)--(c11);
            \draw(d10)--(c12);
      \draw(d11)--(c12);
            \draw(d11)--(c1);

                  \node[right](e1) at (11*1,11*0){$\{\sigma_{10},\;\sigma_{11},\;\sigma_0,\;\sigma_{1},\;\sigma_{2}\}$};
      \node[right](e2) at (11*0.9, 11*0.5){$\{\sigma_{11},\;\sigma_0,\;\sigma_{1},\;\sigma_{2},\;\sigma_{3}\}$};
      \node[right](e3) at (11*0.5,11*0.9){$\{\sigma_0,\;\sigma_{1},\;\sigma_{2},\;\sigma_{3},\;\sigma_{4}\}$};
      \node[above](e4) at (11*0,11*1){$\{\sigma_{1},\;\sigma_{2},\;\sigma_{3},\;\sigma_{4},\;\sigma_{5}\}$};
      \node[left](e5) at (11*-0.5, 11*0.9){$\{\sigma_{2},\;\sigma_{3},\;\sigma_{4},\;\sigma_{5},\;\sigma_{6}\}$};
      \node[left](e6) at (11*-0.9,11* 0.5){$\{\sigma_{3},\;\sigma_{4},\;\sigma_{5},\;\sigma_{6},\;\sigma_{7}\}$};
      \node[left](e7) at (11*-1,11*0){$\{\sigma_{4},\;\sigma_{5},\;\sigma_{6},\;\sigma_{7},\;\sigma_{8}\}$};
      \node[left](e8) at (11*-0.9, 11*-0.5){$\{\sigma_{5},\;\sigma_{6},\;\sigma_{7},\;\sigma_{8},\;\sigma_{9}\}$};
      \node[left](e9) at (11*-0.5, 11*-0.9){$\{\sigma_{6},\;\sigma_{7},\;\sigma_{8},\;\sigma_{9},\;\sigma_{10}\}$};
      \node[below](e10) at (11*0,11*-1){$\{\sigma_{7},\;\sigma_{8},\;\sigma_{9},\;\sigma_{10},\;\sigma_{11}\}$};
      \node[right](e11) at (11*.5, 11*-0.9){$\{\sigma_{8},\;\sigma_{9},\;\sigma_{10},\;\sigma_{11},\;\sigma_0\}$};
      \node[right](e12) at (11*.9, 11*-0.5) {$\{\sigma_{9},\;\sigma_{10},\;\sigma_{11},\;\sigma_0,\;\sigma_{1}\}$};
      	\draw(d12)--(e1);
      \draw(d12)--(e2);
      \draw(d1)--(e2);
      \draw(d1)--(e3);
      \draw(d2)--(e3);
            \draw(d2)--(e4);
      \draw(d3)--(e4);
            \draw(d3)--(e5);
      \draw(d4)--(e5);
            \draw(d4)--(e6);
      \draw(d5)--(e6);
            \draw(d5)--(e7);
      \draw(d6)--(e7);
            \draw(d6)--(e8);
      \draw(d7)--(e8);
            \draw(d7)--(e9);
      \draw(d8)--(e9);
            \draw(d8)--(e10);
      \draw(d9)--(e10);
            \draw(d9)--(e11);
      \draw(d10)--(e11);
            \draw(d10)--(e12);
      \draw(d11)--(e12);
            \draw(d11)--(e1);
\node[right] at (16,8){{\bf{Key}}};
\node[right] at (16,7){$\sigma_{0}:=1$};
\node[right] at (16,6){$\sigma_{1}:=s_1$};
\node[right] at (16,5){$\sigma_{2}:=s_2s_1$};
\node[right] at (16,4){$\sigma_{3}:=s_1s_2s_1$};
\node[right] at (16,3){$\sigma_{4}:=(s_2s_1)^2$};
\node[right] at (16,2){$\sigma_{5}:=s_1(s_2s_1)^2$};
\node[right] at (16,1){$\sigma_{6}:=(s_2s_1)^3$};
\node[right] at (16,0){$\sigma_{7}:=s_1(s_2s_1)^3$};
\node[right] at (16,-1){$\sigma_{8}:=(s_2s_1)^4$};
\node[right] at (16,-2){$\sigma_{9}:=s_1(s_2s_1)^4$};
\node[right] at (16,-3){$\sigma_{10}:=(s_2s_1)^5$};
\node[right] at (16,-4){$\sigma_{11}:=s_1(s_2s_1)^5$};
\draw(15.5,-4.5)--(19,-4.5)--(19,8.5)--(15.5,8.5)--(15.5,-4.5);
   \end{tikzpicture}
    }
  \caption{Hasse diagram of the Weyl alternation sets $\A(\lambda,0)$ of Lie algebra of type $G_2$.}\label{fig:G2HasseDiagram}
\end{figure}

Studying these posets opens a new avenue for research in this field.

\addresseshere

\newpage

\appendix

\section{Tables of Weyl alternation sets}\label{alttables}
In this section we consider the Lie algebras of type $B_r$ (with $2\leq r\leq 6$) and provide tables for the value of $m_q(\lambda,\mu)$ for pairs of weights $\lambda=\ell\w_1$ (with $1\leq \ell\leq 10$) and $\mu$ such that $m(\lambda,\mu)=1$ along with the Weyl alternation sets. Note that for $\ell\in\ZZ_{>0}$ the set $M_\ell$ denotes the set of weights $\mu$ for which $m(\ell\w_1,\mu)=1$. These tables show the rapid growth of the Weyl alternation sets and illustrate the complication in the description of its elements.

\vspace{.25in}
\tablehead{\hline\rowcolor{midgray}\multicolumn{4}{|c|}{The $q$-multiplicity and Weyl alternation sets for pairs of weights in the Lie algebra of type $B_2$.}\\\hline $\lambda = \ell\varpi_1$ & $\mu\in M_\ell$ & $m_q(\lambda,\mu)$ & $\sigma\in\A(\lambda,\mu)$ \\\hline\hline}
\begin{supertabular}{|c|c|c|p{4in}|}
\hline
\rowcolor{lightgray}{$\ell=1$}& $0$ & $q^{2}$ & $1, s_2$\\
\hline
{$\ell=3$}& $2\w_2$ & $q^{3}$ & $1, s_2$\\\hline
{$\ell=3$}& $2\w_1$ & $q^{2}$ & $1, s_2$\\
\hline
\rowcolor{lightgray}{$\ell=5$}& $4\w_2$ & $q^{4}$ & $1, s_2$\\\hline
\rowcolor{lightgray}{$\ell=5$}& $2\w_1+2\w_2$ & $q^{3}$ & $1, s_2$\\\hline
\rowcolor{lightgray}{$\ell=5$}& $4\w_1$ & $q^{2}$ & $1, s_2$\\
\hline
{$\ell=7$}& $6\w_2$ & $q^{5}$ & $1, s_2$\\\hline
{$\ell=7$}& $2\w_1+4\w_2$ & $q^{4}$ & $1, s_2$\\\hline
{$\ell=7$}& $4\w_1+2\w_2$ & $q^{3}$ & $1, s_2$\\\hline
{$\ell=7$}& $6\w_1$ & $q^{2}$ & $1, s_2$\\
\hline
\rowcolor{lightgray}{$\ell=9$}& $8\w_2$ & $q^{6}$ & $1, s_2$\\\hline
\rowcolor{lightgray}{$\ell=9$}& $2\w_1+6\w_2$ & $q^{5}$ & $1, s_2$\\\hline
\rowcolor{lightgray}{$\ell=9$}& $4\w_1+4\w_2$ & $q^{4}$ & $1, s_2$\\\hline
\rowcolor{lightgray}{$\ell=9$}& $6\w_1+2\w_2$ & $q^{3}$ & $1, s_2$\\\hline
\rowcolor{lightgray}{$\ell=9$}& $8\w_1$ & $q^{2}$ & $1, s_2$\\
\hline
\end{supertabular}

\vspace{.25in}

\tablehead{\hline\rowcolor{midgray}\multicolumn{4}{|c|}{The $q$-multiplicity and Weyl alternation sets for pairs of weights in the Lie algebra of type $B_3$.}\\\hline $\lambda = \ell\varpi_1$ & $\mu\in M_\ell$ & $m_q(\lambda,\mu)$ & $\sigma\in\A(\lambda,\mu)$ \\\hline\hline}
\begin{supertabular}{|c|c|c|p{3.6in}|}
\hline
\rowcolor{lightgray}{$\ell=1$}& $0$ & $q^{3}$ & $1, s_2, s_3$\\
\hline
{$\ell=3$}& $2\w_1$ & $q^{3}$ & $1, s_2, s_3$\\
\hline
\rowcolor{lightgray}{$\ell=4$}& $2\w_3$ & $q^{6}$ & $1, s_2, s_3, s_2s_3$\\
\hline
{$\ell=5$}& $2\w_2$ & $q^{5}$ & $1, s_2, s_3$\\\hline
{$\ell=5$}& $4\w_1$ & $q^{3}$ & $1, s_2, s_3$\\
\hline
\rowcolor{lightgray}{$\ell=6$}& $2\w_1+2\w_3$ & $q^{6}$ & $1, s_2, s_3, s_2s_3$\\
\hline
{$\ell=7$}& $4\w_3$ & $q^{9}$ & $1, s_2, s_3, s_2s_3$\\\hline
{$\ell=7$}& $2\w_1+2\w_2$ & $q^{5}$ & $1, s_2, s_3$\\\hline
{$\ell=7$}& $6\w_1$ & $q^{3}$ & $1, s_2, s_3$\\
\hline
\rowcolor{lightgray}{$\ell=8$}& $2\w_2+2\w_3$ & $q^{8}$ & $1, s_2, s_3, s_2s_3$\\\hline
\rowcolor{lightgray}{$\ell=8$}& $4\w_1+2\w_3$ & $q^{6}$ & $1, s_2, s_3, s_2s_3$\\
\hline
{$\ell=9$}& $2\w_1+4\w_3$ & $q^{9}$ & $1, s_2, s_3, s_2s_3$\\\hline
{$\ell=9$}& $4\w_2$ & $q^{7}$ & $1, s_2, s_3$\\\hline
{$\ell=9$}& $4\w_1+2\w_2$ & $q^{5}$ & $1, s_2, s_3$\\\hline
{$\ell=9$}& $8\w_1$ & $q^{3}$ & $1, s_2, s_3$\\
\hline
\rowcolor{lightgray}{$\ell=10$}& $6\w_3$ & $q^{12}$ & $1, s_2, s_3, s_2s_3$\\\hline
\rowcolor{lightgray}{$\ell=10$}& $2\w_1+2\w_2+2\w_3$ & $q^{8}$ & $1, s_2, s_3, s_2s_3$\\\hline
\rowcolor{lightgray}{$\ell=10$}& $6\w_1+2\w_3$ & $q^{6}$ & $1, s_2, s_3, s_2s_3$\\
\hline
\end{supertabular}

\clearpage
{\footnotesize
\tablehead{\hline\rowcolor{midgray}\multicolumn{4}{|c|}{The $q$-multiplicity and Weyl alternation sets for pairs of weights in the Lie algebra of type $B_4$.}\\\hline $\lambda = \ell\varpi_1$ & $\mu\in M_\ell$ & $m_q(\lambda,\mu)$ & $\sigma\in\A(\lambda,\mu)$ \\\hline\hline}
\begin{supertabular}{|c|c|c|p{4in}|}
\hline
\rowcolor{lightgray}{$\ell=1$}& $0$ & $q^{4}$ & $1, s_2, s_3, s_4, s_4s_2$\\
\hline
{$\ell=3$}& $2\w_1$ & $q^{4}$ & $1, s_2, s_3, s_4, s_4s_2$\\
\hline
\rowcolor{lightgray}{$\ell=5$}& $2\w_4$ & $q^{10}$ & $1, s_2, s_3, s_4, s_2s_3, s_4s_2, s_3s_2, s_3s_4, s_2s_3s_2, s_2s_3s_4$\\\hline
{$\ell=5$}& $2\w_2$ & $q^{6}$ & $1, s_2, s_3, s_4, s_4s_2$\\\hline
\rowcolor{lightgray}{$\ell=5$}& $4\w_1$ & $q^{4}$ & $1, s_2, s_3, s_4, s_4s_2$\\
\hline
{$\ell=7$}& $2\w_1+2\w_4$ & $q^{10}$ & $1, s_2, s_3, s_4, s_2s_3, s_4s_2, s_3s_2, s_3s_4, s_2s_3s_2, s_2s_3s_4$\\\hline
{$\ell=7$}& $2\w_3$ & $q^{10}$ & $1, s_2, s_3, s_4, s_2s_3, s_4s_2$\\\hline
{$\ell=7$}& $2\w_1+2\w_2$ & $q^{6}$ & $1, s_2, s_3, s_4, s_4s_2$\\\hline
{$\ell=7$}& $6\w_1$ & $q^{4}$ & $1, s_2, s_3, s_4, s_4s_2$\\
\hline
\rowcolor{lightgray}{$\ell=9$}& $4\w_4$ & $q^{16}$ & 
$1, s_2, s_3, s_4, s_2s_3, s_4s_2, s_3s_2, s_3s_4, 
s_2s_3s_2, s_2s_3s_4, s_3s_4s_2, s_2s_3s_4s_2$\\\hline
\rowcolor{lightgray}{$\ell=9$}& $2\w_2+2\w_4$ & $q^{12}$ & $1, s_2, s_3, s_4, s_2s_3, s_4s_2, s_3s_2, s_3s_4, s_2s_3s_2, s_2s_3s_4$\\\hline
\rowcolor{lightgray}{$\ell=9$}& $4\w_1+2\w_4$ & $q^{10}$ & $1, s_2, s_3, s_4, s_2s_3, s_4s_2, s_3s_2, s_3s_4, s_2s_3s_2, s_2s_3s_4$\\\hline
\rowcolor{lightgray}{$\ell=9$}& $2\w_1+2\w_3$ & $q^{10}$ & $1, s_2, s_3, s_4, s_2s_3, s_4s_2$\\\hline
\rowcolor{lightgray}{$\ell=9$}& $4\w_2$ & $q^{8}$ & $1, s_2, s_3, s_4, s_4s_2$\\\hline
\rowcolor{lightgray}{$\ell=9$}& $4\w_1+2\w_2$ & $q^{6}$ & $1, s_2, s_3, s_4, s_4s_2$\\\hline
\rowcolor{lightgray}{$\ell=9$}& $8\w_1$ & $q^{4}$ & $1, s_2, s_3, s_4, s_4s_2$\\
\hline
\end{supertabular}

\vspace{.25in}

\tablehead{\hline\rowcolor{midgray}\multicolumn{4}{|c|}{The $q$-multiplicity and Weyl alternation sets for pairs of weights in the Lie algebra of type $B_5$.}\\\hline $\lambda = \ell\varpi_1$ & $\mu\in M_\ell$ & $m_q(\lambda,\mu)$ & $\sigma\in\A(\lambda,\mu)$ \\\hline\hline}
\begin{supertabular}{|c|c|c|p{4in}|}
\hline
\rowcolor{lightgray}{$\ell=1$}& $0$ & $q^{5}$ & $1, s_2, s_3, s_4, s_5, s_4s_2, s_5s_2, s_5s_3$\\
\hline
{$\ell=3$}& $2\w_1$ & $q^{5}$ & $1, s_2, s_3, s_4, s_5, s_4s_2, s_5s_2, s_5s_3$\\
\hline
\rowcolor{lightgray}{$\ell=5$}& $2\w_2$ & $q^{7}$ & $1, s_2, s_3, s_4, s_5, s_4s_2,
s_5s_2, s_5s_3$\\\hline
\rowcolor{lightgray}{$\ell=5$}& $4\w_1$ & $q^{5}$ & $1, s_2, s_3, s_4, s_5, s_4s_2, s_5s_2, s_5s_3$\\
\hline
{$\ell=6$}& $2\w_5$ & $q^{15}$ &
$1, s_2, s_3, s_4, s_5, s_2s_3, s_4s_2, s_5s_2, s_3s_2, s_3s_4, s_5s_3, s_4s_3, \newline
s_4s_5, s_2s_3s_2, s_2s_3s_4, s_5s_2s_3, s_4s_2s_3, s_4s_5s_2, s_3s_4s_2, s_5s_3s_2,\newline
s_3s_4s_3, s_3s_4s_5, s_2s_3s_4s_2, s_5s_2s_3s_2, s_2s_3s_4s_3, s_2s_3s_4s_5$\\
\hline
\rowcolor{lightgray}{$\ell=7$}& $2\w_3$ & $q^{11}$ & $1, s_2, s_3, s_4, s_5, s_2s_3, s_4s_2, s_5s_2, s_5s_3, s_5s_2s_3$\\\hline
\rowcolor{lightgray}{$\ell=7$}& $2\w_1+2\w_2$ & $q^{7}$ & $1, s_2, s_3, s_4, s_5, s_4s_2, s_5s_2, s_5s_3$\\\hline
\rowcolor{lightgray}{$\ell=7$}& $6\w_1$ & $q^{5}$ & $1, s_2, s_3, s_4, s_5, s_4s_2, s_5s_2, s_5s_3$\\
\hline
{$\ell=8$}& $2\w_1+2\w_5$ & $q^{15}$ & 
$1, s_2, s_3, s_4, s_5, s_2s_3, s_4s_2, s_5s_2, s_3s_2, s_3s_4, s_5s_3, s_4s_3, 
\newline 
s_4s_5, s_2s_3s_2, s_2s_3s_4, s_5s_2s_3, s_4s_2s_3, s_4s_5s_2, s_3s_4s_2, s_5s_3s_2,\newline
s_3s_4s_3, s_3s_4s_5, s_2s_3s_4s_2, s_5s_2s_3s_2, s_2s_3s_4s_3, s_2s_3s_4s_5$\\
\hline
\rowcolor{lightgray}{$\ell=9$}& $2\w_4$ & $q^{17}$ & 
$1, s_2, s_3, s_4, s_5, s_2s_3, s_4s_2, s_5s_2, s_3s_2, s_3s_4, s_5s_3, 
s_2s_3s_2, \newline s_2s_3s_4, s_5s_2s_3, s_3s_4s_2, s_5s_3s_2, s_2s_3s_4s_2, s_5s_2s_3s_2$\\\hline
\rowcolor{lightgray}{$\ell=9$}& $2\w_1+2\w_3$ & $q^{11}$ & $1, s_2, s_3, s_4, s_5, s_2s_3, s_4s_2, s_5s_2, s_5s_3, s_5s_2s_3$\\\hline
\rowcolor{lightgray}{$\ell=9$}& $4\w_2$ & $q^{9}$ & $1, s_2, s_3, s_4, s_5, s_4s_2, s_5s_2, s_5s_3$\\\hline
& $4\w_1+2\w_2$ & $q^{7}$ & $1, s_2, s_3, s_4, s_5, s_4s_2, s_5s_2, s_5s_3$\\\hline
{$\ell=9$}& $8\w_1$ & $q^{5}$ & $1, s_2, s_3, s_4, s_5, s_4s_2, s_5s_2, s_5s_3$\\
\hline
{$\ell=10$}& $2\w_2+2\w_5$ & $q^{17}$ & 
$1, s_2, s_3, s_4, s_5, s_2s_3, s_4s_2, s_5s_2, s_3s_2, s_3s_4, s_5s_3, s_4s_3, s_4s_5,\newline
s_2s_3s_2, s_2s_3s_4, s_5s_2s_3, s_4s_2s_3, s_4s_5s_2, s_3s_4s_2, s_5s_3s_2, s_3s_4s_3,\newline s_3s_4s_5, s_2s_3s_4s_2, s_5s_2s_3s_2, s_2s_3s_4s_3, s_2s_3s_4s_5$\\\hline
{$\ell=10$}& $4\w_1+2\w_5$ & $q^{15}$ & $1, s_2, s_3, s_4, s_5, s_2s_3, s_4s_2, s_5s_2, s_3s_2, s_3s_4, s_5s_3, s_4s_3, s_4s_5, \newline s_2s_3s_2, s_2s_3s_4, s_5s_2s_3, s_4s_2s_3, s_4s_5s_2, s_3s_4s_2, s_5s_3s_2, s_3s_4s_3,\newline s_3s_4s_5,  s_2s_3s_4s_2, s_5s_2s_3s_2, s_2s_3s_4s_3, s_2s_3s_4s_5$\\\hline
\end{supertabular}
}

\clearpage

{\footnotesize
\tablehead{\hline\rowcolor{midgray}\multicolumn{4}{|c|}{The $q$-multiplicity and Weyl alternation sets for pairs of weights in the Lie algebra of type $B_6$.}\\\hline $\lambda = \ell\varpi_1$ & $\mu\in M_\ell$ & $m_q(\lambda,\mu)$ & $\sigma\in\A(\lambda,\mu)$ \\\hline\hline}
\begin{supertabular}{|c|c|c|p{4in}|}
\rowcolor{lightgray}
\hline
{$\ell=1$}& $0$ & $q^{6}$ & $1, s_2, s_3, s_4, s_5, s_6, s_4s_2, s_5s_2, s_6s_2, s_5s_3, s_6s_3, s_6s_4, s_6s_4s_2$\\
\hline
{$\ell=3$}& $2\w_1$ & $q^{6}$ & $1, s_2, s_3, s_4, s_5, s_6, s_4s_2, s_5s_2, s_6s_2, s_5s_3, s_6s_3, s_6s_4, s_6s_4s_2$\\
\hline
\rowcolor{lightgray}{$\ell=5$}& $2\w_2$ & $q^{8}$ & $1, s_2, s_3, s_4, s_5, s_6, s_4s_2, s_5s_2, s_6s_2, s_5s_3, s_6s_3, s_6s_4, s_6s_4s_2$\\\hline
\rowcolor{lightgray}{$\ell=5$}& $4\w_1$ & $q^{6}$ & $1, s_2, s_3, s_4, s_5, s_6, s_4s_2, s_5s_2, s_6s_2, s_5s_3, s_6s_3, s_6s_4, s_6s_4s_2$\\
\hline
{$\ell=7$}& $2\w_6$ & $q^{21}$ & 
$1, s_2, s_3, s_4, s_5, s_6, s_2s_3, s_4s_2, s_5s_2, s_6s_2, s_3s_2, s_3s_4, s_5s_3, s_6s_3,\newline
s_4s_3, s_4s_5, s_6s_4, s_5s_4, s_5s_6, s_2s_3s_2, s_2s_3s_4, s_5s_2s_3, s_6s_2s_3, \newline
s_4s_2s_3, s_4s_5s_2, s_6s_4s_2, s_5s_4s_2, s_5s_6s_2, s_3s_4s_2, s_5s_3s_2, s_6s_3s_2,\newline
s_3s_4s_3, s_3s_4s_5, s_6s_3s_4, s_5s_3s_4, s_5s_6s_3, s_4s_3s_2, s_4s_5s_3, s_6s_4s_3,\newline 
s_4s_5s_4, s_4s_5s_6, s_2s_3s_4s_2, s_5s_2s_3s_2, s_6s_2s_3s_2, s_2s_3s_4s_3, s_2s_3s_4s_5,\newline
s_6s_2s_3s_4, s_5s_2s_3s_4, s_5s_6s_2s_3, s_4s_2s_3s_2, s_4s_5s_2s_3, s_6s_4s_2s_3,\newline
s_4s_5s_4s_2, s_4s_5s_6s_2, s_3s_4s_2s_3, s_3s_4s_5s_2, s_6s_3s_4s_2, s_5s_3s_4s_2,\newline
s_5s_6s_3s_2, s_3s_4s_3s_2, s_3s_4s_5s_3, s_6s_3s_4s_3, s_3s_4s_5s_4, s_3s_4s_5s_6,\newline
s_6s_4s_3s_2, s_2s_3s_4s_2s_3, s_2s_3s_4s_5s_2, s_6s_2s_3s_4s_2, s_5s_2s_3s_4s_2,\newline
s_5s_6s_2s_3s_2, s_2s_3s_4s_3s_2, s_2s_3s_4s_5s_3, s_6s_2s_3s_4s_3, s_2s_3s_4s_5s_4,\newline
s_2s_3s_4s_5s_6, s_6s_4s_2s_3s_2, s_3s_4s_2s_3s_2, s_6s_3s_4s_2s_3, s_3s_4s_5s_4s_2,\newline
s_6s_3s_4s_3s_2, s_2s_3s_4s_2s_3s_2, s_6s_2s_3s_4s_2s_3, s_2s_3s_4s_5s_4s_2,\newline s_6s_2s_3s_4s_3s_2, s_6s_3s_4s_2s_3s_2, s_6s_2s_3s_4s_2s_3s_2$\\[3pt]\hline
{$\ell=7$}& $2\w_3$ & $q^{12}$ & 
$1, s_2, s_3, s_4, s_5, s_6, s_2s_3, s_4s_2, s_5s_2, s_6s_2, s_5s_3, s_6s_3, s_6s_4,\newline s_5s_2s_3, s_6s_2s_3, s_6s_4s_2$\\\hline
{$\ell=7$}& $2\w_1+2\w_2$ & $q^{8}$ 
& $1, s_2, s_3, s_4, s_5, s_6, s_4s_2, s_5s_2, s_6s_2, s_5s_3, s_6s_3, s_6s_4, s_6s_4s_2$\\\hline
{$\ell=7$}& $6\w_1$ & $q^{6}$ & $1, s_2, s_3, s_4, s_5, s_6, s_4s_2, s_5s_2, s_6s_2, s_5s_3, s_6s_3, s_6s_4, s_6s_4s_2$\\
\hline
\rowcolor{lightgray}{$\ell=9$}& $2\w_1+2\w_6$ & $q^{21}$ & 
$1, s_2, s_3, s_4, s_5, s_6, s_2s_3, s_4s_2, s_5s_2, s_6s_2, s_3s_2, s_3s_4, s_5s_3, s_6s_3,\newline
s_4s_3, s_4s_5, s_6s_4, s_5s_4, s_5s_6, s_2s_3s_2, s_2s_3s_4, s_5s_2s_3, s_6s_2s_3, \newline
s_4s_2s_3, s_4s_5s_2, s_6s_4s_2, s_5s_4s_2, s_5s_6s_2, s_3s_4s_2, s_5s_3s_2, s_6s_3s_2,\newline
s_3s_4s_3, s_3s_4s_5, s_6s_3s_4, s_5s_3s_4, s_5s_6s_3, s_4s_3s_2, s_4s_5s_3, s_6s_4s_3,\newline
s_4s_5s_4, s_4s_5s_6, s_2s_3s_4s_2, s_5s_2s_3s_2, s_6s_2s_3s_2, s_2s_3s_4s_3, s_2s_3s_4s_5,\newline
s_6s_2s_3s_4, s_5s_2s_3s_4, s_5s_6s_2s_3, s_4s_2s_3s_2, s_4s_5s_2s_3, s_6s_4s_2s_3,\newline
s_4s_5s_4s_2, s_4s_5s_6s_2, s_3s_4s_2s_3, s_3s_4s_5s_2, s_6s_3s_4s_2, s_5s_3s_4s_2,\newline
s_5s_6s_3s_2, s_3s_4s_3s_2, s_3s_4s_5s_3, s_6s_3s_4s_3, s_3s_4s_5s_4, s_3s_4s_5s_6,\newline
s_6s_4s_3s_2, s_2s_3s_4s_2s_3, s_2s_3s_4s_5s_2, s_6s_2s_3s_4s_2, s_5s_2s_3s_4s_2,\newline
s_5s_6s_2s_3s_2, s_2s_3s_4s_3s_2, s_2s_3s_4s_5s_3, s_6s_2s_3s_4s_3, s_2s_3s_4s_5s_4,\newline
s_2s_3s_4s_5s_6, s_6s_4s_2s_3s_2, s_3s_4s_2s_3s_2, s_6s_3s_4s_2s_3, s_3s_4s_5s_4s_2,\newline
s_6s_3s_4s_3s_2, s_2s_3s_4s_2s_3s_2, s_6s_2s_3s_4s_2s_3, s_2s_3s_4s_5s_4s_2,\newline
s_6s_2s_3s_4s_3s_2, s_6s_3s_4s_2s_3s_2, s_6s_2s_3s_4s_2s_3s_2$\\\hline
\rowcolor{lightgray}{$\ell=9$}& $2\w_4$ & $q^{18}$ & 
$1, s_2, s_3, s_4, s_5, s_6, s_2s_3, s_4s_2, s_5s_2, s_6s_2, s_3s_2, s_3s_4, s_5s_3, \newline
s_6s_3, s_6s_4, s_2s_3s_2, s_2s_3s_4, s_5s_2s_3, s_6s_2s_3, s_6s_4s_2, s_3s_4s_2, \newline
s_5s_3s_2, s_6s_3s_2, s_6s_3s_4, s_2s_3s_4s_2, s_5s_2s_3s_2, s_6s_2s_3s_2, \newline s_6s_2s_3s_4,
s_6s_3s_4s_2, s_6s_2s_3s_4s_2$\\\hline
\rowcolor{lightgray}{$\ell=9$}& $2\w_1+2\w_3$ & $q^{12}$ & 
$1, s_2, s_3, s_4, s_5, s_6, s_2s_3, s_4s_2, s_5s_2, s_6s_2, s_5s_3, s_6s_3, s_6s_4,\newline s_5s_2s_3, s_6s_2s_3, s_6s_4s_2$\\\hline
\rowcolor{lightgray}{$\ell=9$}& $4\w_2$ & $q^{10}$ & $1, s_2, s_3, s_4, s_5, s_6, s_4s_2, s_5s_2, s_6s_2, s_5s_3, s_6s_3, s_6s_4, s_6s_4s_2$\\\hline
\rowcolor{lightgray}{$\ell=9$}& $4\w_1+2\w_2$ & $q^{8}$ & $1, s_2, s_3, s_4, s_5, s_6, s_4s_2, s_5s_2, s_6s_2, s_5s_3, s_6s_3, s_6s_4, s_6s_4s_2$\\\hline
\rowcolor{lightgray}{$\ell=9$}& $8\w_1$ & $q^{6}$ & $1, s_2, s_3, s_4, s_5, s_6, s_4s_2, s_5s_2, s_6s_2, s_5s_3, s_6s_3, s_6s_4, s_6s_4s_2$\\
\hline
\end{supertabular}
}

\clearpage

\section{Tables in support of Conjecture \ref{con:powerofq}}\label{app:tables}

In this section we consider the Lie algebra of type $A_r$ ($1\leq r\leq 10$) and provide tables for the value of $m_q(\lambda,\mu)$ for pairs of weights $\lambda=\ell\w_1$ (with $1\leq \ell\leq 10$) and $\mu$ such that $m(\lambda,\mu)=1$. Note that for $\ell\in\ZZ_{>0}$ the set $M_\ell$ denotes the set of weights $\mu$ for which $m(\ell\w_1,\mu)=1$. In all of these cases note that $m_q(\lambda,\mu)$ is a power of $q$.\\
\vspace{.2in}

\twocolumn
\footnotesize
\tablehead{\hline\rowcolor{midgray}\multicolumn{4}{|c|}{Type $A_1$ table}\\\hline $\lambda = \ell\varpi_1$ & $\mu\in M_\ell$ & $m_q(\lambda,\mu)$ & $|\mathcal{A}(\lambda, \mu)|$ \\\hline\hline}
\begin{supertabular}{|c|p{1in}|c|c|}
\rowcolor{lightgray}
$\w_1$& $\varpi_1$ & $1$ & $1$ \\\hline
$2\w_1$& $0$ & $q$ & $1$ \\\hline 
$2\w_1$& $2\varpi_1$ & $1$ & $1$ \\\hline
\rowcolor{lightgray}$3\w_1$& $\varpi_1$ & $q$ & $1$ \\\hline 
\rowcolor{lightgray}$3\w_1$& $3\varpi_1$ & $1$ & $1$\\\hline
$4\w_1$& $0$ & $q^{2}$ & $1$ \\\hline 
$4\w_1$& $2\varpi_1$ & $q$ & $1$ \\\hline 
$4\w_1$& $4\varpi_1$ & $1$ & $1$ \\\hline
\rowcolor{lightgray}$5\w_1$& $\varpi_1$ & $q^{2}$ & $1$ \\\hline 
\rowcolor{lightgray}$5\w_1$& $3\varpi_1$ & $q$ & $1$ \\\hline 
\rowcolor{lightgray}$5\w_1$& $5\varpi_1$ & $1$ & $1$ \\\hline
$6\w_1$& $0$ & $q^{3}$ & $1$ \\\hline 
$6\w_1$& $2\varpi_1$ & $q^{2}$ & $1$ \\\hline 
$6\w_1$& $4\varpi_1$ & $q$ & $1$ \\\hline 
$6\w_1$& $6\varpi_1$ & $1$ & $1$ \\\hline
\rowcolor{lightgray}$7\w_1$& $\varpi_1$ & $q^{3}$ & $1$ \\\hline 
\rowcolor{lightgray}$7\w_1$& $3\varpi_1$ & $q^{2}$ & $1$ \\\hline 
\rowcolor{lightgray}$7\w_1$& $5\varpi_1$ & $q$ & $1$ \\\hline 
\rowcolor{lightgray}$7\w_1$& $7\varpi_1$ & $1$ & $1$ \\\hline
$8\w_1$& $0$ & $q^{4}$ & $1$ \\\hline 
$8\w_1$& $2\varpi_1$ & $q^{3}$ & $1$ \\\hline 
$8\w_1$& $4\varpi_1$ & $q^{2}$ & $1$ \\\hline 
$8\w_1$& $6\varpi_1$ & $q$ & $1$ \\\hline 
$8\w_1$& $8\varpi_1$ & $1$ & $1$ \\\hline
\rowcolor{lightgray}$9\w_1$& $\varpi_1$ & $q^{4}$ & $1$ \\\hline 
\rowcolor{lightgray}$9\w_1$& $3\varpi_1$ & $q^{3}$ & $1$ \\\hline 
\rowcolor{lightgray}$9\w_1$& $5\varpi_1$ & $q^{2}$ & $1$ \\\hline 
\rowcolor{lightgray}$9\w_1$& $7\varpi_1$ & $q$ & $1$ \\\hline 
\rowcolor{lightgray}$9\w_1$& $9\varpi_1$ & $1$ & $1$ \\\hline
$10\w_1$& $0$ & $q^{5}$ & $1$ \\\hline 
$10\w_1$& $2\varpi_1$ & $q^{4}$ & $1$ \\\hline 
$10\w_1$& $4\varpi_1$ & $q^{3}$ & $1$ \\\hline 
$10\w_1$& $6\varpi_1$ & $q^{2}$ & $1$ \\\hline 
$10\w_1$& $8\varpi_1$ & $q$ & $1$\\\hline 
$10\w_1$& $10\varpi_1$ & $1$ & $1$ \\\hline
\end{supertabular}

\vspace{.15in}
 
\tablehead{\hline\rowcolor{midgray}\multicolumn{4}{|c|}{Type $A_2$ table}\\\hline $\lambda = \ell\varpi_1$ & $\mu\in M_\ell$ & $m_q(\lambda,\mu)$ & $|\mathcal{A}(\lambda, \mu)|$\\\hline\hline}
\begin{supertabular}{|c|p{1in}|c|c|}
\rowcolor{lightgray}$w_1$& $\varpi_1$ & $1$ & $1$ \\\hline
$2\w_1$& $\varpi_2$ & $q$ & $1$ \\\hline
$2\w_1$& $2\varpi_1$ & $1$ & $1$ \\\hline
\rowcolor{lightgray}$3\w_1$& $0$ & $q^{3}$ & $2$ \\\hline 
\rowcolor{lightgray}$3\w_1$& $\varpi_1+\varpi_2$ & $q$ & $1$ \\\hline 
\rowcolor{lightgray}$3\w_1$& $3\varpi_1$ & $1$ & $1$ \\\hline
$4\w_1$& $\varpi_1$ & $q^{3}$ & $2$ \\\hline 
$4\w_1$& $2\varpi_2$ & $q^{2}$ & $1$ \\
\hline
\end{supertabular}

\vspace{.15in}

\tablehead{\hline\rowcolor{midgray}\multicolumn{4}{|c|}{Type $A_3$ table}\\\hline $\lambda = \ell\varpi_1$ & $\mu\in M_\ell$ & $m_q(\lambda,\mu)$ & $|\mathcal{A}(\lambda, \mu)|$ \\\hline\hline}
\begin{supertabular}{|c|p{1in}|c|c|}
\rowcolor{lightgray}$\w_1$& $\varpi_1$ & $1$ & $1$ \\\hline
$2\w_1$& $\varpi_2$ & $q$ & $1$ \\\hline $2\w_1$
& $2\varpi_1$ & $1$ & $1$ \\\hline
\rowcolor{lightgray}$3\w_1$& $\varpi_3$ & $q^{3}$ & $2$ \\\hline 
\rowcolor{lightgray}$3\w_1$& $\varpi_1+\varpi_2$ & $q$ & $1$ \\\hline 
\rowcolor{lightgray}$3\w_1$& $3\varpi_1$ & $1$ & $1$ \\\hline
$4\w_1$& $0$ & $q^{6}$ & $4$ \\\hline 
$4\w_1$& $\varpi_1+\varpi_3$ & $q^{3}$ & $2$ \\
\hline
\end{supertabular}

 \vspace{.15in}

\tablehead{\hline\rowcolor{midgray}\multicolumn{4}{|c|}{Type $A_4$ table}\\\hline $\lambda = \ell\varpi_1$ & $\mu\in M_\ell$ & $m_q(\lambda,\mu)$ & $|\mathcal{A}(\lambda,\mu)| $\\\hline\hline}
\begin{supertabular}{|c|p{1in}|c|c|}
\rowcolor{lightgray}$\w_1$& $\varpi_1$ & $1$ & $1$ \\\hline
$2\w_1$& $\varpi_2$ & $q$ & $1$ \\\hline $2\w_1$
& $2\varpi_1$ & $1$ & $1$ \\\hline
\rowcolor{lightgray}$3\w_1$& $\varpi_3$ & $q^{3}$ & $2$ \\\hline 
\rowcolor{lightgray}$3\w_1$& $\varpi_1+\varpi_2$ & $q$ & $1$ \\\hline 
\rowcolor{lightgray}$3\w_1$& $3\varpi_1$ & $1$ & $1$ \\\hline
$4\w_1$& $\varpi_4$ & $q^{6}$ & $4$ \\\hline $4\w_1$
& $\varpi_1+\varpi_3$ & $q^{3}$ & $2$ \\
\hline
\end{supertabular}

\vspace{.15in}

\tablehead{\hline\rowcolor{midgray}\multicolumn{4}{|c|}{Type $A_5$ table}\\\hline $\lambda = \ell\varpi_1$ & $\mu\in M_\ell$ & $m_q(\lambda,\mu)$ & $|\mathcal{A}(\lambda,\mu)|$ \\\hline\hline}
\begin{supertabular}{|c|p{1in}|c|c|}
\rowcolor{lightgray}$\w_1$& $\varpi_1$ & $1$ & $1$ \\\hline
$2\w_1$& $\varpi_2$ & $q$ & $1$ \\\hline
$2\w_1$& $2\varpi_1$ & $1$ & $1$ \\\hline
\rowcolor{lightgray}$3\w_1$& $\varpi_3$ & $q^{3}$ & $2$ \\\hline
\rowcolor{lightgray}$3\w_1$& $\varpi_1+\varpi_2$ & $q$ & $1$ \\\hline
\rowcolor{lightgray}$3\w_1$& $3\varpi_1$ & $1$ & $1$ \\\hline
$4\w_1$& $\varpi_4$ & $q^{6}$ & $4$ \\\hline
$4\w_1$& $\varpi_1+\varpi_3$ & $q^{3}$ & $2$ \\
\hline
\end{supertabular}

\vspace{.15in}

\tablehead{\hline\rowcolor{midgray}\multicolumn{4}{|c|}{Type $A_6$ table}\\\hline $\lambda = \ell\varpi_1$ & $\mu\in M_\ell$ & $m_q(\lambda,\mu)$ & $|\mathcal{A}(\lambda,\mu)|$ \\\hline\hline}
\begin{supertabular}{|c|p{1in}|c|c|}
\rowcolor{lightgray}{$\ell=1$}& $\varpi_1$ & $1$ & $1$ \\\hline
{$\ell=2$}& $\varpi_2$ & $q$& $1$ \\\hline
{$\ell=2$}& $2\varpi_1$ & $1$ & $1$ \\\hline
\rowcolor{lightgray}{$\ell=3$}& $\varpi_3$ & $q^{3}$& $2$ \\\hline
\rowcolor{lightgray}{$\ell=3$}& $\varpi_1+\varpi_2$ & $q$& $1$ \\\hline
\rowcolor{lightgray}{$\ell=3$}& $3\varpi_1$ & $1$ & $1$ \\\hline
{$\ell=4$}& $\varpi_4$ & $q^{6}$& $4$ \\\hline
{$\ell=4$}& $\varpi_1+\varpi_3$ & $q^{3}$& $2$ \\
\hline
\end{supertabular}

\vspace{.15in}

\tablehead{\hline\rowcolor{midgray}\multicolumn{4}{|c|}{Type $A_7$ table}\\\hline $\lambda = \ell\varpi_1$ & $\mu\in M_\ell$ & $m_q(\lambda,\mu)$ & $|\mathcal{A}(\lambda,\mu)|$\\\hline\hline}
\begin{supertabular}{|c|p{1in}|c|c|}
\rowcolor{lightgray}{$\ell=1$}& $\varpi_1$ & $1$ & $1$ \\\hline
{$\ell=2$}& $\varpi_2$ & $q$& $1$ \\\hline
{$\ell=2$}& $2\varpi_1$ & $1$ & $1$ \\\hline
\rowcolor{lightgray}{$\ell=3$}& $\varpi_3$ & $q^{3}$& $2$ \\\hline
\rowcolor{lightgray}{$\ell=3$}& $\varpi_1+\varpi_2$ & $q$& $1$ \\\hline
\rowcolor{lightgray}{$\ell=3$}& $3\varpi_1$ & $1$ & $1$ \\\hline
{$\ell=4$}& $\varpi_4$ & $q^{6}$& $4$ \\\hline
{$\ell=4$}& $\varpi_1+\varpi_3$ & $q^{3}$& $2$ \\
\hline
\end{supertabular}

\vspace{.15in}

\tablehead{\hline\rowcolor{midgray}\multicolumn{4}{|c|}{Type $A_8$ table}\\\hline $\lambda = \ell\varpi_1$ & $\mu\in M_\ell$ & $m_q(\lambda,\mu)$ & $|\mathcal{A}(\lambda,\mu)|$  \\\hline\hline}
\begin{supertabular}{|c|p{1in}|c|c|}
\rowcolor{lightgray}{$\ell=1$}& $\varpi_1$ & $1$ & $1$ \\\hline
{$\ell=2$}& $\varpi_2$ & $q$ & $1$ \\\hline
{$\ell=2$}& $2\varpi_1$ & $1$ & $1$ \\\hline
\rowcolor{lightgray}{$\ell=3$}& $\varpi_3$ & $q^{3}$ & $2$ \\\hline
\rowcolor{lightgray}{$\ell=3$}& $\varpi_1+\varpi_2$ & $q$ & $1$ \\\hline
\rowcolor{lightgray}{$\ell=3$}& $3\varpi_1$ & $1$ & $1$ \\\hline
{$\ell=4$}& $\varpi_4$ & $q^{6}$ & $4$ \\\hline
{$\ell=4$}& $\varpi_1+\varpi_3$ & $q^{3}$ & $2$ \\
\hline
\end{supertabular}

\vspace{.15in}

\tablehead{\hline\rowcolor{midgray}\multicolumn{4}{|c|}{Type $A_9$ table}\\\hline $\lambda = \ell\varpi_1$ & $\mu\in M_\ell$ & $m_q(\lambda,\mu)$ & $|\mathcal{A}(\lambda,\mu)|$\\\hline\hline}
\begin{supertabular}{|c|p{1in}|c|c|}
\rowcolor{lightgray}{$\ell=1$}& $\varpi_1$ & $1$ & $1$ \\\hline
{$\ell=2$}& $\varpi_2$ & $q$& $1$ \\\hline
{$\ell=2$}& $2\varpi_1$ & $1$ & $1$ \\\hline
\rowcolor{lightgray}{$\ell=3$}& $\varpi_3$ & $q^{3}$& $2$ \\\hline
\rowcolor{lightgray}{$\ell=3$}& $\varpi_1+\varpi_2$ & $q$& $1$ \\\hline
\rowcolor{lightgray}{$\ell=3$}& $3\varpi_1$ & $1$ & $1$ \\\hline
{$\ell=4$}& $\varpi_4$ & $q^{6}$& $4$ \\\hline
{$\ell=4$}& $\varpi_1+\varpi_3$ & $q^{3}$& $2$ \\
\hline

\end{supertabular}

\vspace{.15in}

\tablehead{\hline\rowcolor{midgray}\multicolumn{4}{|c|}{Type $A_{10}$ table}\\\hline $\lambda = \ell\varpi_1$ & $\mu\in M_\ell$ & $m_q(\lambda,\mu)$ & $|\mathcal{A}(\lambda,\mu)|$ \\\hline\hline}
\begin{supertabular}{|c|p{1in}|c|c|}
\rowcolor{lightgray}{$\ell=1$}& $\varpi_1$ & $1$ & $1$ \\\hline
{$\ell=2$}& $\varpi_2$ & $q$& $1$ \\\hline
{$\ell=2$}& $2\varpi_1$ & $1$ & $1$ \\\hline
\rowcolor{lightgray}{$\ell=3$}& $\varpi_3$ & $q^{3}$& $2$ \\\hline
\rowcolor{lightgray}{$\ell=3$}& $\varpi_1+\varpi_2$ & $q$& $1$ \\\hline
\rowcolor{lightgray}{$\ell=3$}& $3\varpi_1$ & $1$ & $1$ \\\hline
{$\ell=4$}& $\varpi_4$ & $q^{6}$& $4$ \\\hline
{$\ell=4$}& $\varpi_1+\varpi_3$ & $q^{3}$& $2$ \\
\hline
\end{supertabular}

\end{document}